\documentclass[review,onefignum,onetabnum]{siamart220329}

\makeatletter
\AtBeginDocument{%
  \def\refstepcounter@optarg[#1]#2{%
    \cref@old@refstepcounter{#2}%
    \cref@constructprefix{#2}{\cref@result}%
    \@ifundefined{cref@#1@alias}%
      {\def\@tempa{#1}}%
      {\def\@tempa{\csname cref@#1@alias\endcsname}}%
    \protected@edef\cref@currentlabel{%
      [\@tempa][\arabic{#2}][\cref@result]%
      \csname p@#2\endcsname\csname the#2\endcsname}}%
}
\makeatother

\usepackage{amsfonts}
\usepackage{graphicx}
\usepackage{epstopdf}
\usepackage{algorithmic}
\usepackage{amssymb} 
\ifpdf
  \DeclareGraphicsExtensions{.eps,.pdf,.png,.jpg}
\else
  \DeclareGraphicsExtensions{.eps}
\fi

\usepackage{amsmath}
\usepackage[shortlabels]{enumitem}
\usepackage[symbol]{footmisc}
\usepackage{tikz}
\usetikzlibrary{arrows.meta,calc,decorations.pathmorphing,decorations.pathreplacing,positioning}
\newcommand{\ml}{\left(}
\newcommand{\mr}{\right)}

\newcommand{\intr}{\text{int}}
\newcommand{\p}{\partial}
\newcommand{\D}{\mathrm{d}}

\DeclareMathOperator{\diag}{diag}
\DeclareMathOperator{\spn}{span}

\nolinenumbers

\newsiamremark{remark}{Remark}
\newsiamremark{hypothesis}{Hypothesis}
\newsiamthm{claim}{Claim}
\newsiamthm{assumption}{Assumption}

\headers{Generic Structural Stability for $n \times n$ Conservation Laws}{H. K. Tan and A. L. Bertozzi}

\title{Generic Structural Stability for Riemann Solutions to $n \times n$ Systems of Hyperbolic Conservation Laws\thanks{Submitted to the editors September 23, 2026.
\funding{This work was supported by the Simons Foundation Math + X award number 510776, ONR grant N00014-23-1-2565, and NSF grant DMS-2407006, and the DARPA ExpMath program.} }}

\author{Hong Kiat Tan\thanks{Department of Mathematics, UCLA, Los Angeles, CA 90095, US (\email{maxtanhk@math.ucla.edu}, \email{bertozzi@g.ucla.edu})}
\and
Andrea L. Bertozzi\footnotemark[2]
}

\usepackage{amsopn}

\begin{document}

\maketitle

\begin{abstract}
This paper proves generic structural stability for Riemann solutions to $n \times n$ systems of hyperbolic conservation laws in one spatial dimension. Under assumptions of strict hyperbolicity, genuine nonlinearity, and a regular manifold hypothesis on the Rankine--Hugoniot map, we show that for almost every pair of left and right states, any $n$-wave Riemann solution consisting of Lax-admissible shocks and rarefactions is structurally stable under perturbations of the left state, the right state, and the flux function in the $C^2$ topology. The central new idea is \emph{sequential transversality}, which chains the $n$ waves through intermediate states and transports their tangent contributions to a common reference point via pushforward maps, reducing the structural stability condition to the invertibility of an $n \times n$ transversality matrix.
We apply the results to the $p$-system, polydisperse particle-laden thin films, and machine-learned flux approximations.\end{abstract}

\begin{keywords}
$n \times n$ systems of conservation laws, Riemann problem, structural stability, transversality, sequential transversality, genericity, implicit function theorem on Banach spaces, Thom's parametric transversality theorem, p-system, particle-laden flows, machine-learned flux approximations.
\end{keywords}

\begin{AMS}
35L65, 35L67, 37C20, 46N20, 57R45, 74K35
\end{AMS}


\section{Introduction}
\label{sec:intro}

The Riemann problem for systems of conservation laws is a cornerstone of the theory of nonlinear hyperbolic PDEs, originating with Lax's seminal work~\cite{Lax} on the existence and uniqueness of Riemann solutions in the small-amplitude regime via the implicit function theorem.
For large-amplitude (``strong'') waves, however, the classical implicit function theorem argument is insufficient, and the analysis of how Riemann solutions respond to perturbations requires tools from differential topology.
This paper establishes the generic structural stability of Riemann solutions for general $n\times n$ systems of hyperbolic conservation laws in one spatial dimension.

Consider the system of $n \times n$ conservation laws
\begin{equation}\label{eq:conservation-law}
  u_t + F(u)_x = 0
\end{equation}
with \(u(x,t)\) defined on \((x,t)\in\mathbb{R}\times[0,\infty)\). Let \(U\subset\mathbb{R}^n\) be an open, connected set and let \(F\in C^2(U;\mathbb{R}^n)\). We write \(A(u):=DF(u)\) for the Jacobian of \(F\). The system is supplemented with Riemann initial data
\begin{equation}\label{eq:Riemann-data}
u(x,0) = \begin{cases}
u_l &\text{ for } x < 0, \\
u_r &\text{ for } x > 0,
\end{cases}
\end{equation}
for \(u_l, u_r \in U\).

Throughout we impose three standing assumptions on \(F\). The first two are stated here; the third, the \emph{regular manifold hypothesis} (Assumption~\ref{ass:regular}), requires the Rankine--Hugoniot map and is stated in Section~\ref{ssec:hugoniot} once the necessary machinery has been developed.

\begin{assumption}[Strict hyperbolicity]\label{ass:SH}
For every $u\in U$, the eigenvalues of $A(u)$ are real and distinct, with $\lambda_1(u)<\cdots<\lambda_n(u)$.
\end{assumption}

\begin{assumption}[Genuine nonlinearity]\label{ass:GN}
For each characteristic family \(k\in\{1,\dots,n\}\),
\(
  \nabla \lambda_k(u)\cdot r_k(u)\neq 0 \qquad \text{for all } u\in U,
\)
where \(r_k(u)\) is a right eigenvector of \(A(u)=DF(u)\) associated with \(\lambda_k(u)\).
\end{assumption}

Under these three assumptions, we show that for almost every pair of left and right states $(u_l,u_r)\in U^2$, any $n$-wave Riemann solution consisting of Lax-admissible shocks and rarefactions is structurally stable with respect to perturbations of the left state, the right state, and the flux function $F$ in the $C^2$ topology.

\subsection{Prior work and motivation}

In our prior work~\cite{TanBertozzi2x2}, we proved generic structural stability for the \(2\times 2\) case.
That paper introduced the key conceptual framework of connecting structural stability to transversality of wave curves in the state space, and then establishing transversality as a generic property via a foliated parametric transversality theorem.
This work generalizes these ideas to \(n\times n\) systems and requires fundamentally new ideas rather than a routine extension of previous arguments.
The concluding section of that paper suggested that a coordinate-free approach to the $n\times n$ problem might circumvent the coordinate-dependent graph condition, and the eigenchart construction of Section~\ref{ssec:rarefaction} does so.

Other important prior work on structural stability for systems of conservation laws includes Schecter, Marchesin, and Plohr~\cite{Schecter}, who proved structural stability for \(2\times 2\) systems using traveling wave connections from a viscous approximation.
Their approach imposes transversality assumptions that cannot be verified without knowing the intermediate states in advance.
By contrast, this work shows that the analogous transversality condition holds \emph{generically}, that is, for almost every choice of left and right states, without prior knowledge of the solution.
Additionally, Majda~\cite{Majda1983Stability,Majda1984Compressible} and Majda--Thomann~\cite{MajdaThomann1987CPDE} studied structural stability for multidimensional shock fronts, remarking that the multidimensional problem has \emph{stronger} structural stability than the one-dimensional case.
Their notion of structural stability, however, concerns smooth perturbations to the initial data rather than perturbations to the flux, and does not address the \(n\)-wave Riemann problem studied here.

A different line of work, begun by Isaacson, Marchesin, Palmeira, and Plohr~\cite{Global_formalism} and continued in~\cite{Azevedo,Topo_Riemann_2024}, resolves Riemann problems for pairs of conservation laws inside the wave manifold, a three-dimensional manifold whose points represent shock and rarefaction waves and on which wave curves that overlap in state space are disentangled, worked out in detail for quadratic flux functions in~\cite{Topo_Riemann_2024}.
There a Riemann problem is solved by intersecting the fast wave curve with a surface generated from the slow wave curve by shock curves, transversality of that intersection gives stability of the solution, and the construction aims at structural stability for all Riemann data at once rather than near a given solution.
This work treats $n\times n$ systems with general $C^2$ flux, stays in state space, and instead proves that the local transversality condition holds for almost every pair of states.

Writing a multiwave Riemann solution as the zero of a stacked system of wave relations, with invertibility of the derivative in the intermediate states as the nondegeneracy hypothesis, appears in Lin and Schecter~\cite{LinSchecter2003} and Schecter and Szmolyan~\cite{SchecterSzmolyan2009}, who study self-similar solutions of the Dafermos regularization and impose that invertibility on a given solution alongside a viscous-profile condition. This work uses no regularization, reads the same nondegeneracy as a chain of transversality conditions, and proves that it holds generically.

Results related to genericity for conservation laws also exist in the literature.
Schecter, Plohr, and Marchesin~\cite{Schecter2} extended their earlier structural stability analysis by studying Riemann solutions of different codimension, showing that structurally stable codimension-zero solutions have a boundary consisting of codimension-one solutions that lack structural stability.
The Schaeffer Regularity Theorem~\cite{Schaeffer} establishes that for almost every initial data in the Schwartz space, scalar conservation law solutions are piecewise smooth with finitely many shock curves, though this result does not extend to systems~\cite{Caravenna} and excludes Riemann initial data.
Bressan, Chen, and Huang~\cite{bressan_pressureless} have also analyzed singularity formation for pressureless gases with generic smooth initial data.

\subsection{New contributions}

The passage from \(2\times 2\) to \(n\times n\) systems introduces a fundamental geometric challenge.
In the \(2\times 2\) case, structural stability requires the transverse intersection of two one-dimensional wave curves in \(\mathbb{R}^2\).
For an \(n\times n\) system, however, the wave curves are still one-dimensional, but the state space is \(\mathbb{R}^n\).
Two such curves in \(\mathbb{R}^n\) have codimensions summing to \(2(n-1)\), which exceeds \(n\) for \(n\ge 3\), so that generically two wave curves do not intersect at all.
The pairwise intersection framework of the \(2\times 2\) theory is therefore inadequate, and a fundamentally different approach is needed.

The central new idea of this paper is \emph{sequential transversality}.
We chain the \(n\) waves sequentially from \(u_l\) through the intermediate states \(u^{(1)},\dots,u^{(n-1)}\) to \(u_r\), and transport the tangent vectors from each wave curve to a common reference point via \emph{pushforward maps} arising from the implicit function theorem applied to the Rankine--Hugoniot or rarefaction relations linking consecutive states.
The sequential transversality condition then asks whether these transported tangent vectors span the full state space \(\mathbb{R}^n\), which reduces to the nonvanishing of the determinant of an \(n\times n\) \emph{transversality matrix}.
For \(n=2\), sequential transversality recovers the classical pairwise transversality of our prior work, while for \(n\ge 3\) it provides the correct generalization.
The concept is developed in detail in Section~\ref{ssec:transversality-3x3} for the \(3\times 3\) case and in Section~\ref{ssec:transversality-nxn} for the general \(n\times n\) case.

Sequential transversality is a new concept in differential topology, motivated by the Riemann problem for $n\times n$ systems.
In standard transversality~\cite{GolubitskyGuillemin}, two submanifolds share a common point where the span of their tangent spaces is analyzed.
Sequential transversality gives the determinant condition a reading of exactly this form. The first $n-1$ columns of the transversality matrix span the tangent space at $u^{(n-1)}$ of the $(n-1)$-dimensional manifold swept out by the first $n-1$ waves, and the last column is tangent to the backward $n$-wave curve from $u_r$, so the condition says that this curve meets that manifold transversely.
What is new is the way this tangent space arises. The intermediate states \(u^{(1)},\dots,u^{(n-1)}\) are \emph{causally chained}, with each state determined by the previous one through a wave relation.
Perturbing \(u^{(1)}\) propagates through the entire chain, affecting \(u^{(2)}\), then \(u^{(3)}\), and so on.
The transport maps linking consecutive states are determined by the defining equations of the chain rather than by any ambient geometric structure (such as a Riemannian metric or a symplectic form), making sequential transversality a concept in smooth topology that requires only smooth manifolds and smooth relations with a non-degeneracy condition.

A second contribution is a coordinate-free construction of rarefaction curves via the \(C^1\) eigenline bundle and local \emph{eigencharts}, retaining the eigenvalue parameterization \(\xi=\lambda_k\) for a \(C^2\) flux and removing the graph condition imposed in our prior work~\cite{TanBertozzi2x2}. In the eigenvalue-parameterized form (see, e.g.,~\cite{LeVeque}), the rarefaction direction is proportional to \(r_k(u)/(\nabla\lambda_k(u)\cdot r_k(u))\), which is generally only \(C^0\) when \(F\in C^2\), so Picard--Lindel\"of does not give uniqueness for the \(\xi\)-ODE directly. The usual remedies are to require a $C^3$ flux, at the cost of excess regularity, or to switch to arc-length, at the cost of the characteristic-speed coordinate.

Our prior paper instead invoked the graph condition to reparameterize the curve by a state coordinate in a fixed Cartesian chart, recovering a $C^1$ right-hand side. This choice also gives up the parameter $\xi$ and depends on the coordinates, and extending the single-chart reduction to $n\ge 3$ forces additional structural assumptions that quickly become restrictive. Here the intrinsic object is the eigenline bundle $u\mapsto E_k(u)$ rather than any particular eigenvector field. The implicit function theorem with an affine coordinate normalization produces a $C^1$ representative $r_k$ of the eigenline near each state, on a neighborhood we call an \emph{eigenchart} (Lemma~\ref{lem:C1-evec-IFT}). On each eigenchart the unscaled flow ODE $\dot u=r_k(u)$ has a $C^1$ right-hand side and Picard--Lindel\"of applies; genuine nonlinearity makes $\lambda_k$ strictly monotone along the flow, and inverting turns the eigenvalue itself into the curve parameter. The $\xi$-ODE is unchanged under rescaling of the representative, so the chartwise pieces agree on overlaps by uniqueness of the $C^1$ flow and glue into a unique maximal rarefaction curve (Proposition~\ref{prop:Gamma-maximal}). The construction works in any dimension, requires only $C^2$ flux, and keeps the characteristic speed $\xi=\lambda_k$ as the wave parameter, which the lifted viewpoint below uses as its auxiliary speed coordinate.

Iguchi and LeFloch~\cite{IguchiLeFloch2003} also construct rarefaction curves in two steps, an eigenvector ODE reparameterized by a monotone weight, in a smooth local setting with eigenvector regularity taken as given. At bare $C^2$ regularity on a global state space that regularity is the central difficulty, and Lemma~\ref{lem:C1-evec-IFT} resolves it.

A third contribution is a unified \emph{lifted} viewpoint on wave curves in which the wave speed is carried as an auxiliary coordinate. Rather than viewing the Hugoniot locus and the rarefaction curve as subsets of the state space $U$ alone, we lift each to $U\times\mathbb{R}$, pairing a state $u$ with its shock speed $s$ along the Hugoniot locus and with its characteristic speed $\xi=\lambda_k(u)$ along the rarefaction curve. The parameter along each lifted curve then coincides with the wave-speed profile of the solution as it is traversed from left to right in the spacetime diagram, so that the Lax admissibility and gap conditions separating consecutive waves become statements about heights in the lift. Figures~\ref{fig:double-rare-example},~\ref{fig:shock-rare-example}, and~\ref{fig:triple-shock-example} illustrate this picture for three prototypical $2\times 2$ and $3\times 3$ wave configurations. This lifted viewpoint is what places shocks and rarefactions on a common footing in the proofs of the main structural stability and genericity theorems (\Cref{thm:structural-stability,thm:genericity}).

\subsection{Main results}

Under Assumptions~\ref{ass:SH},~\ref{ass:GN}, and~\ref{ass:regular}, the paper establishes three main theorems.
\begin{enumerate}[label=\textup{(\roman*)},leftmargin=*]
\item \textbf{Structural stability} (\cref{thm:structural-stability}).
  If the transversality condition holds at the zero of the grand objective map determined by the solution, then the Riemann solution is structurally stable under perturbations of the left and right states and the flux function. If in addition it holds strictly on a compact set whose interior contains the states of the solution, the perturbed solution is the unique admissible solution of that wave pattern nearby.
  The proof applies the implicit function theorem on Banach spaces (\cref{thm:IFT-Banach}) to the grand objective map~\eqref{eq:extended-objective}.
\item \textbf{Genericity of transversality} (\cref{thm:genericity}).
  For almost every \((u_l,u_r)\in U^2\), the transversality condition holds at every non-degenerate zero of the grand objective map, a zero with \(u^{(k)}\neq u^{(k-1)}\) for all \(k\).
  The proof uses the foliated parametric transversality theorem (\cref{thm:foliated-parametric-transversality}).
\item \textbf{Generic structural stability} (\cref{thm:generic-structural-stability}).
  Combining (i) and (ii), for almost every \((u_l,u_r)\in U^2\), the Riemann solution is structurally stable.
\end{enumerate}

\subsection{Applications}

The structural stability theorems have immediate consequences for a broad class of systems.
We illustrate this with three families of applications in Section~\ref{sec:applications}.
The first is the classical \(p\)-system of compressible isentropic flow in Lagrangian coordinates (Section~\ref{ssec:p-system}), for which we verify all three standing assumptions (strict hyperbolicity, genuine nonlinearity, and the regular manifold hypothesis) via direct computation.
The regular manifold hypothesis is verified by a direct rank computation on the lifted Rankine--Hugoniot Jacobian $D_{(u,s)}H_{u_0}$.
The second is gravity-driven particle-laden thin films (Section~\ref{ssec:PLF}), where a slurry of \(n-1\) distinct particle species leads to an \(n\times n\) system of conservation laws.
The \(n=2\) case recovers the monodisperse system treated in our prior work~\cite{TanBertozzi2x2}.
The general \(n\times n\) framework developed here provides the first structural stability results for polydisperse particle-laden flows.
The third family concerns machine-learned flux approximations (Section~\ref{ssec:ML-flux}).
A rapidly growing body of work in scientific machine learning seeks to learn or approximate the flux function of a conservation law from data, whether through neural networks that learn the physical flux directly~\cite{GoRINNs2025,SymCLaw2026}, symbolic regression that recovers closed-form flux expressions~\cite{ConsLawNet2023}, numerical flux networks that operate at the cell-interface level~\cite{CFN2024}, or trained neural networks that approximate the Riemann solver itself~\cite{MagieraRiemann2020,RuggeriRiemann2022,NogueiraRiemann2025,HCNRS2026}.
The structural stability theorems of this paper specify how close a learned flux must be to the true one for the wave structure of the Riemann solution to persist.
If the learned flux approximation is \(C^2\)-close to the true flux and the system satisfies the three standing assumptions, then the wave structure of the Riemann solution is guaranteed to be preserved for almost every left and right state.

\subsection{Outline}

The paper is organized as follows.
Section~\ref{sec:geometric} develops the geometric theory for wave curves, reviewing the Rankine--Hugoniot map, Hugoniot loci, rarefaction curves, and the construction of wave curves for the Riemann problem.
Section~\ref{sec:transversality} introduces the concept of sequential transversality, beginning with the \(2\times 2\) case, building intuition through the \(3\times 3\) case, and culminating in the general \(n\times n\) framework.
Section~\ref{sec:structural-stability} states and proves the three main theorems on structural stability, genericity, and generic structural stability.
Section~\ref{sec:applications} applies the theorems to the \(p\)-system, polydisperse particle-laden flows, and machine-learned flux approximations.
Section~\ref{sec:conclusion} summarizes the contribution and lists open directions.
The appendices collect relevant background material and proofs of auxiliary lemmas.

\section{Geometric Theory for Wave Curves}
\label{sec:geometric}

\subsection{The Rankine--Hugoniot map and the Hugoniot loci}
\label{ssec:hugoniot}

Throughout this section, $n \ge 2$ is a fixed integer and $u_0 \in U$ is a fixed base state.

A piecewise-smooth solution with a jump discontinuity propagating at
speed \(s\) must satisfy the \emph{Rankine--Hugoniot condition}
\begin{equation}\label{eq:rankine-hugoniot}
  F(u) - F(u_0) = s\,(u - u_0),
\end{equation}
defining candidate right states $u$ from the given left state $u_0$ for an admissible shock connection.

\begin{definition}[Rankine--Hugoniot map]\label{def:RH-map}
Define the Rankine--Hugoniot map based at \(u_0\) by
\begin{equation}\label{eq:RH-map}
  \begin{aligned}
    H_{u_0}\colon (U\setminus\{u_0\})\times\mathbb{R} &\to\mathbb{R}^n,\\
    H_{u_0}(u,s) &:= F(u) - F(u_0) - s\,(u - u_0).
  \end{aligned}
\end{equation}
When we wish to emphasize the base state explicitly, we also write
\(H(u,s;u_0) := H_{u_0}(u,s)\).
Throughout this section, the flux \(F\) is fixed and excluded from the notation for simplicity.
Its Jacobian with respect to \((u,s)\in\mathbb{R}^{n+1}\) is the
\(n\times(n{+}1)\) matrix
\begin{equation}\label{eq:jacobian}
\begin{aligned}
  D_{(u,s)}H_{u_0}(u,s)
  &= \bigl[\,D_u H_{u_0}(u,s) \;\bigm|\; D_s H_{u_0}(u,s)\,\bigr] \\
  &= \bigl[\,DF(u) - s\,I_n \;\bigm|\; -(u - u_0)\,\bigr].
\end{aligned}
\end{equation}
Here \(D_uH\in\mathbb{R}^{n\times n}\) and \(D_sH\in\mathbb{R}^{n\times 1}\); the first \(n\)
columns come from differentiating in \(u\), and the last column from differentiating in \(s\), with $I_n$ as the $n \times n$ identity matrix.
\end{definition}

\begin{definition}[Lifted Hugoniot locus]\label{def:lifted}
The \emph{lifted Hugoniot locus} emanating from \(u_0\) is given by
\[
  \widehat{\mathcal{H}}_{u_0}
  := \bigl\{(u,s)\in(U\setminus\{u_0\})\times\mathbb{R} : H_{u_0}(u,s)=0\bigr\}
  = H_{u_0}^{-1}(0).
\]
\end{definition}

\begin{definition}[Projection and projected Hugoniot locus]\label{def:projected}
Let \(\pi\colon(U\setminus\{u_0\})\times\mathbb{R}\to
U\setminus\{u_0\}\) denote the canonical projection
\(\pi(u,s)=u\).  The \emph{projected Hugoniot locus} is given by
\(
  \mathcal{H}_{u_0} := \pi\bigl(\widehat{\mathcal{H}}_{u_0}\bigr)
  \subset U\setminus\{u_0\}.
\)
\end{definition}

\begin{assumption}[Regular manifold hypothesis]\label{ass:regular}
For every \(u_0\in U\), the origin \(0\in\mathbb{R}^n\) is a regular value of
\(H_{u_0}\colon(U\setminus\{u_0\})\times\mathbb{R}\to\mathbb{R}^n\).
Equivalently, \(H_{u_0}\) is a submersion at every point of \(\widehat{\mathcal{H}}_{u_0}\), that is, the Jacobian \(D_{(u,s)}H_{u_0}\) has rank~\(n\) there.
\end{assumption}

See the prior work \cite{TanBertozzi2x2} for motivation and
geometric intuition behind this hypothesis, including why the base
state \(u_0\) must be excluded from the domain.

\begin{proposition}[Lifted locus is a 1D embedded submanifold]\label{prop:lifted-mfld}
Under Assumption~\ref{ass:regular}, the lifted Hugoniot locus
\(\widehat{\mathcal{H}}_{u_0}=H_{u_0}^{-1}(0)\) is a \(C^1\) embedded submanifold of
\((U\setminus\{u_0\})\times\mathbb{R}\) of dimension
\(
  (n+1) - n = 1.
\)
At each point \((u,s)\in\widehat{\mathcal{H}}_{u_0}\), the tangent space is given by
\(
  T_{(u,s)}\widehat{\mathcal{H}}_{u_0} = \ker\bigl(D_{(u,s)}H_{u_0}(u,s)\bigr).
\)
\end{proposition}

\begin{proof}
This is a direct application of the Regular Value Theorem
(Theorem~\ref{thm:rvt}) to \(G=H_{u_0}\), \(y=0\), with \(N=n+1\) and
\(k=n\).
\end{proof}

Note here that the lifted locus \(\widehat{\mathcal{H}}_{u_0}\) is a one-parameter family of
pairs \((u,s)\) satisfying the Rankine--Hugoniot jump conditions based
at \(u_0\). Next, we continue with one of the key geometric relations between the lifted and the projected Hugoniot loci.

\begin{proposition}[Equivalence Principle for Hugoniot Loci]\label{prop:equv-hugoniot}
Assume the regular manifold hypothesis
(Assumption~\ref{ass:regular}) holds. Then
\[
  \pi\big|_{\widehat{\mathcal{H}}_{u_0}} \colon
  \widehat{\mathcal{H}}_{u_0}\to U\setminus\{u_0\}
\]
is a \(C^1\) embedding. In particular, its image
\(\mathcal{H}_{u_0}=\pi(\widehat{\mathcal{H}}_{u_0})\) is a one-dimensional
\(C^1\) embedded submanifold of \(U\setminus\{u_0\}\).
\end{proposition}

The proof splits into three steps.

\begin{lemma}[Injectivity of the projection]\label{lem:injectivity}
For each \(u\in U\setminus\{u_0\}\), there is at most one \(s\) such
that \((u,s)\in\widehat{\mathcal{H}}_{u_0}\).  Hence
\(\pi|_{\widehat{\mathcal{H}}_{u_0}}\) is injective.
\end{lemma}

\begin{proof}
Suppose \((u,s_1)\) and \((u,s_2)\) both lie in
\(\widehat{\mathcal{H}}_{u_0}\).  Then \(H_{u_0}(u,s_1)=0=H_{u_0}(u,s_2)\) implying $
  0 = H_{u_0}(u,s_2) - H_{u_0}(u,s_1) = -(s_2 - s_1)(u - u_0).
$
Since \(u\neq u_0\), the vector \(u-u_0\in\mathbb{R}^n\) is nonzero,
so \(s_2-s_1=0\).
\end{proof}

\begin{lemma}[Projection is an immersion]\label{lem:immersion}
The restriction \(\pi|_{\widehat{\mathcal{H}}_{u_0}}\) is an immersion, i.e.\
its derivative is injective at every point of
\(\widehat{\mathcal{H}}_{u_0}\).
\end{lemma}

\begin{proof}
Since \(\widehat{\mathcal{H}}_{u_0}\subset (U\setminus\{u_0\})\times\mathbb{R}\subset\mathbb{R}^{n+1}\)
is an embedded submanifold, we view
\(T_{(u,s)}\widehat{\mathcal{H}}_{u_0}\) as a subspace of \(\mathbb{R}^{n+1}\)
via the canonical inclusion map. Let \(e_s:=(0,\dots,0,1)\in\mathbb{R}^{n+1}\) denote the unit vector
in the \(s\)-direction.  From the Jacobian~\eqref{eq:jacobian},
\[
  D_{(u,s)}H_{u_0}(u,s)\;e_s
  = \frac{\p H_{u_0}}{\p s}(u,s)
  = -(u - u_0).
\]
Recall that \(\widehat{\mathcal{H}}_{u_0}=H_{u_0}^{-1}(0)\) is the zero level set
of the Rankine--Hugoniot map.  By the Regular Value Theorem
(Theorem~\ref{thm:rvt}),
\(T_{(u,s)}\widehat{\mathcal{H}}_{u_0}=\ker\bigl(D_{(u,s)}H_{u_0}(u,s)\bigr)\)
for every \((u,s)\in\widehat{\mathcal{H}}_{u_0}\).
Since \(u\neq u_0\) on the domain, \(D_{(u,s)}H_{u_0}(u,s)e_s=-(u-u_0)\neq 0\), hence
\begin{equation}\label{eq:es-not-tangent}
  e_s \notin T_{(u,s)}\widehat{\mathcal{H}}_{u_0} = \ker\bigl(D_{(u,s)}H_{u_0}(u,s)\bigr).
\end{equation}

Next, the derivative of \(\pi(u,s)=u\) is given by
\[
  D\pi_{(u,s)}\colon\mathbb{R}^{n+1}\to\mathbb{R}^n,\qquad
  D\pi_{(u,s)}(\delta u,\delta s) = \delta u,
\]
so \(\ker(D\pi_{(u,s)}) = \spn\{e_s\}\).  Hence, by~\eqref{eq:es-not-tangent},
\[
  \ker(D\pi_{(u,s)}) \cap T_{(u,s)}\widehat{\mathcal{H}}_{u_0}
  = \spn\{e_s\} \cap T_{(u,s)}\widehat{\mathcal{H}}_{u_0}
  = \{0\}.
\]
Equivalently,
\(\ker\bigl(D\pi_{(u,s)}|_{T_{(u,s)}\widehat{\mathcal{H}}_{u_0}}\bigr)=\{0\}\),
so \(D\pi_{(u,s)}|_{T_{(u,s)}\widehat{\mathcal{H}}_{u_0}}\) is injective.
\end{proof}

Lemmas~\ref{lem:injectivity} and~\ref{lem:immersion} make \(\pi|_{\widehat{\mathcal{H}}_{u_0}}\) an injective immersion, hence a \emph{local} \(C^1\) embedding by Proposition~\ref{cor:immersion-local}. The Rankine--Hugoniot condition then yields an explicit \(C^1\) inverse \(\sigma\colon\mathcal{H}_{u_0}\to\mathbb{R}\) recovering the shock speed from the state, so Theorem~\ref{thm:embedding-image} promotes \(\mathcal{H}_{u_0}\) to a one-dimensional \(C^1\) embedded submanifold, with dimension inherited from Proposition~\ref{prop:lifted-mfld}. See Appendix~\ref{app:equiv-hugoniot-proof} for the full proof.


\begin{remark}[Removing \(u_0\)]\label{rem:remove-u0}
The partial derivative \(\p_s H_{u_0}(u,s) = -(u-u_0)\) degenerates at
\(u=u_0\), and the lifted set typically has a singularity there because every speed satisfies the Rankine--Hugoniot condition at \(u_0\).  Removing \(u_0\) from the domain
guarantees that the ``vertical direction is not tangent'' argument of
Lemma~\ref{lem:immersion} holds.
Exactly $n$ branches of $\widehat{\mathcal{H}}_{u_0}$, one per characteristic family, nevertheless approach the base state, and each extends continuously to a finite limiting point there.

Along the $k$-th branch the projected tangent at $u_0$ is a right eigenvector $r_k(u_0)$, and Taylor-expanding the relation $F(u)-F(u_0)=s(u-u_0)$ gives $s\to\lambda_k(u_0)$ as $u\to u_0$, since $DF(u_0)\,r_k(u_0)=\lambda_k(u_0)\,r_k(u_0)$. The open circle at $(u_0,\lambda_k(u_0))$ in our figures (e.g., Figure~\ref{fig:lifted-shock}) marks this limiting point.
\end{remark}

{\sloppy
\begin{remark}[Dimension Independence]\label{rem:nxn}
The argument is independent of the number of conservation laws $n$.  The flux
\(F\colon U\subset\mathbb{R}^n\to\mathbb{R}^n\) may have any
\(n\ge 2\).  The dimension count \((n{+}1)-n=1\) for the lifted locus,
the injectivity argument (nonzero \(u-u_0\) implies unique \(s\)),
and the immersion argument (vertical direction not tangent) are all
independent of~\(n\).
\end{remark}}

\begin{remark}[Equivalence Principle]
Proposition~\ref{prop:equv-hugoniot} is the technical bridge between the
``lifted'' viewpoint, where shock connections are naturally described as
the zero set \(H_{u_0}^{-1}(0)\subset (U\setminus\{u_0\})\times\mathbb{R}\), and the
traditional ``projected'' viewpoint in state space
\(\mathcal{H}_{u_0}\subset U\setminus\{u_0\}\).

Concretely, since \(\pi(u,s)=u\), we have
\(
  \pi(u_1,\dots,u_n,s)=(u_1,\dots,u_n)
\), and therefore the Jacobian matrix is given by
\[
  D\pi_{(u,s)} = \begin{bmatrix} I_n & 0 \end{bmatrix}
  \colon \mathbb{R}^{n+1}\to\mathbb{R}^n.
\]
In particular, for a tangent vector \((v,\dot{s})\in\mathbb{R}^n\times\mathbb{R}\cong T_{(u,s)}\bigl((U\setminus\{u_0\})\times\mathbb{R}\bigr)\),
\[
  D\pi_{(u,s)}(v,\dot{s}) = \begin{bmatrix} I_n & 0 \end{bmatrix}
  \binom{v}{\dot{s}} = v.
\]
Lemma~\ref{lem:immersion} implies that this loss of the
\(s\)-component does not collapse any tangent direction along
\(\widehat{\mathcal{H}}_{u_0}\), so \(D\pi\) restricts to an isomorphism on
tangent spaces. In other words, for every \((u,s)\in\widehat{\mathcal{H}}_{u_0}\),$ D\pi\big|_{T_{(u,s)}\widehat{\mathcal{H}}_{u_0}}\colon T_{(u,s)}\widehat{\mathcal{H}}_{u_0} \to T_u\mathcal{H}_{u_0}$
identifies tangent directions on the lifted locus with tangent directions on the projected locus (see Figure~\ref{fig:lifted-shock}).
\end{remark}
\begin{figure}[htbp]
  \centering
  \begin{tikzpicture}[>=Stealth, scale=1.0,
      every node/.style={font=\small}]
    \draw[thick] (0,0) ellipse (3.8cm and 1.6cm);
    \node[anchor=west] at (3.9,0) {$U\setminus\{u_0\}$};

    \coordinate (u0) at (-1.8,-0.7);
    \draw[fill=white, thick] (u0) circle (2.5pt);
    \node[below left] at (u0) {$u_0$};

    \draw[thick] (-3.3,0.1) .. controls (-2.8,-0.1) and (-2.4,-0.4) .. (u0)
                             .. controls (-1.2,-1.0) and (0.0,-0.1) .. (0.6,0.05)
                             .. controls (1.6,0.2) and (2.6,0.3) .. (3.2,0.5);
    \node[font=\small, rotate=20, anchor=south] at (-0.5,-0.25)
      {$\mathcal{H}_{u_0}=\pi(\widehat{\mathcal{H}}_{u_0})$};

    \coordinate (proj) at (0.6,0.05);
    \draw[thick] (proj) ++(-2.5pt,-2.5pt) -- ++(5pt,5pt)
                 (proj) ++(-2.5pt,2.5pt) -- ++(5pt,-5pt);

    \coordinate (proj-vend) at ($(proj)+(1.2,0.0)$);
    \draw[->, orange!80!red, very thick] (proj) -- (proj-vend)
      node[below, midway, orange!80!red] {$v$};

    \draw[dotted, thin] (-4.0,3.6) -- (-1.8,3.6);
    \node[font=\scriptsize, anchor=east] at (-4.0,3.6) {$\lambda_k(u_0)$};

    \draw[->, thick] (-4.0,4.0) -- (-4.0,4.6) node[above] {$s$};

    \coordinate (lifted) at (0.6,3.2);
    \coordinate (lifted-u0) at (-1.8,3.6);
    \draw[thick] (-3.3,2.9) .. controls (-2.9,2.95) and (-2.5,3.4) .. (lifted-u0)
                             .. controls (-1.1,3.7) and (-0.5,3.5) .. (lifted)
                             .. controls (1.0,3.07) and (2.2,2.8) .. (3.6,2.6);
    \node[anchor=south, font=\small] at (-0.5,3.7) {$\widehat{\mathcal{H}}_{u_0}$};

    \draw[fill=white, thick] (lifted-u0) circle (2.5pt);

    \draw[thick] (lifted) ++(-2.5pt,-2.5pt) -- ++(5pt,5pt)
                 (lifted) ++(-2.5pt,2.5pt) -- ++(5pt,-5pt);

    \draw[dashed, thick] (lifted) -- (proj);

    \coordinate (tangent-end) at ($(lifted)+(1.2,-0.4)$);
    \coordinate (vend) at ($(lifted)+(1.2,0)$);

    \draw[->, red, very thick] (lifted) -- (tangent-end);
    \node[below right, red] at (tangent-end)
      {$(\textcolor{orange!80!red}{v},\textcolor{blue}{\dot{s}})$};

    \draw[->, orange!80!red, thick] (lifted) -- (vend)
      node[above, midway, orange!80!red] {$v$};

    \draw[->, blue, thick] (vend) -- (tangent-end)
      node[pos=0.5, right, blue] {$\dot{s}$};

    \draw[orange!80!red, dashed, thick] (vend) -- (proj-vend);

    \draw[->, thick] (-4.5,3.2) to[bend right=20] node[left] {$\pi$} (-4.5,1.2);

    \draw[->, thick] (4.8,1.2) to[bend right=20] node[right] {$\pi^{-1}\big|_{\mathcal{H}_{u_0}}$} (4.8,3.2);
  \end{tikzpicture}
  \caption{Diagram illustrating the relationship between the lifted and the projected Hugoniot loci, and how the projection map maps the tangent vectors accordingly.}
  \label{fig:lifted-shock}
\end{figure}
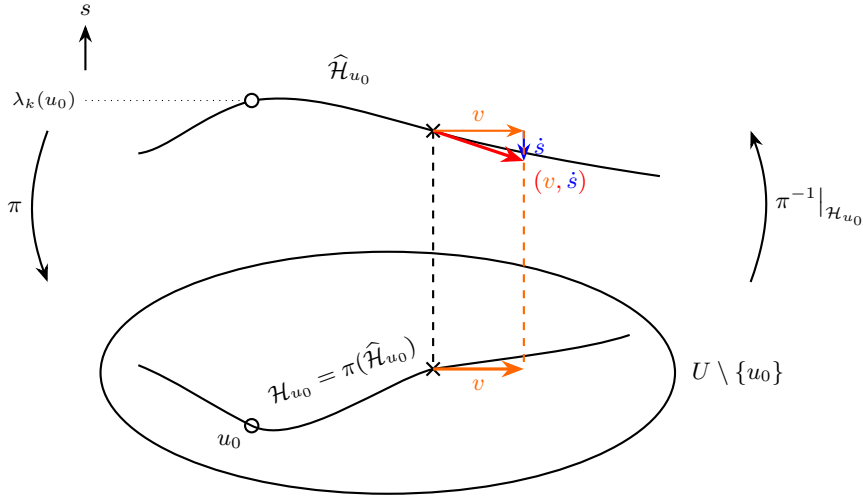
This paradigm extends to rarefaction curves in the
next subsection. There, the lifted object again lives in \(U\times\mathbb{R}\), but the auxiliary coordinate
is the characteristic speed \(\xi=\lambda_k(u)\) (rather than the shock speed \(s\)).

\subsection{Rarefaction curves}
\label{ssec:rarefaction}

\subsubsection{Flow construction and \(\xi\)-parameterization}

Fix a characteristic family \(k\in\{1,\dots,n\}\), a compact subset \(K\subset U\), and a base state \(u_0\in \operatorname{int}(K)\). In this subsection, the rarefaction curve is continued maximally while it remains in \(\operatorname{int}(K)\) (equivalently, up to first contact with \(\partial K\)).
Following the standard derivation~\cite{LeVeque,LeFloch}, a centered $k$-rarefaction fan is a nonconstant self-similar solution $u(x,t)=w(\xi)$, in the similarity variable $\xi=x/t$, so that \eqref{eq:conservation-law} gives
\[
  \bigl(DF(w(\xi))-\xi I_n\bigr)\,w'(\xi)=0.
\]
Wherever $w'(\xi)\neq 0$, the tangent $w'(\xi)$ is a right eigenvector of $DF(w(\xi))$ with eigenvalue $\xi$, so for the $k$-th family $\lambda_k(w(\xi))=\xi$ and $w'(\xi)=\alpha(\xi)\,r_k(w(\xi))$ for some scalar $\alpha(\xi)$. Differentiating $\lambda_k(w(\xi))=\xi$ gives $\nabla\lambda_k\cdot w'=1$, and substituting $w'=\alpha\,r_k$ determines the scalar, $\alpha=1/(\nabla\lambda_k\cdot r_k)$, well defined since $\nabla\lambda_k\cdot r_k\neq 0$ by genuine nonlinearity (Assumption~\ref{ass:GN}). Denoting the fan profile through the base state $u_0$ by $\Gamma_k(\cdot\,;u_0)$, the two relations become
\begin{equation}\label{eq:Gamma-lambda-param}
  \lambda_k\bigl(\Gamma_k(\xi;u_0)\bigr)=\xi
\end{equation}
and
\begin{equation}\label{eq:Gamma-xi-ODE}
  \Gamma_k'(\xi;u_0)
  =
  \frac{r_k(\Gamma_k(\xi;u_0))}
  {\nabla\lambda_k(\Gamma_k(\xi;u_0))\cdot r_k(\Gamma_k(\xi;u_0))},
\end{equation}
whose right-hand side is independent of the choice of representative $r_k$, since rescaling $r_k\mapsto c(u)\,r_k$ leaves the quotient unchanged.

At $F\in C^2$ the gradient $\nabla\lambda_k$ is only continuous, so the right-hand side of \eqref{eq:Gamma-xi-ODE} is merely $C^0$ in the state; Peano's theorem gives existence, but neither uniqueness nor $C^1$ dependence. For the same reason, the normalization $\nabla\lambda_k\cdot r=1$ need not define a $C^1$ eigenvector field, since it depends on $u$ through the $C^0$ coefficient $\nabla\lambda_k(u)$. We therefore proceed in two steps.
The implicit function theorem with the affine normalization on the $i$-th coordinate $r_i=1$, which involves no derivative of the eigenvalue, provides local $C^1$ representatives; these are integrated in their own flow time $t$, and the flow is reparameterized by $\xi=\lambda_k$ afterward, restoring the canonical normalization $\nabla\lambda_k\cdot\Gamma_k'=1$ exactly.

Restricting to \(K\), the Jacobian \(DF|_K\in C^1(K)\) since \(F\in C^2(U)\), and \(\lambda_k\) is simple at every point of \(K\) by Assumption~\ref{ass:SH}. The canonical object is the \(k\)-th eigenline
\[
  E_k(u):=\ker(DF(u)-\lambda_k(u)I_n),
\]
not a globally fixed normalization of a right eigenvector. If $r$ spans $E_k(u)$, then so does $\alpha r$ for every $\alpha\neq 0$, so passing to the eigenline discards exactly the two arbitrary features of an eigenvector, its length ($\alpha\neq 1$) and its sign ($\alpha<0$). All the construction requires is a choice of representative $r_k(u)\in E_k(u)$ that varies in a $C^1$ way as $u$ varies locally, that is, a nonvanishing $C^1$ map $r_k\colon\Omega\to\mathbb{R}^n\setminus\{0\}$ with $r_k(u)\in E_k(u)$ on an open neighborhood $\Omega$, called a \emph{local $C^1$ section} of the \emph{$k$-th eigenline bundle} $u\mapsto E_k(u)$; Appendix~\ref{app:eigenline-geometry} develops this language and the geometric picture behind it. The lemma below gives the construction. It is stated for a general matrix field $\mathcal{A}$; in this paper $\mathcal{A}=DF$ is the Jacobian of the flux.
\begin{lemma}[Local eigencharts for a simple eigenline]\label{lem:C1-evec-IFT}
Let \(V\subset\mathbb{R}^n\) be open and let \(\mathcal{A}\in C^1(V;\mathbb{R}^{n\times n})\). Fix \(u_0\in V\), and suppose \(\lambda_0\) is a simple eigenvalue of \(\mathcal{A}(u_0)\) with right eigenvector \(r_0\) satisfying \(\mathcal{A}(u_0)r_0=\lambda_0 r_0\). Choose an index \(i\) such that \((r_0)_i\neq 0\). Then there is an open neighborhood \(\Omega\subset V\) of \(u_0\) and there are \(C^1\) maps \(u\mapsto \lambda(u)\in\mathbb{R}\) and \(u\mapsto r(u)\in\mathbb{R}^n\) such that for all \(u\in\Omega\),
\[
  \mathcal{A}(u)r(u)=\lambda(u)r(u),\qquad r_i(u)=1,
\]
and \(\lambda(u_0)=\lambda_0\), \(r(u_0)=r_0/(r_0)_i\). After shrinking \(\Omega\), \(\lambda(u)\) is the unique simple eigenvalue of \(\mathcal{A}(u)\) near \(\lambda_0\). In the setting of this paper, with \(\mathcal{A}=DF\) on \(V=\operatorname{int}(K)\), \(\lambda_0=\lambda_k(u_0)\), and Assumption~\ref{ass:SH}, this gives \(\lambda(u)=\lambda_k(u)\), so \(r(u)\) is a local \(C^1\) representative of the eigenline \(E_k(u)\); on overlaps, any two such representatives differ by multiplication by a nonvanishing \(C^1\) scalar.
\end{lemma}

The proof of \cref{lem:C1-evec-IFT} is deferred to Appendix~\ref{app:evec-IFT-proof}. We call a neighborhood \(\Omega\) produced by the lemma, equipped with its normalized \(C^1\) representative \(r_k\), an \emph{eigenchart}. The normalization \(r_i=1\) generalizes the graph condition from our prior \(2\times 2\) work~\cite{TanBertozzi2x2}, where a global graph condition on the flux let each rarefaction curve be written as a graph over one state coordinate in a fixed Cartesian chart and produced a \(C^1\) right-hand side. Choosing \(i\) with \((r_0)_i\neq 0\) achieves the same locally on each eigenchart, removing the need for any single-chart structural assumption. The normalization $\nabla\lambda_k\cdot r_k=1$ used for genuinely nonlinear fields is generally only $C^0$ at $F\in C^2$, as discussed before the lemma, while the unit normalization $|r_k|=1$ is locally $C^1$ only after a choice of sign, and a continuous global choice of sign need not exist. The coordinate normalization above, $r_i=1$ for a coordinate $i$ with $(r_0)_i\neq 0$, thus avoids both issues on each eigenchart.

On each eigenchart \( \Omega\subset \operatorname{int}(K) \), the chosen representative \(r_k\in C^1(\Omega)\) is locally Lipschitz. We therefore first consider the flow map induced by the integral curves of \(r_k\) in its natural temporal coordinate \(t\), given by
\begin{equation}\label{eq:rk-flow}
  \dot u(t)=r_k(u(t)),\qquad u(0)=u_0\in\Omega.
\end{equation}

\begin{proposition}[Maximal \(t\)-flow and compact-set continuation]\label{prop:rk-maximal}
Let \(\Omega\subset\mathbb{R}^n\) be open and \(r_k\in C^1(\Omega)\). For each \(u_0\in\Omega\), there exists a unique maximal solution
\(u:(t_-,t_+)\to\Omega\) of \eqref{eq:rk-flow} with \(u(0)=u_0\), where \(-\infty\le t_-<0<t_+\le\infty\). If \(t_+<\infty\), then \(u(t)\) leaves every compact subset of \(\Omega\) as \(t\uparrow t_+\) (and similarly at \(t_-\)).
\end{proposition}
The proof follows from the Picard--Lindel\"of theorem together with the extension theorem for maximal solutions, since $r_k\in C^1(\Omega)$ is locally Lipschitz; see~\cite[Ch.~II, Theorems~1.1 and~3.1]{Hartman}.

For an eigenchart \(\Omega\), denote the local maximal flow for \eqref{eq:rk-flow} by
\[
  (t,u_*)\mapsto \phi_k(t;u_*).
\]
The right-hand sides of the flow \eqref{eq:rk-flow} and the rarefaction ODE \eqref{eq:Gamma-xi-ODE} are nonzero scalar multiples of one another, so solutions of the two ODEs are tangent to the same direction at every state and trace the same geometric curve; the two ODEs differ only in how fast this curve is traversed. Comparing how $\lambda_k$ advances along each solution quantifies the two speeds. The chain rule gives
\[
  \frac{d}{dt}\,\lambda_k(\phi_k(t;u_*))=\nabla\lambda_k\cdot r_k
  \qquad\text{and}\qquad
  \frac{d}{d\xi}\,\lambda_k(\Gamma_k(\xi;u_0))=\nabla\lambda_k\cdot\frac{r_k}{\nabla\lambda_k\cdot r_k}=1,
\]
where the first rate is nonvanishing by Assumption~\ref{ass:GN}, so $\xi=\lambda_k$ is strictly monotone along the flow (in dynamical-systems terms, a strict Lyapunov-type function for \eqref{eq:rk-flow}), and the second recovers the parameterization \eqref{eq:Gamma-lambda-param}. Hence \eqref{eq:Gamma-xi-ODE} is the reparameterization of \eqref{eq:rk-flow} by $\xi=\lambda_k$, with conversion rate $d\xi/dt=\nabla\lambda_k\cdot r_k$; an infinitesimal increment $dt$ of the auxiliary, representative-dependent time $t$ advances the eigenvalue parameter $\xi$ by $\nabla\lambda_k\cdot r_k\,dt$. \Cref{fig:level-set-crossing} illustrates this crossing picture; the parameterization \eqref{eq:Gamma-lambda-param} places $\Gamma_k(\xi;u_0)$ on the level set $\{\lambda_k=\xi\}$, so the $t$-flow crosses the level sets of $\lambda_k$ at rate $\nabla\lambda_k\cdot r_k$ per unit time while $\Gamma_k$ crosses exactly one per unit of $\xi$. The chartwise $\xi$-pieces obtained in this way inherit uniqueness from the $t$-flow and glue across eigencharts into a global curve, as described in the following proposition.

\begin{proposition}[Global maximal \(\xi\)-rarefaction curve in \(\operatorname{int}(K)\)]\label{prop:Gamma-maximal}
Fix $u_0\in \operatorname{int}(K)$ and set $\xi_0:=\lambda_k(u_0)$. There exists a unique maximal pair \(I=(\xi_-,\xi_+)\ni \xi_0\) and \(\Gamma_k(\cdot;u_0)\in C^1(I;\operatorname{int}(K))\) satisfying the initial condition \(\Gamma_k(\xi_0;u_0)=u_0\), the parameterization \eqref{eq:Gamma-lambda-param} on \(I\), and the ODE \eqref{eq:Gamma-xi-ODE} on every eigenchart traversed by the curve. The interval \(I\) is maximal among intervals on which such a curve remains in \(\operatorname{int}(K)\). If $\xi_+<\infty$, then $\Gamma_k(\xi;u_0)$ leaves every compact subset of $\operatorname{int}(K)$ as $\xi\uparrow\xi_+$ (and similarly at $\xi_-$).
\end{proposition}

We briefly sketch the proof and leave the details to Appendix~\ref{app:Gamma-maximal-proof}. On an eigenchart, \(\Xi(t):=\lambda_k(\phi_k(t;u_*))\) is strictly monotone by the chain-rule computation above, so the inverse function theorem gives a local inverse \(t=t(\xi)\), and \(\Gamma_k(\xi):=\phi_k(t(\xi);u_*)\) satisfies \eqref{eq:Gamma-lambda-param} and \eqref{eq:Gamma-xi-ODE}. The rest of the proof relies on the uniqueness of the \(t\)-flow from \cref{prop:rk-maximal}. If two local \(\xi\)-pieces overlap, pulling both back to \(t\)-flows in a common eigenchart forces agreement, so the compatible local pieces glue into a unique maximal \(\xi\)-curve.

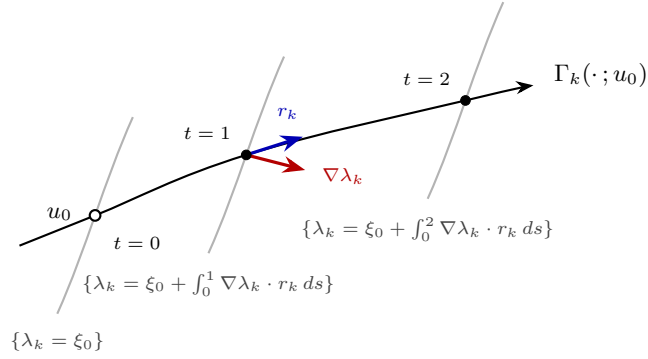
\begin{figure}[t]
  \centering
  \begin{tikzpicture}[>=Stealth, scale=1.0,
      every node/.style={font=\small}]

    \coordinate (P0) at (0.0,-0.8);
    \coordinate (P1) at (2.0, 0.0);
    \coordinate (P2) at (4.9, 0.72);

    \draw[gray!60, thick]
      ($(P0)+(-0.50,-1.30)$) .. controls ($(P0)+(-0.11,-0.43)$)
      and ($(P0)+(0.11,0.43)$) .. ($(P0)+(0.50,1.30)$);
    \draw[gray!60, thick]
      ($(P1)+(-0.50,-1.30)$) .. controls ($(P1)+(-0.11,-0.43)$)
      and ($(P1)+(0.11,0.43)$) .. ($(P1)+(0.50,1.30)$);
    \draw[gray!60, thick]
      ($(P2)+(-0.50,-1.30)$) .. controls ($(P2)+(-0.11,-0.43)$)
      and ($(P2)+(0.11,0.43)$) .. ($(P2)+(0.50,1.30)$);

    \node[gray!50!black, font=\scriptsize, anchor=north]
      at ($(P0)+(-0.50,-1.42)$) {$\{\lambda_k=\xi_0\}$};
    \node[gray!50!black, font=\scriptsize, anchor=north]
      at ($(P1)+(-0.50,-1.42)$)
      {$\{\lambda_k=\xi_0+\int_0^1\nabla\lambda_k\cdot r_k\,ds\}$};
    \node[gray!50!black, font=\scriptsize, anchor=north]
      at ($(P2)+(-0.50,-1.42)$)
      {$\{\lambda_k=\xi_0+\int_0^2\nabla\lambda_k\cdot r_k\,ds\}$};

    \draw[thick, ->]
      plot [smooth, tension=0.7] coordinates
      {(-1.0,-1.2) (P0) (P1) (P2) (5.8,0.92)};
    \node[anchor=west] at (5.95,1.05) {$\Gamma_k(\cdot\,;u_0)$};

    \draw[->, very thick, blue!70!black] (P1) -- ($(P1)+(0.75,0.24)$);
    \node[blue!70!black, font=\scriptsize, anchor=south]
      at ($(P1)+(0.55,0.35)$) {$r_k$};

    \draw[->, very thick, red!70!black] (P1) -- ($(P1)+(0.78,-0.20)$);
    \node[red!70!black, font=\scriptsize, anchor=west]
      at ($(P1)+(0.88,-0.28)$) {$\nabla\lambda_k$};

    \node[font=\scriptsize, anchor=north west]
      at ($(P0)+(0.12,-0.16)$) {$t=0$};
    \node[font=\scriptsize, anchor=south east]
      at ($(P1)+(-0.08,0.08)$) {$t=1$};
    \node[font=\scriptsize, anchor=south east]
      at ($(P2)+(-0.08,0.08)$) {$t=2$};

    \draw[fill=white, thick] (P0) circle (2.0pt);
    \node[anchor=east] at ($(P0)+(-0.18,0.02)$) {$u_0$};
    \draw[fill=black] (P1) circle (1.8pt);
    \draw[fill=black] (P2) circle (1.8pt);

  \end{tikzpicture}
  \caption{The rarefaction curve through \(u_0\), marked at unit
    flow times \(t=0,1,2\). The flow \eqref{eq:rk-flow} crosses the level
    sets of \(\lambda_k\) at rate \(\nabla\lambda_k\cdot r_k\) per unit
    \(t\), while \eqref{eq:Gamma-xi-ODE} crosses one level set per unit
    \(\xi\).}
  \label{fig:level-set-crossing}
\end{figure}

From now on, let \(I\) denote the maximal \(\xi\)-interval from Proposition~\ref{prop:Gamma-maximal} and \(\xi_0=\lambda_k(u_0)\). Define the forward and backward branches of the rarefaction curve as
\[
  \mathcal{R}^+_{k,u_0}:=\{\Gamma_k(\xi;u_0):\xi\in I,\ \xi>\xi_0\},
  \qquad
  \mathcal{R}^-_{k,u_0}:=\{\Gamma_k(\xi;u_0):\xi\in I,\ \xi<\xi_0\}
\]
respectively. The sign of \(\nabla\lambda_k(u_0)\cdot r_k(u_0)\) fixes the orientation of increasing \(t\) versus increasing \(\xi\), since the chain rule gives \(\tfrac{d}{dt}\lambda_k(\phi_k(t;u_0))=\nabla\lambda_k\cdot r_k\) along the $r_k$-flow.

\subsubsection{Rarefaction map}

Compared with the lifted/projected Hugoniot construction, rarefaction curves are first constructed directly in state space. Hence, the goal of this subsection is to introduce the auxiliary parameter \(\xi=\lambda_k\) so as to place rarefaction and shock objects in the same lifted framework.

\begin{definition}[Rarefaction map]\label{def:rarefaction-map}
Define the rarefaction map based at \(u_0\), with \(I=(\xi_-,\xi_+)\) from Proposition~\ref{prop:Gamma-maximal}, by
\[
  R_{k,u_0}\colon \operatorname{int}(K)\times I\to\mathbb{R}^n,\qquad
  R_{k,u_0}(u,\xi) := u - \Gamma_k(\xi;u_0).
\]
To emphasize the base state explicitly, we can also write
\(R_k(u,\xi;u_0) := R_{k,u_0}(u,\xi)\).
\end{definition}

Since \(\Gamma_k(\cdot;u_0)\in C^1(I;\operatorname{int}(K))\), the map \(R_{k,u_0}\) is \(C^1\) on \(\operatorname{int}(K)\times I\).
The lifted \(k\)-th rarefaction curve from \(u_0\) is the zero level set of
\(R_{k,u_0}\).  Differentiating with respect to \((u,\xi)\),
\begin{equation}\label{eq:rarefaction-jacobian}
  D_{(u,\xi)}R_{k,u_0}(u,\xi)
  = \bigl[\,I_n \;\bigm|\; -\Gamma_k'(\xi;u_0)\,\bigr].
\end{equation}

\begin{definition}[Lifted and projected rarefaction curves]\label{def:lifted-rarefaction}
Let \(\pi\colon \operatorname{int}(K)\times I\to \operatorname{int}(K)\) be the projection \(\pi(u,\xi)=u\).
The \emph{lifted \(k\)-th rarefaction curve} is defined by
\[
  \widehat{\mathcal{R}}_{k,u_0}
  := \bigl\{(u,\xi)\in \operatorname{int}(K)\times I : R_{k,u_0}(u,\xi)=0\bigr\}
  = R_{k,u_0}^{-1}(0).
\]
Equivalently,
\begin{equation}\label{eq:lifted-rarefaction-graph}
  \widehat{\mathcal{R}}_{k,u_0}
  = \bigl\{(\Gamma_k(\xi;u_0),\xi):\xi\in I\bigr\}.
\end{equation}
Correspondingly, the \emph{projected \(k\)-th rarefaction curve} is defined by
\[
  \mathcal{R}_{k,u_0}
  := \pi(\widehat{\mathcal{R}}_{k,u_0})
  = \Gamma_k(I;u_0)
  \subset \operatorname{int}(K).
\]
\end{definition}

\begin{figure}[t]
  \centering
  \begin{tikzpicture}[>=Stealth, scale=1.0,
      every node/.style={font=\small}]
    \draw[thick] (0,0) ellipse (3.8cm and 1.6cm);
    \node[anchor=west] at (3.9,0) {$U$};

    \draw[thick] (-2.8,-0.3) .. controls (-1.8,-0.4) and (-0.6,0.0) .. (0.4,0.3)
                              .. controls (1.4,0.6) and (2.4,0.4) .. (3.2,0.5);
    \node[font=\small, rotate=12, anchor=south] at (-1.5,0.05)
      {$\mathcal{R}_{k,u_0}=\pi(\widehat{\mathcal{R}}_{k,u_0})$};

    \coordinate (proj) at (0.4,0.3);
    \draw[thick] (proj) ++(-2.5pt,-2.5pt) -- ++(5pt,5pt)
                 (proj) ++(-2.5pt,2.5pt) -- ++(5pt,-5pt);
    \node[below] at (proj) {$u_0$};

    \draw[->, thick] (-4.0,3.6) -- (-4.0,4.3) node[above] {$\xi$};

    \coordinate (lifted) at (0.4,3.2);
    \draw[thick] (-2.2,2.2) .. controls (-1.0,2.6) and (0.0,3.0) .. (lifted)
                             .. controls (1.2,3.6) and (2.2,4.0) .. (3.4,4.4);
    \node[anchor=south, font=\small] at (-0.8,3.0) {$\widehat{\mathcal{R}}_{k,u_0}$};

    \node[anchor=west, font=\small\itshape] at (3.0,4.6) {strictly monotone};
    \draw[->, thin] (3.0,4.55) -- (2.6,4.2);

    \draw[thick] (lifted) ++(-2.5pt,-2.5pt) -- ++(5pt,5pt)
                 (lifted) ++(-2.5pt,2.5pt) -- ++(5pt,-5pt);

    \draw[dashed, thick] (lifted) -- (proj);

    \draw[->, thick] (-4.5,3.2) to[bend right=20] node[left] {$\pi$} (-4.5,1.2);

    \draw[->, thick] (4.8,1.2) to[bend right=20] node[right] {$\pi^{-1}\big|_{\mathcal{R}_{k,u_0}}$} (4.8,3.2);
  \end{tikzpicture}
  \caption{Diagram illustrating the relationship between the lifted and the original rarefaction curves.}
  \label{fig:lifted-rare}
\end{figure}
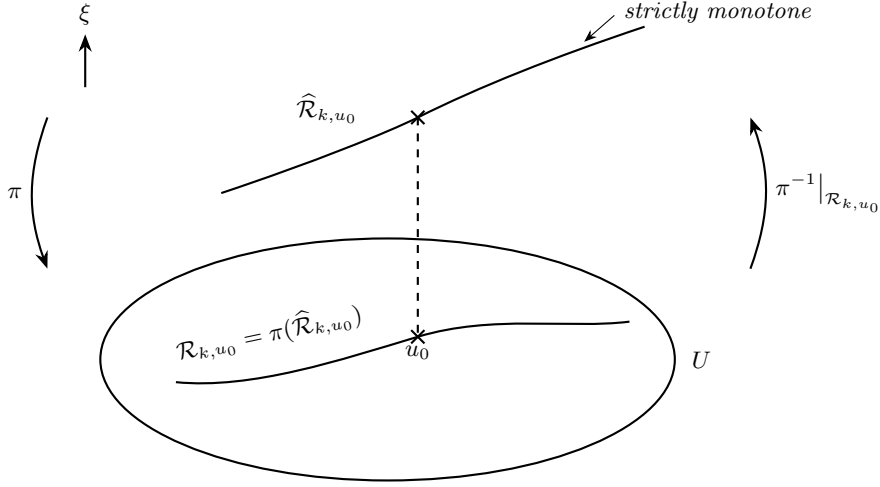

\begin{proposition}[Lifted rarefaction curve is a 1D submanifold]\label{prop:lifted-rarefaction}
The lifted rarefaction curve \(\widehat{\mathcal{R}}_{k,u_0}\) is a \(C^1\)
embedded submanifold of \(\operatorname{int}(K)\times I\) of dimension~\(1\).
\end{proposition}

\begin{proof}
The Jacobian~\eqref{eq:rarefaction-jacobian} contains \(I_n\) as its
first \(n\) columns (so it is an \(n\times(n+1)\) matrix with an identity block), hence \(D_{(u,\xi)}R_{k,u_0}(u,\xi)\) has rank~\(n\) everywhere.
Hence \(0\) is a regular value of \(R_{k,u_0}\), and the Regular Value
Theorem (Theorem~\ref{thm:rvt}) gives
\(\dim\widehat{\mathcal{R}}_{k,u_0}=(n+1)-n=1\).
\end{proof}

\begin{proposition}[Projected rarefaction curve is a 1D embedded submanifold]\label{prop:rarefaction-embedding}
Under the standing assumptions of this subsection, the parameterization
\[
  \Gamma_k(\cdot;u_0)\colon I\to \operatorname{int}(K)
\]
is a \(C^1\) embedding onto its image
\(\mathcal{R}_{k,u_0}=\Gamma_k(I;u_0)\). Consequently, \(\mathcal{R}_{k,u_0}\)
is a one-dimensional \(C^1\) embedded submanifold of~\(\operatorname{int}(K)\).
\end{proposition}

\begin{proof}
\emph{Immersion.}
From~\eqref{eq:Gamma-xi-ODE},
\[
  \Gamma_k'(\xi;u_0)
  =\frac{r_k(\Gamma_k(\xi;u_0))}
  {\nabla\lambda_k(\Gamma_k(\xi;u_0))\cdot r_k(\Gamma_k(\xi;u_0))}.
\]
The numerator is nonzero (eigenvectors are nonzero), and the denominator is
nonzero by genuine nonlinearity. Hence \(\Gamma_k'(\xi;u_0)\neq 0\) on \(I\),
so \(\Gamma_k\) is an immersion.

\emph{Injectivity.}
If \(\Gamma_k(\xi_1;u_0)=\Gamma_k(\xi_2;u_0)\), then applying \(\lambda_k\) and
using~\eqref{eq:Gamma-lambda-param} gives
\[
  \xi_1=\lambda_k(\Gamma_k(\xi_1;u_0))
  =\lambda_k(\Gamma_k(\xi_2;u_0))
  =\xi_2.
\]
Thus \(\Gamma_k\) is injective.

\emph{Topological embedding and conclusion.}
On \(\mathcal{R}_{k,u_0}=\Gamma_k(I;u_0)\), the inverse is given by
\[
  (\Gamma_k(\cdot;u_0))^{-1}(u)=\lambda_k(u),
\]
by~\eqref{eq:Gamma-lambda-param}. Since \(\lambda_k\) is continuous on
\(\operatorname{int}(K)\), this inverse is continuous (for the subspace topology
on \(\mathcal{R}_{k,u_0}\)). Therefore \(\Gamma_k\colon I\to\mathcal{R}_{k,u_0}\) is a
\(C^1\) embedding. By Theorem~\ref{thm:embedding-image},
\(\mathcal{R}_{k,u_0}\) is a \(C^1\) embedded submanifold of
\(\operatorname{int}(K)\), and since \(\dim I=1\), it is one-dimensional.
\end{proof}

\begin{remark}\label{rem:rarefaction-no-separate-proj-proof}
We do not separately prove \(\pi|_{\widehat{\mathcal{R}}_{k,u_0}}\) is a \(C^1\) embedding. For the Hugoniot loci, the natural geometric object is the lifted space 
\(U\times\mathbb{R}\). 
For rarefaction curves, the natural object already lives in the state
space as \(\Gamma_k(I;u_0)\subset \operatorname{int}(K)\), produced by the flow \eqref{eq:rk-flow} in the state variable alone, with wave speed given by the known function $\xi=\lambda_k(u)$ of the state rather than by an additional unknown (shock speed). 
\end{remark}

\subsection{Riemann problem for $n \times n$ systems}
\label{ssec:wave-curves}

The Riemann problem asks one to solve the conservation law with piecewise-constant initial data \eqref{eq:conservation-law}--\eqref{eq:Riemann-data}.
The standard approach seeks a self-similar solution consisting of \(n\) elementary waves, each belonging to a distinct characteristic family.
Entropy conditions then select which waves are physically admissible.
This section introduces the Lax entropy condition, defines wave curves, and illustrates the resulting structure through three examples, two of which are $2 \times 2$ systems and one of which is a $3 \times 3$ system.
In these examples, we illustrate how the lifted-space framework of Sections~\ref{ssec:hugoniot}--\ref{ssec:rarefaction} provides a convenient visualization of the speed-ordering conditions between successive waves.

\subsubsection{Lax entropy conditions and wave curves}

\begin{definition}[Lax entropy condition]\label{def:lax-entropy}
A discontinuity connecting a left state \(u_l\) to a right state \(u_r\) with speed~\(s\) is a \(k\)\textbf{-shock} (\(1\le k\le n\)) if \(H_{u_l}(u_r,s)=0\) (see \eqref{eq:RH-map}) and
\begin{equation}\label{eq:lax-entropy}
\lambda_k(u_l) > s > \lambda_k(u_r),\quad
\lambda_{k-1}(u_l) < s < \lambda_{k+1}(u_r),
\end{equation}
where, in the second pair of inequalities, the half involving $\lambda_0$ or $\lambda_{n+1}$ is omitted when the subscript falls outside $\{1,\dots,n\}$~\cite{Lax}.
\end{definition}

In view of the entropy condition, we divide the Hugoniot locus into two branches.

\begin{definition}[Hugoniot sub-loci \(\mathcal{H}^-_{k,u_0}\), \(\mathcal{H}^+_{k,u_0}\)]\label{def:hugoniot-branches}
For a given state \(u_0\in\operatorname{int}(K)\), define
\[
  \mathcal{H}^-_{k,u_0}
  :=\bigl\{u\in\mathcal{H}_{u_0}:
     \lambda_k(u_0)>s>\lambda_k(u),\;
     \lambda_{k-1}(u_0)<s<\lambda_{k+1}(u)\bigr\},
\]
\[
  \mathcal{H}^+_{k,u_0}
  :=\bigl\{u\in\mathcal{H}_{u_0}:
     \lambda_k(u_0)<s<\lambda_k(u),\;
     \lambda_{k+1}(u_0)>s>\lambda_{k-1}(u)\bigr\},
\]
where \(s\) is the unique shock speed associated to the pair \((u_0,u)\) by the Rankine--Hugoniot map \(H_{u_0}\) \eqref{eq:RH-map}.
The superscript sign indicates whether the connecting state has a higher (\(+\)) or lower (\(-\)) \(k\)-eigenvalue than \(u_0\).
\end{definition}

The rarefaction branches \(\mathcal{R}^+_{k,u_0}\) and \(\mathcal{R}^-_{k,u_0}\) were defined after Proposition~\ref{prop:Gamma-maximal}. The \(+\) (resp.\ \(-\)) branch consists of states along the \(k\)-rarefaction curve with \(\lambda_k\) increasing (resp.\ decreasing) from \(\lambda_k(u_0)\).

\begin{definition}[Forward and backward wave curves]\label{def:wave-curves}
For each family \(k\in\{1,\dots,n\}\) and base state \(u_0\in\operatorname{int}(K)\), the \textbf{forward} and \textbf{backward} \(k\)-wave curves are given by
$$
  \mathcal{W}^f_{k,u_0}
  := \mathcal{R}^+_{k,u_0}\cup\mathcal{H}^-_{k,u_0},\quad
  \mathcal{W}^b_{k,u_0}
  := \mathcal{R}^-_{k,u_0}\cup\mathcal{H}^+_{k,u_0},
$$
following the notation of \cite{Liu_shocks}.
\end{definition}

In an \(n\times n\) Riemann problem, one connects \(u_l\) to \(u_r\) through a chain of \(n\) elementary waves with intermediate states
\[
  u^{(0)}=u_l,\quad
  u^{(1)},\;\dots,\;u^{(n-1)},\quad
  u^{(n)}=u_r,
\]
where the \(j\)-th connection lies on the \(j\)-family forward wave curve, so that \(u^{(j)}\in\mathcal{W}^f_{j,u^{(j-1)}}\) for each \(j=1,\dots,n\).
The wave speeds (shock speeds or characteristic speeds) must be ordered so that each wave propagates faster than the preceding one.

\begin{remark}[Role of backward wave curves]\label{rem:backward-curves}
The forward-chain formulation above is the standard way to pose the Riemann problem.
In practice, however, one cannot determine the intermediate states \(u^{(1)},\dots,u^{(n-1)}\) \textit{a priori}.
For instance, in the \(2\times 2\) case, finding \(u^{(1)}\) amounts to intersecting the forward 1-wave curve from \(u_l\) with the backward 2-wave curve from \(u_r\).
If the intermediate state were already known, the solution would simply be ``forward from \(u_l\) to \(u^{(1)}\), then forward from \(u^{(1)}\) to \(u_r\).''
Finding \(u^{(1)}\), however, is precisely the intersection problem.
The backward formulation makes this intersection geometrically visible, and motivates the approach taken in our prior \(2\times 2\) paper \cite{TanBertozzi2x2}.
\end{remark}

\subsubsection{Example I: Double Rarefaction (\texorpdfstring{$2\times 2$}{2x2})}
\label{sssec:double-rare}

Consider a \(2\times 2\) system in which the Riemann problem is solved by a 1-rarefaction followed by a 2-rarefaction:
\[
  u_l \xrightarrow{\;\xi\in[\lambda_1(u_l),\,\lambda_1(u_1^*)]\;} u_1^*
       \xrightarrow{\;\xi\in[\lambda_2(u_1^*),\,\lambda_2(u_r)]\;} u_r.
\]
The intermediate state \(u_1^*\) lies in \(\mathcal{R}^+_{1,u_l}\cap \mathcal{R}^-_{2,u_r}\), and
the 1-rarefaction fan occupies characteristic speeds from \(\lambda_1(u_l)\) to \(\lambda_1(u_1^*)\), while the 2-rarefaction fan sweeps from \(\lambda_2(u_1^*)\) to \(\lambda_2(u_r)\).
For the two fans to be separated, one needs
\begin{equation}\label{eq:double-rare-gap}
  \lambda_1(u_1^*) < \lambda_2(u_1^*).
\end{equation}
Strict hyperbolicity guarantees \(\lambda_1(\cdot)<\lambda_2(\cdot)\) everywhere, so this gap condition is automatically satisfied.
This stands in contrast to the shock--rarefaction case \eqref{eq:shock-rare-gap}, where the gap $s_1^*<\lambda_2(u_1^*)$ comes from admissibility of the 1-shock rather than from strict hyperbolicity. See \cref{fig:double-rare-example} for an illustration.

\input{figures/fig-double-rare}

\subsubsection{Example II: Shock--Rarefaction (\texorpdfstring{$2\times 2$}{2x2})}
\label{sssec:shock-rare}

Consider a \(2\times 2\) system (\(n=2\)) in which the Riemann problem is solved by a 1-shock followed by a 2-rarefaction:
\[
  u_l \xrightarrow{\;s_1^*\;} u_1^* \xrightarrow{\;\xi\in[\lambda_2(u_1^*),\,\lambda_2(u_r)]\;} u_r.
\]
The intermediate state \(u_1^*\) lies in \(\mathcal{H}^-_{1,u_l}\cap \mathcal{R}^-_{2,u_r}\), with
the Lax entropy condition for the 1-shock giving \(\lambda_1(u_l)>s_1^*>\lambda_1(u_1^*)\), while the 2-rarefaction connects \(u_1^*\) to \(u_r\) through a fan of characteristic speeds sweeping from \(\lambda_2(u_1^*)\) to \(\lambda_2(u_r)\).
Hence, a necessary condition for the waves to be separated in the \((x,t)\)-plane is given by
\begin{equation}\label{eq:shock-rare-gap}
  s_1^* < \lambda_2(u_1^*).
\end{equation}
See \cref{fig:shock-rare-example} for an illustration.

\input{figures/fig-shock-rare}

\subsubsection{Example III: Triple Shock (\texorpdfstring{$3\times 3$}{3x3})}
\label{sssec:triple-shock}

Consider a \(3\times 3\) system (\(n=3\)) in which the Riemann problem is solved by three shocks:
\[
  u_l \xrightarrow{\;s_1^*\;} u_1^*
       \xrightarrow{\;s_2^*\;} u_2^*
       \xrightarrow{\;s_3^*\;} u_r.
\]
The three shocks appear left to right in the $(x,t)$-plane, satisfying
\begin{equation}\label{eq:triple-shock-ordering}
  s_1^* < s_2^* < s_3^*.
\end{equation}
To visualize the solution in the lifted state space, we draw three Hugoniot branches.
The lifted branch over the forward 1-Hugoniot $\mathcal{H}^-_{1,u_l}$ emanates from the open circle at $(u_l,\lambda_1(u_l))$ and reaches $u_1^*$ at height $s_1^*$.
The lifted branch over the forward 2-Hugoniot $\mathcal{H}^-_{2,u_1^*}$ emanates from the open circle at $(u_1^*,\lambda_2(u_1^*))$ and reaches $u_2^*$ at height $s_2^*$.
The lifted branch over the backward 3-Hugoniot $\mathcal{H}^+_{3,u_r}$ emanates from the open circle at $(u_r,\lambda_3(u_r))$ and reaches $u_2^*$ at height $s_3^*$.
At $u_1^*$, the gap between $s_1^*$ (end of the 1-Hugoniot) and $\lambda_2(u_1^*)$ (start of the 2-Hugoniot) reflects the Lax condition $s_1^* < \lambda_2(u_1^*)$, into which the next shock speed $s_2^*$ must fit.
At $u_2^*$, the gap between $s_2^*$ (end of the forward 2-Hugoniot) and $s_3^*$ (end of the backward 3-Hugoniot) is the speed-ordering condition $s_2^* < s_3^*$ from~\eqref{eq:triple-shock-ordering}.
See \cref{fig:triple-shock-example} for an illustration.

\input{figures/fig-triple-shock}

\section{Transversality and the Riemann Problem}
\label{sec:transversality}

\subsection{Preliminaries on Transversality}
\label{ssec:transversality-prelim}

We begin by recalling the notion of transversality for submanifolds, following \cite{Hirsch,Pollack}.

\begin{definition}[Transversality of submanifolds]\label{def:transversality}
Let \(\mathcal{M}\) and \(\mathcal{N}\) be submanifolds of a manifold \(\mathcal{Y}\). We say that \(\mathcal{M}\) and \(\mathcal{N}\) are \textbf{transverse}, written \(\mathcal{M}\pitchfork\mathcal{N}\), if at every point \(x\in\mathcal{M}\cap\mathcal{N}\),
\begin{equation}\label{eq:transversality}
  T_x\mathcal{M} + T_x\mathcal{N} = T_x\mathcal{Y}.
\end{equation}
By convention, if \(\mathcal{M}\cap\mathcal{N}=\emptyset\), then the transversality condition is vacuously satisfied.
\end{definition}

The notion of transversality extends naturally from submanifolds to smooth maps. The following definition is used throughout the genericity arguments in Section~\ref{sec:structural-stability}.

\begin{definition}[Transversality of a map to a submanifold]\label{def:transversality-map}
Let \(f\colon\mathcal{X}\to\mathcal{Y}\) be a \(C^r\) map with \(r\ge 1\), where \(\mathcal{X}\) and \(\mathcal{Y}\) are \(C^r\) manifolds, and let \(\mathcal{Z}\) be a \(C^r\) submanifold of \(\mathcal{Y}\). We say that \(f\) is \textbf{transverse} to \(\mathcal{Z}\), written \(f\pitchfork\mathcal{Z}\), if for every \(a\in f^{-1}(\mathcal{Z})\),
\begin{equation}\label{eq:transversality-map}
  Df_a\bigl(T_a\mathcal{X}\bigr) + T_{f(a)}\mathcal{Z} = T_{f(a)}\mathcal{Y}.
\end{equation}
If \(f^{-1}(\mathcal{Z})=\emptyset\), the condition is vacuously satisfied.
When \(f\) is the inclusion of a submanifold \(\mathcal{M}\hookrightarrow\mathcal{Y}\), taking \(\mathcal{Z}=\mathcal{N}\) recovers \cref{def:transversality}.
\end{definition}

The following lemma provides a reduction of the block-structured Jacobians
that arise in the \(n\times n\) Riemann problem to a single \(n\times n\) determinant condition.

\begin{lemma}[Block lower-bidiagonal reduction]\label{lem:block-reduction}
Let \(n\ge 2\).  Consider an \(n^2\times n^2\) matrix \(M\) partitioned into \(n\) block rows of height \(n\),
\begin{equation}\label{eq:block-M}
  M =
  \begin{pmatrix}
    A_1    &        &        &          & a_1 &     &        &     &  \\[2pt]
    B_2    & A_2    &        &          &     & a_2 &        &     &  \\[2pt]
           & B_3    & \ddots &          &     &     & \ddots &     &  \\[2pt]
           &        & \ddots & A_{n-1}  &     &     &        & a_{n-1}&   \\[2pt]
           &        &        & B_n      &     &     &        &     & a_n
  \end{pmatrix},
\end{equation}
where \(A_1,\ldots,A_{n-1}\in\mathbb{R}^{n\times n}\) and \(B_2,\ldots,B_n\in\mathbb{R}^{n\times n}\)
are the ``state'' blocks, and \(a_1,\ldots,a_n\in\mathbb{R}^{n\times 1}\) are the ``speed'' columns.
The first \((n-1)n\) columns comprise a block lower-bidiagonal matrix with diagonal
entries \(A_1,\ldots,A_{n-1}\) and sub-diagonal entries \(B_2,\ldots,B_n\), while the last \(n\)
columns form \(\diag(a_1,\ldots,a_n)\).

If \(A_1,\ldots,A_{n-1}\) and \(B_n\) are all invertible, then
\begin{equation}\label{eq:det-block-factor}
  \det M = \biggl(\prod_{k=1}^{n-1}\det A_k\biggr)\cdot\det B_n\cdot\det T,
\end{equation}
where \(T\in\mathbb{R}^{n\times n}\) is the \emph{reduced matrix} obtained by forward block elimination of the state unknowns, normalized by \(B_n^{-1}\).
The columns of \(T\) are given explicitly as follows.
Define the \textbf{propagation matrices}
\begin{equation}\label{eq:propagation-matrix}
  P_m := -A_m^{-1}\,B_m, \qquad m=2,\ldots,n-1,
\end{equation}
and the \textbf{seed vectors}
\begin{equation}\label{eq:seed-vector}
  c_k := -A_k^{-1}\,a_k, \qquad k=1,\ldots,n-1.
\end{equation}
Then
\begin{equation}\label{eq:T-column-formula}
  T_k =
  \begin{cases}
    P_{n-1}\cdots P_{k+1}\,c_k & k=1,\ldots,n-2,\\[4pt]
    c_{n-1} & k=n-1,\\[4pt]
    B_n^{-1}a_n & k=n,
  \end{cases}
\end{equation}
In particular, \(M\) is invertible if and only if \(T\) is invertible.
\end{lemma}

\begin{proof}
We reduce \(M\) to block upper-triangular form by a sequence of \(n-1\) block elementary row operations~(EROs).

\textit{Step~1.}
Subtract \(B_2\,A_1^{-1}\) times block row~1 from block row~2.
In the state columns, this replaces \(B_2\) by \(B_2 - B_2\,A_1^{-1}\,A_1 = 0\).
In the speed columns, block row~2 acquires the entry
\(-B_2\,A_1^{-1}\,a_1 = B_2\,c_1\)
in speed column~1 (since \(c_1=-A_1^{-1}\,a_1\)),
while its original entry \(a_2\) in speed column~2 is unchanged.
All other block rows are untouched, so the matrix becomes
\begin{equation}\label{eq:M-after-step1}
  \begin{pmatrix}
    A_1 &        &        &          & a_1     &     &        &        &     \\[2pt]
    0   & A_2    &        &          & B_2c_1  & a_2 &        &        &     \\[2pt]
        & B_3    & \ddots &          &         &     & \ddots &        &     \\[2pt]
        &        & \ddots & A_{n-1}  &         &     &        & a_{n-1}&     \\[2pt]
        &        &        & B_n      &         &     &        &        & a_n
  \end{pmatrix}.
\end{equation}

\textit{Step~2.}
Subtract \(B_3\,A_2^{-1}\) times the (current) block row~2 from block row~3.
In the state columns, \(B_3\) is zeroed out.
In the speed columns, the factor \(-B_3\,A_2^{-1}\) acts on each entry of block row~2:
\begin{alignat*}{2}
  \text{speed column 1:}&\quad -B_3\,A_2^{-1}\cdot B_2\,c_1
    &&= B_3\,\underbrace{(-A_2^{-1}\,B_2)}_{P_2}\,c_1
    = B_3\,P_2\,c_1,\\
  \text{speed column 2:}&\quad -B_3\,A_2^{-1}\cdot a_2
    &&= B_3\,\underbrace{(-A_2^{-1}\,a_2)}_{c_2}
    = B_3\,c_2,
\end{alignat*}
while the original entry \(a_3\) in speed column~3 is unchanged.
The matrix is now
\begin{equation}\label{eq:M-after-step2}
  \begin{pmatrix}
    A_1 &     &        &          & a_1        &        &        &        &     \\[2pt]
    0   & A_2 &        &          & B_2c_1     & a_2    &        &        &     \\[2pt]
    0   & 0   & \ddots &          & B_3P_2c_1  & B_3c_2 & \ddots &        &     \\[2pt]
        &     & \ddots & A_{n-1}  &            &        &        & a_{n-1}&     \\[2pt]
        &     &        & B_n      &            &        &        &        & a_n
  \end{pmatrix}.
\end{equation}

\textit{Continuation.}
At each subsequent step~\(k\) (\(k=1,\ldots,n-1\)), we subtract \(B_{k+1}\,A_k^{-1}\) times block row~\(k\) from block row~\(k+1\).
Since \(M\) is block lower-bidiagonal, this only modifies block row~\(k+1\). The sub-diagonal entry \(B_{k+1}\) is zeroed out in the state columns, and the speed entries of block row~\(k\) are propagated down with the factor \(-B_{k+1}\,A_k^{-1}\).
At each step with $k\ge 2$, the propagation factor \(-A_k^{-1}\,B_k = P_k\) compounds the entries inherited from block row~$k$, while \(-A_k^{-1}\,a_k = c_k\) creates a new entry from the original speed column. At $k=1$ block row~1 carries only $a_1$, so nothing is inherited.
After all \(n-1\) steps, every sub-diagonal block \(B_2,\ldots,B_n\) has been eliminated, and the matrix takes the form
\begin{equation}\label{eq:M-eliminated}
  \begin{pmatrix}
    A_1 &        &        &          & a_1   & 0     & \cdots & 0      & 0   \\[2pt]
    0   & A_2    &        &          & *     & a_2   & \cdots & 0      & 0   \\[2pt]
        &        & \ddots &          & \vdots& \vdots& \ddots & \vdots & \vdots\\[2pt]
        &        &        & A_{n-1}  & *     & *     & \cdots & a_{n-1}& 0   \\[2pt]
    0   & 0      & \cdots & 0        & S_1   & S_2   & \cdots & S_{n-1}& S_n
  \end{pmatrix},
\end{equation}
where the \(*\) entries in block rows \(1,\ldots,n-1\) arise from the propagation but do not affect the determinant.
The bottom block row is entirely determined by the accumulated propagation, giving the speed-column entries
\begin{gather*}
  S_1 = B_n\,P_{n-1}\cdots P_2\,c_1,\quad
  S_2 = B_n\,P_{n-1}\cdots P_3\,c_2,\quad
  \ldots,\\
  S_{n-1} = B_n\,c_{n-1},\quad
  S_n = a_n,
\end{gather*}
that is, \(S_k = B_n\,T_k\) with \(T_k\) as in the column formula~\eqref{eq:T-column-formula}, the last column being \(a_n = B_n\ml B_n^{-1}a_n\mr\).

\textit{Determinant.}
Each block ERO at Step~\(k\) amounts to left-multiplying \(M\) by the block matrix \(E_k\) that equals the identity except for the off-diagonal block \(-B_{k+1}A_k^{-1}\) in position \((k+1,k)\).
Since \(E_k\) is block lower-triangular with identity blocks on the diagonal, \(\det E_k = 1\), so the ERO preserves the determinant.
(Note that only the invertibility of \(A_k\) is required to form \(E_k\); no assumption on \(B_{k+1}\) is needed.) After all \(n-1\) EROs, the matrix \(E_{n-1}\cdots E_1\,M\) is block upper-triangular with diagonal blocks \(A_1,\ldots,A_{n-1},B_n\,T\).
The determinant of a block-triangular matrix equals the product of the determinants of its diagonal blocks, so
\[
  \det M = \det(E_{n-1}\cdots E_1\,M) = \biggl(\prod_{k=1}^{n-1}\det A_k\biggr)\cdot\det B_n\cdot\det T,
\]
and, since \(\det B_n\neq 0\), \(\det M\neq 0\) if and only if \(\det T\neq 0\).
\end{proof}

In the Riemann problem, the domain of the objective map depends on the left and right states \((u_l,u_r)\), so the standard parametric transversality theorem (which assumes a fixed domain) does not directly apply. The following variant accommodates parameter-dependent domains and is used in the proof of the genericity theorem in Section~\ref{sec:structural-stability}.

\begin{theorem}[Foliated parametric transversality]\label{thm:foliated-parametric-transversality}
Let \(\mathcal{P}\) and \(\mathcal{Y}\) be \(C^r\) manifolds, and let \(\mathcal{Z}\) be a \(C^r\) submanifold of \(\mathcal{Y}\).
Suppose that for each \(p\in\mathcal{P}\) we are given a \(C^r\) manifold \(\mathcal{X}_p\) of common dimension \(\dim\mathcal{X}\), and define the foliated set
\[
  \mathcal{XP} := \bigcup_{p\in\mathcal{P}} \mathcal{X}_p \times \{p\}.
\]
Consider a map \(\mathbf{F}\colon\mathcal{XP}\to\mathcal{Y}\) and, for each parameter $p\in\mathcal{P}$, the associated map \(\mathbf{F}_p\colon\mathcal{X}_p\to\mathcal{Y}\) given by $\mathbf{F}_p(x):=\mathbf{F}(x,p)$.
Suppose that
\begin{enumerate}[leftmargin=*]
\item \(r > \max\{0,\,\dim\mathcal{X}+\dim\mathcal{Z}-\dim\mathcal{Y}\}\),
\item \(\mathcal{XP}\) is a \(C^r\) manifold with \(\dim\mathcal{XP}=\dim\mathcal{X}+\dim\mathcal{P}\),
\item \(T_{(x,p)}\mathcal{XP}\cong T_x\mathcal{X}_p\times T_p\mathcal{P}\) for each \((x,p)\in\mathcal{XP}\),
\item the map \((x,p)\mapsto \mathbf{F}_p(x)\) is \(C^r\), and
\item \(\mathbf{F}\pitchfork\mathcal{Z}\).
\end{enumerate}
Then, for almost every \(p\in\mathcal{P}\),\footnote{Here, ``almost every'' is understood in the sense that the set of \(p\) such that \(\mathbf{F}_p\pitchfork\mathcal{Z}\) fails is of measure zero in \(\mathcal{P}\).} \(\mathbf{F}_p\pitchfork\mathcal{Z}\).
\end{theorem}
The proof mimics that of the standard parametric transversality theorem, using assumptions~(2) and~(3) to replace arguments involving the Cartesian product structure \(\mathcal{X}\times\mathcal{P}\) with the foliated set \(\mathcal{XP}\); see \cite{TanBertozzi2x2} for full details.

\subsection{Transversality and the Riemann problem for \texorpdfstring{$2\times 2$}{2x2} systems}
\label{ssec:transversality-2x2}

For a \(2\times 2\) system, the Riemann problem reduces to finding a single intermediate state \(u_1^*\) in the intersection of the forward 1-wave curve from \(u_l\) and the backward 2-wave curve from \(u_r\) (see \cref{rem:backward-curves}).
A solution persists under small perturbations of \(u_l\) and \(u_r\) whenever this intersection is \emph{transverse}, meaning that the tangent vectors to the two projected wave curves at \(u_1^*\) span \(T_{u_1^*}U\cong\mathbb{R}^2\).
Our prior paper \cite{TanBertozzi2x2} established structural stability for \(2\times 2\) systems by verifying this geometric condition directly in the projected state space.
In this section, we reformulate the transversality condition in the lifted-space framework developed in Sections~\ref{ssec:hugoniot} and~\ref{ssec:rarefaction}, and show that it recovers the classical projected-space condition as a consequence.

We illustrate the argument with the shock--rarefaction configuration below. Recall that the Riemann solution
\[
  u_l \xrightarrow{\;s_1^*\;} u_1^*
       \xrightarrow{\;\xi\in[\lambda_2(u_1^*),\,\lambda_2(u_r)]\;} u_r
\]
consists of a 1-shock connecting \(u_l\) to \(u_1^*\) at speed \(s_1^*\), followed by a 2-rarefaction connecting \(u_1^*\) to \(u_r\).
In the lifted state space (\cref{fig:shock-rare-example}, top-right panel), the staircase trajectory \textcircled{\raisebox{-0.5pt}{\scriptsize 1}}--\textcircled{\raisebox{-0.5pt}{\scriptsize 5}} encodes the full solution.
The ``active'' steps that determine \(u_1^*\) are \textcircled{\raisebox{-0.5pt}{\scriptsize 2}}--\textcircled{\raisebox{-0.5pt}{\scriptsize 4}}, which respectively trace the shock jump along the lifted 1-Hugoniot, the speed gap from \(s_1^*\) to \(\lambda_2(u_1^*)\) at the intermediate state, and the 2-rarefaction curve.

\subsubsection{The objective map}

The conditions that define \(u_1^*\) can be written as a single system of equations in the lifted variables \((s_1^*,u_1^*,\lambda_2^*)\in\mathbb{R}\times (U\setminus \{u_l,u_r\}) \times\mathbb{R}\), where the exclusion \(u_1^*\neq u_l\) and \(u_1^*\neq u_r\) ensures that each wave is non-degenerate. The 1-shock condition requires that \(u_1^*\) lie on the Hugoniot locus of \(u_l\) at speed \(s_1^*\), and the 2-rarefaction condition requires that \(u_1^*\) lie on the backward 2-rarefaction curve from \(u_r\) at eigenvalue parameter \(\lambda_2^*=\lambda_2(u_1^*)\).
These correspond to the zero set of the combined \emph{objective map}
\begin{equation}\label{eq:objective-2x2}
  J(s_1^*,u_1^*,\lambda_2^*)
  :=
  \begin{pmatrix}
    H(u_1^*,s_1^*;\,u_l)\\[4pt]
    R_2(u_1^*,\lambda_2^*;\,u_r)
  \end{pmatrix}
  \in\mathbb{R}^{2n} = \mathbb{R}^4
\end{equation}
where \(H\) is the Rankine--Hugoniot map \eqref{eq:RH-map} and \(R_2\) is the rarefaction map (\cref{def:rarefaction-map}).
Since \(n=2\), the unknown \((s_1^*,u_1^*,\lambda_2^*)\in\mathbb{R}\times\mathbb{R}^2\times\mathbb{R}=\mathbb{R}^4\) has the same dimension as the target.
The Riemann solution corresponds to a zero of \(J\), and by the implicit function theorem, structural stability holds whenever the Jacobian
\begin{equation}\label{eq:jacobian-2x2}
  D_{(s_1^*,\,u_1^*,\,\lambda_2^*)}J
\end{equation}
is invertible at the solution. This Jacobian inherits a block structure from the two components of \(J\). To compute this, we first differentiate the Hugoniot component \(H(u_1^*,s_1^*;\,u_l)\) with respect to \((s_1^*,u_1^*,\lambda_2^*)\), and using the Jacobian formula~\eqref{eq:jacobian}, this yields the first row of blocks
\[
  D_{(s_1^*,\,u_1^*,\,\lambda_2^*)}H
  =
  \bigl[\;
    -(u_1^*-u_l)
    \;\bigm|\;
    DF(u_1^*)-s_1^*I_n
    \;\bigm|\;
    0
  \;\bigr],
\]
where the zero block appears because \(H\) does not depend on \(\lambda_2^*\).
For the rarefaction component \(R_2(u_1^*,\lambda_2^*;\,u_r)=u_1^*-\Gamma_2(\lambda_2^*;\,u_r)\), the Jacobian formula~\eqref{eq:rarefaction-jacobian} gives
\[
  D_{(s_1^*,\,u_1^*,\,\lambda_2^*)}R_2
  =
  \bigl[\;
    0
    \;\bigm|\;
    I_n
    \;\bigm|\;
    -\Gamma_2'(\lambda_2^*;\,u_r)
  \;\bigr],
\]
where \(\Gamma_2'(\lambda_2^*;\,u_r)\) denotes the derivative of \(\Gamma_2(\cdot\,;u_r)\) with respect to the eigenvalue parameter. By the rarefaction ODE~\eqref{eq:Gamma-xi-ODE}, this tangent vector is a nonzero scalar multiple of the right eigenvector \(r_2(u_1^*)\), so we write \(\Gamma_2'(\lambda_2^*;\,u_r)=\alpha_2\,r_2(u_1^*)\) with \(\alpha_2\neq 0\).
Assembling the two rows, the full Jacobian is given by
\begin{equation}\label{eq:jacobian-2x2-block}
  D_{(s_1^*,\,u_1^*,\,\lambda_2^*)}J
  =
  \begin{pmatrix}
    -(u_1^*-u_l) & DF(u_1^*)-s_1^*I_n & 0\\[4pt]
    0 & I_n & -\alpha_2\,r_2(u_1^*)
  \end{pmatrix},
\end{equation}
in which the block sizes are \((u_1^*-u_l)\in\mathbb{R}^{2\times 1}\), \(DF(u_1^*)-s_1^*I_n\in\mathbb{R}^{2\times 2}\), \(I_n\in\mathbb{R}^{2\times 2}\), and \(\alpha_2\,r_2(u_1^*)\in\mathbb{R}^{2\times 1}\).

\subsubsection{Transversality of the wave curves}

We now reduce the invertibility of the \(4\times 4\) matrix~\eqref{eq:jacobian-2x2-block} to a condition on the projected wave curves.
The matrix \(DF(u_1^*)-s_1^*I_n\) is invertible because the Lax entropy condition~\eqref{eq:lax-entropy} ensures that \(s_1^*\) is not an eigenvalue of \(DF(u_1^*)\).
Subtracting $(DF(u_1^*)-s_1^*I_n)^{-1}$ times the first row of blocks from the second row is the elimination step of \cref{lem:block-reduction}, with the identity block in the role of the sub-diagonal block $B_2$. This row operation preserves the determinant and zeroes out the identity block, while the zero block in the first column of the second row acquires the entry $-(DF(u_1^*)-s_1^*I_n)^{-1}\cdot\bigl(-(u_1^*-u_l)\bigr) = (DF(u_1^*)-s_1^*I_n)^{-1}(u_1^*-u_l)$, so that the Jacobian~\eqref{eq:jacobian-2x2-block} becomes
\[
  \begin{pmatrix}
    -(u_1^*-u_l) & DF(u_1^*)-s_1^*I_n & 0\\[4pt]
    \bigl(DF(u_1^*)-s_1^*I_n\bigr)^{-1}(u_1^*-u_l) & 0 & -\alpha_2\,r_2(u_1^*)
  \end{pmatrix}.
\]
Moving the two middle columns in front of the first is an even permutation of the columns (two transpositions), which leaves the determinant unchanged and yields a block upper-triangular matrix with $DF(u_1^*)-s_1^*I_n$ in the first diagonal block. The determinant of~\eqref{eq:jacobian-2x2-block} therefore equals $\det(DF(u_1^*)-s_1^*I_n)\cdot\det T$, where $T$ is the remaining diagonal block,
\begin{equation}\label{eq:transversality-det-2x2}
  T :=
  \Bigl[\;
    \bigl(DF(u_1^*)-s_1^*I_n\bigr)^{-1}(u_1^*-u_l)
    \;\Bigm|\;
    -\alpha_2\,r_2(u_1^*)
  \;\Bigr]
  \in\mathbb{R}^{2\times 2}.
\end{equation}
The Jacobian~\eqref{eq:jacobian-2x2-block} is therefore invertible if and only if
\begin{equation}\label{eq:transversality-condition-2x2}
  \det T \neq 0.
\end{equation}

The two columns of $T$ span the tangent lines to the projected wave curves at $u_1^*$.
For the first column, recall from~\eqref{eq:jacobian} that a tangent vector \((\dot{u},\dot{s})\) to the lifted Hugoniot curve at \((u_1^*,s_1^*)\) satisfies
\[
  \bigl[\,DF(u_1^*)-s_1^*I_n\;\bigm|\;-(u_1^*-u_l)\,\bigr]
  \begin{pmatrix}\dot{u}\\ \dot{s}\end{pmatrix}
  = 0,
\]
which gives \(\dot{u}=\bigl(DF(u_1^*)-s_1^*I_n\bigr)^{-1}(u_1^*-u_l)\,\dot{s}\).
The projection \(\pi(u,s)=u\) sends this to \(\dot{u}\), so the first column of \(T\) spans the tangent line to the projected Hugoniot locus \(\mathcal{H}_{u_l}\) from \(u_l\) at \(u_1^*\).
In the staircase picture of \cref{fig:shock-rare-example}, this is the projection of step~\textcircled{\raisebox{-0.5pt}{\scriptsize 2}} into the state space.
The second column $-\alpha_2\,r_2(u_1^*)$, a nonzero multiple of $r_2(u_1^*)$, spans the tangent line to the projected 2-rarefaction curve \(\mathcal{R}_{2,u_r}\) at \(u_1^*\), corresponding to the projection of step~\textcircled{\raisebox{-0.5pt}{\scriptsize 4}}.

Since each column of \(T\) spans the respective tangent space \(T_{u_1^*}\mathcal{W}^f_{1,u_l}\) and \(T_{u_1^*}\mathcal{W}^b_{2,u_r}\) (both one-dimensional), condition~\eqref{eq:transversality-condition-2x2} is equivalent to
\begin{equation}\label{eq:wave-curve-transversality-2x2}
  T_{u_1^*}\mathcal{W}^f_{1,u_l} + T_{u_1^*}\mathcal{W}^b_{2,u_r} = T_{u_1^*}U \cong \mathbb{R}^2,
\end{equation}
which, by \cref{def:transversality}, is the transversality condition \(\mathcal{W}^f_{1,u_l}\pitchfork\mathcal{W}^b_{2,u_r}\) at \(u_1^*\).
We have therefore shown that invertibility of the objective-map Jacobian~\eqref{eq:jacobian-2x2} is equivalent to transversality of the forward and backward wave curves at the intermediate state, and hence structural stability of the Riemann solution follows from the implicit function theorem whenever this transversality holds.

\begin{remark}[Other wave-type combinations]\label{rem:other-wave-types}
The shock--rarefaction case was chosen for concreteness, but the same framework applies to all four wave-type combinations (double shock, double rarefaction, shock--rarefaction, rarefaction--shock).
In each case, one replaces the appropriate component of the objective map~\eqref{eq:objective-2x2} with the Hugoniot or rarefaction map, and the resulting Jacobian has an analogous block structure.
The reduction to a \(2\times 2\) transversality determinant proceeds identically, with the columns of \(T\) given by the projected tangent vectors of whichever wave curves are involved.
The general treatment of mixed wave-type configurations, including the explicit rarefaction blocks, is given in Section~\ref{ssec:transversality-nxn}.
\end{remark}

\subsection{Sequential transversality and the Riemann problem for \texorpdfstring{$3\times 3$}{3x3} systems}
\label{ssec:transversality-3x3}

In the \(2\times 2\) case studied in Section~\ref{ssec:transversality-2x2}, structural stability reduces to transversality of two one-dimensional wave curves in \(\mathbb{R}^2\), whose codimensions sum to $1+1=2=\dim\mathbb{R}^2$, so that two such curves generically meet in isolated points.
For a \(3\times 3\) system, however, the analogous wave curves are still one-dimensional, but the ambient state space is \(\mathbb{R}^3\).
Two one-dimensional curves in \(\mathbb{R}^3\) generically miss each other entirely, since their codimensions sum to \(2+2=4>3\).
In particular, the pairwise intersection-based viewpoint of Section~\ref{ssec:transversality-2x2} does not extend directly.

A new approach is required, one that accounts for all three waves simultaneously.
The key idea is to chain the waves sequentially, solving forward from \(u_l\) through the intermediate states \(u_1^*,u_2^*\) to \(u_r\), and then to transport all tangent contributions to a common reference point via pushforward maps.
We develop this approach using the triple-shock configuration of Section~\ref{sssec:triple-shock} as a concrete illustration.
The resulting \emph{sequential transversality} condition replaces the single intersection test of the \(2\times 2\) theory with a linear independence condition on the pushforward-transported tangent vectors.

\subsubsection{The objective map}

Recall from Section~\ref{sssec:triple-shock} that the triple-shock Riemann solution
\[
  u_l \xrightarrow{\;s_1^*\;} u_1^*
       \xrightarrow{\;s_2^*\;} u_2^*
       \xrightarrow{\;s_3^*\;} u_r
\]
is determined by the unknowns \((s_1^*,u_1^*,s_2^*,u_2^*,s_3^*)\in\mathbb{R}\times U\times\mathbb{R}\times U\times\mathbb{R}\),
where we assume each wave is non-degenerate, i.e., \(u_l\neq u_1^*\neq u_2^*\neq u_r\).
With \(n=3\), the total number of scalar unknowns is \(1+3+1+3+1=9\).
The three Rankine--Hugoniot conditions \(H(u_1^*,s_1^*;\,u_l)=0\), \(H(u_2^*,s_2^*;\,u_1^*)=0\), and \(H(u_2^*,s_3^*;\,u_r)=0\) comprise \(3n=9\) scalar equations, so that the system is square.
The first two conditions are anchored at the known previous state, following the solution forward from $u_l$; the third is anchored at the right state, matching the backward 3-Hugoniot $\mathcal{H}^+_{3,u_r}$ of Section~\ref{sssec:triple-shock}, which branches out from $u_r$ and meets the forward sweep at $u_2^*$. This two-sided anchoring, with the first row launched from $u_l$ and the last row from $u_r$, is already the convention of the $2\times 2$ objective~\eqref{eq:objective-2x2}, whose second component is the backward rarefaction condition anchored at $u_r$. By the antisymmetry $H(u_2^*,s_3^*;\,u_r)=-H(u_r,s_3^*;\,u_2^*)$ of the Rankine--Hugoniot map in its two states, this choice has the same zero set as the forward-anchored condition; \cref{rem:backward-anchoring} explains why it is nevertheless the right one.
These conditions are captured by the \emph{objective map}
\begin{equation}\label{eq:objective-3x3}
  J(s_1^*,u_1^*,s_2^*,u_2^*,s_3^*)
  :=
  \begin{pmatrix}
    H(u_1^*,s_1^*;\,u_l)\\[4pt]
    H(u_2^*,s_2^*;\,u_1^*)\\[4pt]
    H(u_2^*,s_3^*;\,u_r)
  \end{pmatrix}
  \in\mathbb{R}^{9}.
\end{equation}
A solution of the Riemann problem corresponds to a zero of \(J\), and by the implicit function theorem, structural stability holds whenever the Jacobian
\begin{equation}\label{eq:jacobian-3x3}
  D_{(s_1^*,\,u_1^*,\,s_2^*,\,u_2^*,\,s_3^*)}J
\end{equation}
is invertible at the solution.

To expose the block structure, we reorder the unknowns as \((u_1^*,u_2^*,s_1^*,s_2^*,s_3^*)\), placing the \(n\)-dimensional intermediate states first and the one-dimensional speeds second.
Using the Jacobian formula~\eqref{eq:jacobian} to differentiate each row of~\eqref{eq:objective-3x3}, we obtain
\begin{equation}\label{eq:jacobian-3x3-block}
\begin{aligned}
  &\; DJ(s_1^*, u_1^*, s_2^*, u_2^*, s_3^*) = \\
  &\begin{pmatrix}
       DF(u_1^*)-s_1^*I_n & 0 & -(u_1^*-u_l) & 0 & 0 \\[4pt]
       -(DF(u_1^*)-s_2^*I_n) & DF(u_2^*)-s_2^*I_n & 0 & -(u_2^*-u_1^*) & 0 \\[4pt]
       0 & DF(u_2^*)-s_3^*I_n & 0 & 0 & u_r-u_2^*
     \end{pmatrix}.
\end{aligned}
\end{equation}
Each \(DF-sI_n\) block is \(3\times 3\) and each state-difference column is \(3\times 1\), giving a \(9\times 9\) matrix overall.
This is an instance of the block lower-bidiagonal structure of \cref{lem:block-reduction} with \(n=3\), where
\begin{alignat*}{2}
  A_1 &= DF(u_1^*)-s_1^*I_n, &\qquad
  A_2 &= DF(u_2^*)-s_2^*I_n, \\
  B_2 &= -(DF(u_1^*)-s_2^*I_n), &\qquad
  B_3 &= DF(u_2^*)-s_3^*I_n,
\end{alignat*}
and the speed columns are \(a_k = -(u^{(k)}-u^{(k-1)})\) for $k=1,2$, while the backward-anchored third row contributes $a_3 = u_r-u_2^*$.

The diagonal blocks \(A_1\) and \(A_2\) are invertible because the Lax entropy condition~\eqref{eq:lax-entropy} ensures that \(s_k^*\) is not an eigenvalue of \(DF\) at the intermediate state \(u_k^*\).
We claim that the off-diagonal blocks \(B_2\) and \(B_3\) are likewise invertible.
For \(B_2 = -(DF(u_1^*)-s_2^*I_n)\), it suffices to show that \(s_2^*\) is not an eigenvalue of \(DF(u_1^*)\).
The Lax condition for the 1-shock gives \(s_1^*>\lambda_1(u_1^*)\), and for the 2-shock gives \(\lambda_2(u_1^*)>s_2^*\).
Combined with the speed ordering \(s_2^*>s_1^*\) from~\eqref{eq:triple-shock-ordering}, this yields
\[
  \lambda_1(u_1^*) < s_1^* < s_2^* < \lambda_2(u_1^*).
\]
By strict hyperbolicity, \(\lambda_1(u_1^*)<\lambda_2(u_1^*)<\lambda_3(u_1^*)\), so \(s_2^*\) lies strictly between consecutive eigenvalues at \(u_1^*\) and hence avoids all three, making \(DF(u_1^*)-s_2^*I_n\) nonsingular.
The same argument applies to \(B_3 = DF(u_2^*)-s_3^*I_n\), with the Lax conditions placing \(s_3^*\) strictly between \(\lambda_2(u_2^*)\) and \(\lambda_3(u_2^*)\); the sign convention of $B_3$ differs from that of $B_2$ only because the third row is anchored backward.

By \cref{lem:block-reduction}, the determinant factors as \(\det A_1\cdot\det A_2\cdot\det B_3\cdot\det T\), where the transversality matrix \(T\in\mathbb{R}^{3\times 3}\) has columns given by the formula~\eqref{eq:T-column-formula}.
The single propagation matrix is given by
\begin{equation}\label{eq:pushforward-hugoniot}
  P_2 = -A_2^{-1}B_2 = \bigl(DF(u_2^*)-s_2^*I_n\bigr)^{-1}\bigl(DF(u_1^*)-s_2^*I_n\bigr) =: \Phi_{12},
\end{equation}
which we call the \textbf{pushforward} from \(u_1^*\) to \(u_2^*\) across the 2-shock.
The seed vectors are \(c_k = -A_k^{-1}a_k = A_k^{-1}(u^{(k)}-u^{(k-1)})\).
Applying the column formula~\eqref{eq:T-column-formula} gives
\begin{equation}\label{eq:transversality-normalized-3x3}
  T =
  \Bigl[\;
    \Phi_{12}\,A_1^{-1}(u_1^*-u_l)
    \;\Bigm|\;
    A_2^{-1}(u_2^*-u_1^*)
    \;\Bigm|\;
    (DF(u_2^*)-s_3^*I_n)^{-1}(u_r-u_2^*)
  \;\Bigr].
\end{equation}
Invertibility of the full Jacobian is therefore equivalent to the \emph{transversality condition}
\begin{equation}\label{eq:transversality-condition-3x3}
  \det T \neq 0.
\end{equation}

\subsubsection{Sequential transversality of wave curves}

We now give geometric meaning to the columns of \(T\) and to the pushforward map \(\Phi_{12}\).

\paragraph{Geometric identification of the columns}
All three columns of \(T\) in~\eqref{eq:transversality-normalized-3x3} are tangent vectors at the common base point \(u_2^*\). To deduce this, we argue as follows:
\begin{itemize}[leftmargin=*]
\item The first column is \(\Phi_{12}\,A_1^{-1}(u_1^*-u_l)\).
The vector \(v_1 := A_1^{-1}(u_1^*-u_l) = \bigl(DF(u_1^*)-s_1^*I_n\bigr)^{-1}(u_1^*-u_l)\) is the tangent to the projected Hugoniot locus from \(u_l\) at \(u_1^*\) (the projection of step~\textcircled{\raisebox{-0.5pt}{\scriptsize 2}} in \cref{fig:triple-shock-example}), as in Section~\ref{ssec:transversality-2x2}.
The pushforward \(\Phi_{12}\) transports this tangent from \(u_1^*\) to \(u_2^*\).
\item The second column is \(v_2 := A_2^{-1}(u_2^*-u_1^*) = \bigl(DF(u_2^*)-s_2^*I_n\bigr)^{-1}(u_2^*-u_1^*)\), the tangent to the projected Hugoniot locus from \(u_1^*\) at \(u_2^*\) (the projection of step~\textcircled{\raisebox{-0.5pt}{\scriptsize 4}}).
\item The third column is \(v_3 := \bigl(DF(u_2^*)-s_3^*I_n\bigr)^{-1}(u_r-u_2^*) = B_3^{-1}a_3\). Since the third wave condition is anchored at $u_r$, the Jacobian formula~\eqref{eq:jacobian} shows that a tangent vector $(\dot{u}_2,\dot{s}_3)$ to the lifted Hugoniot curve $\widehat{\mathcal{H}}_{u_r}$ at $(u_2^*,s_3^*)$ satisfies $B_3\,\dot{u}_2 + a_3\,\dot{s}_3 = 0$, which gives $\dot{u}_2 = -v_3\,\dot{s}_3$. Projecting to state space, $v_3$ spans the tangent line at $u_2^*$ to the projected backward 3-Hugoniot $\mathcal{H}^+_{3,u_r}$ (the projection of step~\textcircled{\raisebox{-0.5pt}{\scriptsize 6}} in \cref{fig:triple-shock-example}; see also \cref{rem:backward-anchoring}).
\end{itemize}

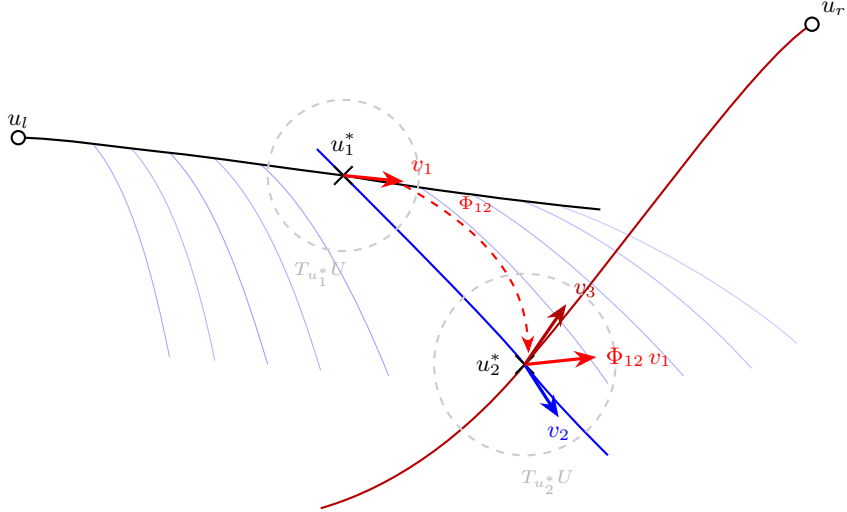
\begin{figure}[htbp]
  \centering
  \begin{tikzpicture}[>=Stealth, scale=1.0,
      every node/.style={font=\small}]


    \coordinate (ul) at (-5.5,1.5);
    \coordinate (u1) at (-1.2,1.0);
    \coordinate (u2) at (1.2,-1.5);
    \coordinate (ur) at (5.0,3.0);

    \coordinate (sp1) at (-4.5,1.4);
    \coordinate (sp2) at (-3.5,1.3);
    \coordinate (sp3) at (-2.3,1.15);
    \coordinate (sp5) at (-0.2,0.85);
    \coordinate (sp6) at (0.5,0.75);
    \coordinate (sp7) at (1.2,0.65);


    \draw[blue!35]
      (sp1) .. controls ($(sp1)+(0.5,-0.6)$) and ($(sp1)+(0.8,-1.8)$) .. ($(sp1)+(1.0,-2.8)$);

    \coordinate (spA) at (-4.0,1.35);
    \draw[blue!30]
      (spA) .. controls ($(spA)+(0.55,-0.6)$) and ($(spA)+(0.9,-1.8)$) .. ($(spA)+(1.1,-2.8)$);

    \draw[blue!38]
      (sp2) .. controls ($(sp2)+(0.6,-0.6)$) and ($(sp2)+(1.0,-1.8)$) .. ($(sp2)+(1.3,-2.8)$);

    \coordinate (spB) at (-2.9,1.22);
    \draw[blue!32]
      (spB) .. controls ($(spB)+(0.65,-0.6)$) and ($(spB)+(1.1,-1.8)$) .. ($(spB)+(1.4,-2.8)$);

    \draw[blue!38]
      (sp3) .. controls ($(sp3)+(0.7,-0.6)$) and ($(sp3)+(1.3,-1.8)$) .. ($(sp3)+(1.7,-2.8)$);


    \draw[blue!35]
      (sp5) .. controls ($(sp5)+(0.8,-0.5)$) and ($(sp5)+(1.8,-1.6)$) .. ($(sp5)+(2.5,-2.6)$);

    \draw[blue!32]
      (sp6) .. controls ($(sp6)+(0.9,-0.5)$) and ($(sp6)+(2.0,-1.5)$) .. ($(sp6)+(2.8,-2.4)$);

    \draw[blue!28]
      (sp7) .. controls ($(sp7)+(1.0,-0.4)$) and ($(sp7)+(2.2,-1.3)$) .. ($(sp7)+(3.0,-2.2)$);

    \coordinate (spC) at (1.7,0.58);
    \draw[blue!22]
      (spC) .. controls ($(spC)+(1.0,-0.35)$) and ($(spC)+(2.3,-1.1)$) .. ($(spC)+(3.1,-1.8)$);


    \draw[thick]
      (ul) .. controls (-5.0,1.5) and (-4.0,1.35) .. (sp2)
           .. controls (-3.0,1.25) and (-2.0,1.1) .. (u1)
           .. controls (-0.6,0.9) and (0.2,0.78) .. (sp6)
           .. controls (1.0,0.68) and (1.5,0.62) .. (2.2,0.55);

    \draw[thick, blue]
      ($(u1)+(-0.35,0.35)$)
      .. controls ($(u1)+(-0.15,0.15)$) and ($(u1)+(-0.0,0.0)$) .. (u1)
      .. controls ($(u1)+(0.6,-0.6)$) and ($(u2)+(-0.4,0.5)$) .. (u2)
      .. controls ($(u2)+(0.4,-0.5)$) and ($(u2)+(0.8,-0.9)$) .. ($(u2)+(1.1,-1.2)$);

    \draw[thick, red!70!black]
      (ur) .. controls (4.2,2.5) and (2.5,0.0) .. (u2)
           .. controls (0.2,-2.7) and (-0.8,-3.2) .. (-1.5,-3.4);


    \draw[fill=white, thick] (ul) circle (2.5pt);
    \node[above] at (ul) {$u_l$};

    \draw[thick] (u1) +(-3.5pt,-3.5pt) -- +(3.5pt,3.5pt)
                 (u1) +(-3.5pt,3.5pt) -- +(3.5pt,-3.5pt);
    \node[above] at ($(u1)+(0.0,0.15)$) {$u_1^*$};

    \draw[thick] (u2) +(-3.5pt,-3.5pt) -- +(3.5pt,3.5pt)
                 (u2) +(-3.5pt,3.5pt) -- +(3.5pt,-3.5pt);
    \node[left] at ($(u2)+(-0.18,0.0)$) {$u_2^*$};

    \draw[fill=white, thick] (ur) circle (2.5pt);
    \node[above right] at (ur) {$u_r$};


    \draw[gray!40, thick, dashed] (u1) circle (1.0cm);
    \node[gray!60, font=\scriptsize] at ($(u1)+(-0.3,-1.3)$)
      {$T_{u_1^*}U$};

    \coordinate (v1end) at ($(u1)+(0.8,-0.08)$);
    \draw[->, red, very thick] (u1) -- (v1end);
    \node[above right=-1pt, red, font=\small] at (v1end) {$v_1$};

    \draw[gray!40, thick, dashed] (u2) circle (1.2cm);
    \node[gray!60, font=\scriptsize] at ($(u2)+(0.3,-1.55)$)
      {$T_{u_2^*}U$};

    \coordinate (v2end) at ($(u2)+(0.45,-0.7)$);
    \draw[->, blue, very thick] (u2) -- (v2end);
    \node[below, blue, font=\small] at ($(v2end)+(0.0,-0.0)$) {$v_2$};

    \coordinate (pv1end) at ($(u2)+(0.95,0.1)$);
    \draw[->, red, very thick] (u2) -- (pv1end);
    \node[right, red, font=\small] at ($(pv1end)+(0.0,0.0)$) {$\Phi_{12}\,v_1$};

    \coordinate (v3end) at ($(u2)+(0.55,0.8)$);
    \draw[->, red!70!black, very thick] (u2) -- (v3end);
    \node[above right=-1pt, red!70!black, font=\small] at (v3end) {$v_3$};


    \draw[->, thick, dashed, red]
      ($(v1end)+(0.0,-0.05)$) .. controls ($(v1end)+(0.8,-0.5)$) and ($(u2)+(0.0,1.3)$) .. ($(u2)+(0.05,0.15)$);
    \node[right, red, font=\scriptsize] at ($(v1end)+(0.6,-0.3)$) {$\Phi_{12}$};

  \end{tikzpicture}
  \caption{Sequential transversality for a \(3\times 3\) triple-shock Riemann solution.
  From each point along the 1-Hugoniot locus (black spine from \(u_l\)), a 2-Hugoniot curve (blue) branches out, forming a two-dimensional surface in state space (light blue ribs).
  The backward 3-Hugoniot from \(u_r\) (dark red) intersects this surface transversely at \(u_2^*\).
  The tangent vector \(v_1\) (red) at \(u_1^*\) is transported to \(u_2^*\) via the pushforward \(\Phi_{12}\) (dashed red arrow), yielding \(\Phi_{12}\,v_1\).
  Together with \(v_2\) (blue, tangent to the 2-Hugoniot) and \(v_3\) (dark red, tangent to the 3-Hugoniot), the sequential transversality condition~\eqref{eq:sequential-transversality-3x3} requires these three vectors to span \(T_{u_2^*}U\cong\mathbb{R}^3\).}
  \label{fig:sequential-transversality}
\end{figure}

\paragraph{The pushforward map via the implicit function theorem}
The map \(\Phi_{12}\) has a natural interpretation through the implicit function theorem as follows.
The 2-Hugoniot condition \(H(u_2,s_2;\,u_1)=0\), evaluated at the fixed speed \(s_2=s_2^*\), defines a relation between \(u_1\) and \(u_2\). Since the partial derivative \(\p_{u_2}H = DF(u_2^*)-s_2^*I_n\) is invertible, the implicit function theorem provides a \(C^1\) map \(\Phi\) defined near \(u_1^*\) satisfying \(u_2=\Phi(u_1)\) and \(\Phi(u_1^*)=u_2^*\).
Differentiating the identity \(H(\Phi(u_1),s_2^*;\,u_1)\equiv 0\) with respect to \(u_1\) and using \(\p_{u_1}H = -(DF(u_1^*)-s_2^*I_n)\) gives
\begin{equation}\label{eq:pushforward-derivative}
  D\Phi(u_1^*)
  = -\bigl(\p_{u_2}H\bigr)^{-1}\p_{u_1}H
  = \bigl(DF(u_2^*)-s_2^*I_n\bigr)^{-1}\bigl(DF(u_1^*)-s_2^*I_n\bigr)
  = \Phi_{12}.
\end{equation}
The pushforward \(\Phi_{12}\) is therefore the derivative at \(u_1^*\) of the map that sends a left state \(u_1\) near \(u_1^*\) to the right state \(u_2\) joined to it by a jump of the fixed speed \(s_2^*\).
As a linear isomorphism \(T_{u_1^*}U\to T_{u_2^*}U\), it transports tangent vectors from the first intermediate state to the second.

\paragraph{The sequential transversality condition.}
All three columns of \(T\) live in the tangent space \(T_{u_2^*}U\cong\mathbb{R}^3\).
The condition \(\det T\neq 0\) is equivalent to
\begin{equation}\label{eq:sequential-transversality-3x3}
  \spn\bigl\{\Phi_{12}\,v_1,\; v_2,\; v_3\bigr\} = T_{u_2^*}U \cong \mathbb{R}^3.
\end{equation}
This is the \textbf{sequential transversality condition} for the \(3\times 3\) system; see Figure~\ref{fig:sequential-transversality} for a geometric illustration.

Unlike the \(2\times 2\) case, where both tangent vectors naturally live at the common intermediate state \(u_1^*\) and the transversality condition is simply \(T_{u_1^*}\mathcal{W}^f_{1,u_l}+T_{u_1^*}\mathcal{W}^b_{2,u_r}=T_{u_1^*}U\), in the \(3\times 3\) case the tangent contribution from the first wave originates at \(u_1^*\) and must be \emph{transported} to \(u_2^*\) via the pushforward \(\Phi_{12}\) before the spanning condition can be evaluated.
The transversality is \emph{sequential} in the sense that the tangent contributions from each wave are composed through the pushforward maps induced by the intermediate Hugoniot connections, and the final condition~\eqref{eq:sequential-transversality-3x3} is evaluated at a single base point.

Furthermore, observe that the choice of base point \(u_2^*\) (the forward-sweep convention) is not unique.
One could equally eliminate in reverse order, reducing to \(u_1^*\) instead, and the resulting condition would involve the inverse pushforward \(\Phi_{21}=\Phi_{12}^{-1}=\bigl(DF(u_1^*)-s_2^*I_n\bigr)^{-1}\bigl(DF(u_2^*)-s_2^*I_n\bigr)\) transporting vectors from \(u_2^*\) to \(u_1^*\).
The two formulations are equivalent since the pushforward is an isomorphism.

\begin{remark}[Generalizations]\label{rem:sequential-transversality-general}
The sequential transversality framework developed above for the triple-shock configuration extends to all wave-type combinations (shock--shock--rarefaction, etc.).
The general pushforward interpretation and explicit rarefaction block formulas are developed in Section~\ref{ssec:transversality-nxn}.
\end{remark}

\subsection{The general \texorpdfstring{$n\times n$}{nxn} case}
\label{ssec:transversality-nxn}

The arguments of Sections~\ref{ssec:transversality-2x2} and~\ref{ssec:transversality-3x3} generalize directly to an \(n\times n\) strictly hyperbolic system, with all $2^n$ wave-type configurations treated at once.

\subsubsection{The objective map}

Consider a Riemann solution with \(n\) waves,
\[
  u^{(0)} = u_l
  \xrightarrow{\;\omega_1^*\;}
  u^{(1)}
  \xrightarrow{\;\omega_2^*\;}
  \cdots
  \xrightarrow{\;\omega_n^*\;}
  u^{(n)} = u_r,
\]
in which the \(k\)-th wave is a Lax \(k\)-shock or a \(k\)-rarefaction, with \(n-1\) intermediate states \(u^{(1)},\dots,u^{(n-1)}\in U\subset\mathbb{R}^n\) and wave parameters \(\omega_1^*,\dots,\omega_n^*\), the waves arranged in order of increasing speed, where each wave is non-degenerate, i.e., \(u^{(k-1)}\neq u^{(k)}\) for all \(k=1,\ldots,n\).

The \(k\)-th \textbf{wave map} is set by the type of the \(k\)-th wave,
\begin{equation}\label{eq:wave-condition-cases}
  W_k \;\in\; \bigl\{\,H,\;\, R_k\,\bigr\},
  \qquad 1\le k\le n-1,
\end{equation}
the Rankine--Hugoniot map~\eqref{eq:RH-map} for a shock and the rarefaction map of \cref{def:rarefaction-map} for a rarefaction, with the running state $u^{(k)}$, the wave parameter $\omega_k$, and the previous state $u^{(k-1)}$. The \textbf{wave parameter} \(\omega_k\) is the wave's single scalar degree of freedom, the shock speed, \(\omega_k = s_k\), for a shock and the eigenvalue at the running state, \(\omega_k = \lambda_k^* := \lambda_k(u^{(k)})\), for a rarefaction. In the latter case \(\omega_k\) sets how far along the integral curve \(\Gamma_k(\,\cdot\,;\,u^{(k-1)})\) one travels, the startpoint eigenvalue \(\lambda_k(u^{(k-1)})\) being determined by the previous state. The terminal wave map $W_n \in \bigl\{\,H,\; R_n\,\bigr\}$ is instead anchored at the right state; \cref{rem:backward-anchoring} explains why this choice is essential for a terminal rarefaction.

Stacking the \(n\) wave maps gives the objective map
\begin{equation}\label{eq:objective-nxn}
  J\ml u^{(1)},\dots,u^{(n-1)},\omega_1,\dots,\omega_n\mr
  :=
  \begin{pmatrix}
    W_1\ml u^{(1)},\omega_1;\,u^{(0)} = u_l\mr \\[3pt]
    W_2\ml u^{(2)},\omega_2;\,u^{(1)}\mr \\[3pt]
    \vdots \\[3pt]
    W_{n-1}\ml u^{(n-1)},\omega_{n-1};\,u^{(n-2)}\mr \\[3pt]
    W_n\ml u^{(n-1)},\omega_n;\,u^{(n)} = u_r\mr
  \end{pmatrix}
  \in\mathbb{R}^{n^2},
\end{equation}
the first $n-1$ rows anchored at the previous state and the terminal row at the right state, exactly as in the $3\times 3$ case.
The unknowns \(\ml u^{(1)},\dots,u^{(n-1)},\omega_1,\dots,\omega_n\mr\in\mathbb{R}^{(n-1)n+n}=\mathbb{R}^{n^2}\) have the same dimension as the target, and the Riemann solution corresponds to a zero of \(J\).

Since the \(k\)-th block row of \(J\) depends only on \(u^{(k-1)}\), \(u^{(k)}\), and \(\omega_k\), the Jacobian of~\eqref{eq:objective-nxn} with respect to the unknowns \(\ml u^{(1)},\dots,u^{(n-1)},\omega_1,\dots,\omega_n\mr\) has the explicit block structure
\begin{equation}\label{eq:jacobian-nxn-block}
  DJ\ml u^{(1)},\dots,u^{(n-1)},\omega_1^*,\dots,\omega_n^*\mr =
  \begin{pmatrix}
    A_1  & 0      & \cdots & 0          & a_1  & 0      & \cdots & 0   \\
    B_2  & A_2    & \cdots & 0          & 0    & a_2    & \cdots & 0   \\
    0    & B_3    & \ddots & \vdots     & \vdots & \ddots & \ddots & \vdots \\
    \vdots & \ddots & \ddots & A_{n-1}  & 0    & \cdots & a_{n-1} & 0 \\
    0    & \cdots & 0      & B_n        & 0    & \cdots & 0       & a_n
  \end{pmatrix},
\end{equation}
where the blocks in row $k$ are the derivatives $A_k = \p_{u^{(k)}}W_k$ (diagonal, $n\times n$), $B_k = \p_{u^{(k-1)}}W_k$ (sub-diagonal, $n\times n$), and $a_k = \p_{\omega_k}W_k$ (speed column, $n\times 1$), each with respect to the indicated variable in whichever argument slot of $W_k$ it occupies.

For a shock wave, the blocks are those of the Rankine--Hugoniot map,
\begin{equation}\label{eq:shock-blocks}
\begin{aligned}
  A_k &= DF\ml u^{(k)}\mr - s_k^* I_n, \qquad && k\le n-1, \\
  B_k &= -\bigl(DF\ml u^{(k-1)}\mr - s_k^* I_n\bigr), \qquad && 2\le k\le n-1, \\
  a_k &= -\ml u^{(k)}-u^{(k-1)}\mr, \qquad && k\le n-1,
\end{aligned}
\end{equation}
with $\omega_k=s_k$. For a rarefaction wave with $1\le k\le n-1$, the wave map is $R_k = u^{(k)} - \Gamma_k(\lambda_k;\,u^{(k-1)})$, $\Gamma_k$ being the integral curve of the eigenvector field $r_k$, and the blocks are given by
\begin{equation}\label{eq:rarefaction-blocks}
\begin{aligned}
  A_k &= \p_{u^{(k)}}R_k = I_n, \qquad && k\le n-1, \\
  B_k &= \p_{u^{(k-1)}}R_k = -D_{u^{(k-1)}}\Gamma_k\ml\lambda_k^*;\,u^{(k-1)}\mr, \qquad && 2\le k\le n-1, \\
  a_k &= \p_{\lambda_k}R_k = -\Gamma_k'\ml\lambda_k^*;\,u^{(k-1)}\mr = -\alpha_k\,r_k\ml u^{(k)}\mr, \qquad && k\le n-1,
\end{aligned}
\end{equation}
where \(\Gamma_k'\) denotes the derivative with respect to the eigenvalue parameter \(\xi=\lambda_k\), the last equality by the rarefaction ODE~\eqref{eq:Gamma-xi-ODE} with nonzero scalar $\alpha_k$.
Similarly, the backward-anchored terminal row contributes
\begin{equation}\label{eq:terminal-blocks}
  \begin{aligned}
  B_n &= \p_{u^{(n-1)}}W_n =
  \begin{cases}
    DF\ml u^{(n-1)}\mr - \omega_n I_n & \text{(shock)},\\[3pt]
    I_n & \text{(rarefaction)},
  \end{cases}
  \\
  a_n &= \p_{\omega_n}W_n =
  \begin{cases}
    u_r - u^{(n-1)} & \text{(shock)},\\[3pt]
    -\alpha_n\,r_n\ml u^{(n-1)}\mr & \text{(rarefaction)},
  \end{cases}
  \end{aligned}
\end{equation}
with $\alpha_n\neq 0$ as in~\eqref{eq:Gamma-xi-ODE}.
This is precisely the block lower-bidiagonal structure of \cref{lem:block-reduction}, extending the pattern observed in the \(3\times 3\) case~\eqref{eq:jacobian-3x3-block}.

\subsubsection{Generalized sequential transversality}

To apply \cref{lem:block-reduction} to the Jacobian~\eqref{eq:jacobian-nxn-block}, we check that the diagonal blocks are invertible. Each rarefaction block is the identity by~\eqref{eq:rarefaction-blocks}, and each shock block $A_k = DF\ml u^{(k)}\mr - s_k^* I_n$ is invertible because the Lax condition~\eqref{eq:lax-entropy} ensures that $s_k^*$ is not an eigenvalue of $DF\ml u^{(k)}\mr$.

The lemma also requires the terminal block $B_n$ to be invertible; the interior sub-diagonal blocks are never inverted. By~\eqref{eq:terminal-blocks}, $B_n = I_n$ for a terminal rarefaction, and for a terminal shock the Lax condition~\eqref{eq:lax-entropy} places $s_n^*$ strictly between $\lambda_{n-1}\ml u^{(n-1)}\mr$ and $\lambda_n\ml u^{(n-1)}\mr$, so $B_n = DF\ml u^{(n-1)}\mr - s_n^* I_n$ is invertible by strict hyperbolicity. \cref{lem:block-reduction} therefore applies in every wave-type configuration.

The propagation matrix \(P_k = -A_k^{-1}B_k\) from \cref{lem:block-reduction} is geometrically a \textbf{pushforward} from $u^{(k-1)}$ to $u^{(k)}$.
To see this, for $2\le k\le n-1$, the wave condition \(W_k\ml u^{(k)},\,\omega_k;\,u^{(k-1)}\mr = 0\), evaluated at a fixed wave parameter \(\omega_k^*\), implicitly defines \(u^{(k)}\) as a function of \(u^{(k-1)}\) since \(A_k = \p_{u^{(k)}}W_k\) is invertible.
By the implicit function theorem, the derivative of this map is given by
\begin{equation}\label{eq:pushforward-def}
  \frac{\p u^{(k)}}{\p u^{(k-1)}} = -A_k^{-1}\,B_k = P_k,
\end{equation}
where \(B_k = \p_{u^{(k-1)}}W_k\).
This is the pushforward \(\Phi_{k-1,\,k}\) that transports tangent vectors along the \(k\)-th wave curve.

Substituting the shock blocks~\eqref{eq:shock-blocks} into~\eqref{eq:pushforward-def} gives
\begin{equation}\label{eq:pushforward-chain}
  \Phi_{j,\,j+1} := P_{j+1} = \bigl(DF\ml u^{(j+1)}\mr - s_{j+1}^*I_n\bigr)^{-1}
    \bigl(DF\ml u^{(j)}\mr - s_{j+1}^*I_n\bigr), \qquad 1\le j\le n-2,
\end{equation}
exactly as derived in~\eqref{eq:pushforward-derivative} for the \(3\times 3\) case, while the rarefaction blocks~\eqref{eq:rarefaction-blocks} give $\Phi_{k-1,\,k}=-B_k=D_{u^{(k-1)}}\Gamma_k\ml\lambda_k^*;\,u^{(k-1)}\mr$, the derivative of the rarefaction curve with respect to its base state at fixed eigenvalue level. Notice that by construction the terminal wave defines no pushforward.
For \(j<m\) we set \(\Phi_{j,\,m}:=\Phi_{m-1,\,m}\circ\cdots\circ\Phi_{j,\,j+1}\).

By \cref{lem:block-reduction}, the determinant of~\eqref{eq:jacobian-nxn-block} factors as $\bigl(\prod_{k=1}^{n-1}\det A_k\bigr)\det B_n\,\det T$, where the \textbf{transversality matrix} $T\in\mathbb{R}^{n\times n}$ is the reduced matrix given by the column formula~\eqref{eq:T-column-formula}. With the seed vectors $c_k = -A_k^{-1}a_k$ from~\eqref{eq:seed-vector} and the composed pushforwards $\Phi_{k,\,n-1}=P_{n-1}\cdots P_{k+1}$, its columns are given by
\begin{equation}\label{eq:transversality-general}
  T_k =
  \begin{cases}
    \Phi_{k,\,n-1}\,c_k & \text{if } 1\le k\le n-1,\\[4pt]
    B_n^{-1}a_n & \text{if } k=n,
  \end{cases}
\end{equation}
where \(\Phi_{n-1,\,n-1}\) is the identity by convention, and the blocks are set by the wave types via the block values~\eqref{eq:shock-blocks}--\eqref{eq:terminal-blocks}. The seed vector is $c_k = A_k^{-1}\ml u^{(k)}-u^{(k-1)}\mr$ for a shock and $c_k = \Gamma_k'\ml\lambda_k^*;\,u^{(k-1)}\mr = \alpha_k\,r_k\ml u^{(k)}\mr$ for a rarefaction, a scalar multiple of the $k$-th right eigenvector. Invertibility of the full Jacobian is therefore equivalent to the \emph{transversality condition} $\det T\neq 0$.

Each column of \(T\) has a geometric interpretation.
The seed vector $c_k$ is tangent at $u^{(k)}$ to the projected wave curve based at $u^{(k-1)}$, the Hugoniot locus for a shock (cf.\ Section~\ref{ssec:transversality-2x2}) and the rarefaction curve for a rarefaction. The pushforward chain \(\Phi_{k,\,n-1}\) then transports this tangent forward through the sequence of intermediate states to the common reference point \(u^{(n-1)}\).
The final column $B_n^{-1}a_n$ is tangent at $u^{(n-1)}$ to the backward $n$-wave curve $\mathcal{W}^b_{n,u_r}$ of \cref{def:wave-curves}, along the branch $\mathcal{H}^+_{n,u_r}$ for a shock and $\mathcal{R}^-_{n,u_r}$ for a rarefaction.
The transversality matrix thus collects \(n\) tangent directions at \(u^{(n-1)}\), each originating from a different wave curve, and the condition $\det T\neq 0$ asks whether these directions span all of \(\mathbb{R}^n\).

For the all-shocks configuration, the columns are given explicitly by
\begin{equation}\label{eq:transversality-normalized-nxn}
  T_k =
  \begin{cases}
    \Phi_{k,\,n-1}\,A_k^{-1}\ml u^{(k)}-u^{(k-1)}\mr & \text{if } 1\le k\le n-1,\\[4pt]
    \bigl(DF\ml u^{(n-1)}\mr - s_n^*I_n\bigr)^{-1}\ml u_r-u^{(n-1)}\mr & \text{if } k=n,
  \end{cases}
\end{equation}
with the pushforwards~\eqref{eq:pushforward-chain}. The final column is the general form of the vector $v_3$ of the $3\times 3$ case, computed and interpreted geometrically in Section~\ref{ssec:transversality-3x3}. In the all-rarefactions configuration, each pushforward is the base-state derivative $\Phi_{j,\,j+1} = D_{u^{(j)}}\Gamma_{j+1}\ml\lambda_{j+1}^*;\,u^{(j)}\mr$, and the columns reduce to transported eigenvector directions, $T_k = \Phi_{k,\,n-1}\,\alpha_k\,r_k\ml u^{(k)}\mr$ for $1\le k\le n-1$ and $T_n = -\alpha_n\,r_n\ml u^{(n-1)}\mr$.

\paragraph{Independence of the base point}
The choice of \(u^{(n-1)}\) as the reference point is a consequence of the forward-sweep elimination; in the all-shocks configuration, the transversality condition itself is independent of this choice.
To evaluate the condition at any other intermediate state \(u^{(m)}\) (\(1\le m\le n-1\)), one applies the inverse pushforward \(\Phi_{m,\,n-1}^{-1}\) to every column of \(T\), obtaining a matrix whose columns are the tangent contributions transported to \(u^{(m)}\).
Since in the all-shocks case each pushforward \(\Phi_{j,\,j+1}\) is a linear isomorphism (both factors of~\eqref{eq:pushforward-chain} are invertible, the Lax condition~\eqref{eq:lax-entropy} keeping $s_{j+1}^*$ off the spectrum of $DF$ at both adjacent states), so is the composed map \(\Phi_{m,\,n-1}\), and therefore
\[
  \det T\neq 0
  \quad\Longleftrightarrow\quad
  \det\bigl(\Phi_{m,\,n-1}^{-1}\,T\bigr)\neq 0.
\]
The transversality condition is thus an intrinsic property of the wave configuration.
If an interior wave $k\in\{2,\dots,n-1\}$ is a rarefaction, its pushforward is not a linear isomorphism (\cref{lem:Gamma-base-rank}). The geometric interpretation nevertheless survives in principle, since the wave's own tangent $\alpha_k\,r_k\ml u^{(k)}\mr$ spans the eigenline $E_k\ml u^{(k)}\mr$, which the eigenchart construction of Section~\ref{ssec:rarefaction} defines along the entire curve, so the tangent slides along the integral curve from $u^{(k)}$ to $u^{(k-1)}$, and a frame transport across the rarefaction can be assembled from this slide. We leave the formal construction to the interested reader. The sequential transversality condition $\det T\neq 0$ therefore provides a unified criterion for structural stability across all wave-type configurations.

\begin{remark}[Why the terminal wave is anchored backward]\label{rem:backward-anchoring}
A forward-anchored terminal rarefaction $W_n = u_r - \Gamma_n(\omega_n;\,u^{(n-1)})$ with $\omega_n=\lambda_n(u_r)$ would have $B_n = -D_{u^{(n-1)}}\Gamma_n$, which has rank $n-1$ (\cref{lem:Gamma-base-rank}), so \cref{lem:block-reduction} would not apply. Anchored at $u_r$ instead, the terminal wave has invertible $B_n$ in both cases~\eqref{eq:terminal-blocks}, and wave $n$ branches backward from $u_r$ to meet the forward-chained waves $1,\dots,n-1$ at $u^{(n-1)}$, where the columns of $T$ collect the $n-1$ forward-transported tangents and the backward wave-curve tangent.
\end{remark}

\section{Main Theorems}\label{sec:structural-stability}

This section establishes the three main results of the paper. The first, \cref{thm:structural-stability}, shows that the transversality condition guarantees structural stability of the Riemann solution under perturbations of the left and right states \emph{and} the flux function. The second, \cref{thm:genericity}, shows that this transversality condition holds for almost every choice of left and right states. Together they yield the third, generic structural stability (\cref{thm:generic-structural-stability}).

\subsection{The grand objective map}\label{ssec:grand-objective}

The objective map~\eqref{eq:objective-nxn} was defined in Section~\ref{ssec:transversality-nxn} with the left state \(u_l\), right state \(u_r\), and flux \(F\) held fixed.
For the structural stability analysis, we must allow these quantities to vary.
Write
\(
  x = \bigl(u^{(1)},\dots,u^{(n-1)},\,\omega_1,\dots,\omega_n\bigr) \in \mathbb{R}^{n^2}
\)
for the wave parameters and
\(
  y = \bigl(u_l,\,u_r,\,F\bigr) \in U^2\times\bigl[C^2(K)\bigr]^n
\)
for the perturbation parameters, where \(K\subset U\) is a compact connected set that is the closure of its interior. Here \([C^2(K)]^n\) is the space of maps \(K\to\mathbb{R}^n\) with components in \(C^2(K)\), a Banach space under the norm \(\|F\|_{[C^2(K)]^n}=\sum_{i=1}^{n}\|F_i\|_{C^2(K)}\); since the analysis involves the flux only through its values on \(K\), we do not distinguish \(F\) from its restriction to \(K\).

The \textbf{grand objective map} is then defined by
\begin{equation}\label{eq:extended-objective}
  \mathcal{J}(x;\,y)
  :=
  \begin{pmatrix}
    W_1\bigl(u^{(1)},\,\omega_1;\,u^{(0)},\,F\bigr) \\[3pt]
    W_2\bigl(u^{(2)},\,\omega_2;\,u^{(1)},\,F\bigr) \\[3pt]
    \vdots \\[3pt]
    W_{n-1}\bigl(u^{(n-1)},\,\omega_{n-1};\,u^{(n-2)},\,F\bigr) \\[3pt]
    W_n\bigl(u^{(n-1)},\,\omega_n;\,u^{(n)},\,F\bigr)
  \end{pmatrix}
  \in\mathbb{R}^{n^2},
\end{equation}
where \(u^{(0)}=u_l\), \(u^{(n)}=u_r\), and the $W_k$ are the wave maps of Section~\ref{ssec:transversality-nxn}, the case selection~\eqref{eq:wave-condition-cases} for $k\le n-1$ and the backward-anchored terminal map of~\eqref{eq:objective-nxn} for $k=n$, now carrying an explicit flux argument $F$.
At fixed parameters \(y=y_0=(u_l,u_r,F)\), the grand objective map \(\mathcal{J}(\,\cdot\,;\,y_0)\) coincides with the objective map \(J\) from~\eqref{eq:objective-nxn}.
In the same way, the transversality matrix~\eqref{eq:transversality-general} acquires a parameter dependence, and we write \(T(x;\,y)\) when we wish to emphasize this.
The transversality matrix \(T(x;\,y)\) is defined at a zero of \(\mathcal{J}(\,\cdot\,;\,y)\) precisely when the diagonal blocks \(A_k\) and the terminal block $B_n$ of the Jacobian \(D_x\mathcal{J}\) are invertible, so that \cref{lem:block-reduction} applies. This holds at every admissible zero, where the Lax entropy condition makes each shock block \(A_k\) and the terminal-shock block $B_n$ invertible and each rarefaction block equals \(I_n\). Note that at a zero violating the Lax conditions, a shock speed can coincide with an eigenvalue of $DF$ at an adjacent state, and $T$ is then undefined. Where \(T\) is defined, the factorization \(\det(D_x\mathcal{J})=\det A_1\cdots\det A_{n-1}\,\det B_n\,\det T\) of \cref{lem:block-reduction} shows that \(\det T\neq 0\) is equivalent to invertibility of \(D_x\mathcal{J}\), hence to transversality of \(\mathcal{J}(\,\cdot\,;\,y)\) to \(\{0\}\) at that zero. We call invertibility of \(D_x\mathcal{J}\) at a zero the \textbf{transversality condition}. At admissible zeros, where \(T\) is defined, it takes the matrix form \(\det T\neq 0\), which is equivalent to the sequential transversality condition of Section~\ref{sec:transversality}.

Let \(\mathcal{O}\subset[C^2(K)]^n\) be an open neighborhood of \(F\) such that Assumptions~\ref{ass:SH} and~\ref{ass:GN} continue to hold on $\intr(K)$ for every \(\tilde{F}\in\mathcal{O}\); such a neighborhood exists since these are open conditions on the compact set $K$.
The grand objective map is then a map
\[
  \mathcal{J}\colon\;
  \overbrace{\underbrace{\intr(K)^{n-1}}_{u^{(1)},\,\dots,\,u^{(n-1)}}\times\underbrace{\mathbb{R}^{n}}_{\omega_1,\,\dots,\,\omega_n}}^{x}
  \;\times\;
  \overbrace{\underbrace{\intr(K)^{2}}_{u_l,\;u_r}\times\underbrace{\mathcal{O}}_{F}}^{y}
  \;\longrightarrow\;\mathbb{R}^{n^2}.
\]

\subsection{Main results}\label{ssec:main-results}

Throughout this subsection, the system~\eqref{eq:conservation-law} satisfies Assumptions~\ref{ass:SH},~\ref{ass:GN}, and~\ref{ass:regular} with Riemann initial data~\eqref{eq:Riemann-data}, and $K\subset U$ denotes a compact connected set as in Section~\ref{ssec:grand-objective}. We consider Riemann solutions $u^{(0)}=u_l,\,u^{(1)},\dots,u^{(n-1)},\,u^{(n)}=u_r$ with $n$ waves, each a Lax shock or a rarefaction, satisfying the Lax entropy condition, whose states and rarefaction arcs lie in $\intr(K)$. The wave parameters $x_0$ of such a solution satisfy $\mathcal{J}(x_0;\,y_0)=0$ at the unperturbed parameters $y_0=(u_l,u_r,F)$, where the grand objective map~\eqref{eq:extended-objective} is formed with the solution's wave pattern.

Following~\cite{Schecter}, associate to each wave of such a solution its \textbf{speed interval} $[\omega_k^-,\,\omega_k^+]$, defined as the single point $[\omega_k,\,\omega_k]$ for the $k$-th shock and as the range of characteristic speeds $[\lambda_k(u^{(k-1)}),\,\lambda_k(u^{(k)})]$ across the fan for the $k$-th rarefaction. The upper endpoint $\omega_k^+$ is the wave parameter $\omega_k$ of Section~\ref{ssec:transversality-nxn} when $k\le n-1$; for a terminal rarefaction the wave parameter is instead the lower endpoint $\omega_n^-=\lambda_n(u^{(n-1)})$, by the backward anchoring~\eqref{eq:objective-nxn}. The same intervals, written $[\tilde{\omega}_k^-,\,\tilde{\omega}_k^+]$, are associated to a perturbed solution.

\begin{definition}[Structural stability]\label{def:structural-stability}
Suppose the states of a Riemann solution as above all lie in $\intr(K)$.
We say that the solution is \textbf{structurally stable on \(K\)} if there exist neighborhoods \(\mathcal{U}_l\) of \(u_l\), \(\mathcal{U}_r\) of \(u_r\) in \(\intr(K)\), and \(\mathcal{O}\) of \(F\) in \([C^2(K)]^n\), together with a \(C^1\) map
\[
  \bigl(\tilde{u}_l,\,\tilde{u}_r,\,\tilde{F}\bigr)
  \;\longmapsto\;
  \bigl(\tilde{u}^{(1)},\dots,\tilde{u}^{(n-1)},\,\tilde{\omega}_1,\dots,\tilde{\omega}_n\bigr)
\]
defined on \(\mathcal{U}_l\times\mathcal{U}_r\times\mathcal{O}\), such that the perturbed Riemann problem with left and right states \((\tilde{u}_l,\tilde{u}_r)\) and flux \(\tilde{F}\) admits a unique nearby solution satisfying the following conditions:
\begin{enumerate}[label=\textup{(\roman*)},leftmargin=*]
\item all perturbed intermediate states remain in \(\intr(K)\),
\item the \(k\)-th wave is a Lax shock (resp.\ rarefaction) whenever the original \(k\)-th wave is a shock (resp.\ rarefaction),
\item the Lax entropy conditions~\eqref{eq:lax-entropy} hold for every shock, and
\item the wave speeds are ordered, that is, the speed intervals of the perturbed solution satisfy
\begin{equation}\label{eq:speed-chain}
  \tilde{\omega}_1^-\le\tilde{\omega}_1^+<\tilde{\omega}_2^-\le\tilde{\omega}_2^+<\cdots<\tilde{\omega}_n^-\le\tilde{\omega}_n^+,
\end{equation}
where equality \(\tilde{\omega}_k^-=\tilde{\omega}_k^+\) holds for that $k$ if and only if the \(k\)-th wave is a shock.
\end{enumerate}
Moreover, since the solution map is \(C^1\), the perturbed intermediate states and speed intervals are close to their unperturbed values, that is, for any \(\varepsilon>0\) there exists \(\delta>0\) such that whenever
\[
  \|\tilde{u}_l-u_l\|_2 + \|\tilde{u}_r-u_r\|_2 + \|\tilde{F}-F\|_{[C^2(K)]^n} < \delta,
\]
the perturbed solution satisfies
\[
  \max_{1\le k\le n-1}\|\tilde{u}^{(k)}-u^{(k)}\|_2 + \max_{1\le k\le n,\ \sigma\in\{-,+\}}|\tilde{\omega}_k^{\sigma}-\omega_k^{\sigma}| < \varepsilon.
\]
\end{definition}

The within-wave strict inequalities in~\eqref{eq:speed-chain} encode rarefaction fan monotonicity, ensuring each fan spreads in the physically correct direction (see the non-zero speed gaps in Section~\ref{ssec:transversality-2x2}) while the between-wave strict inequalities $\tilde{\omega}_k^+<\tilde{\omega}_{k+1}^-$ ensure that consecutive waves are speed-separated. Since all strict inequalities in~\eqref{eq:speed-chain} involve continuous functions of the wave parameters and intermediate states, the speed chain is an open condition. Hence every sufficiently small perturbation preserves the strict inequalities, because the unperturbed solution satisfies them strictly.

Recall from Section~\ref{ssec:grand-objective} that the transversality condition holds at a zero of $\mathcal{J}$ if $D_x\mathcal{J}$ is invertible there. We say that it \textbf{holds on $K$} if it holds at every zero of $\mathcal{J}(\,\cdot\,;\,y_0)$ in $\intr(K)^{n-1}\times\mathbb{R}^n$, and that it \textbf{holds strictly on $K$} if, in addition, every such zero other than $x_0$ \emph{strictly violates} admissibility, satisfying the strict reversal of at least one of the inequalities in the Lax conditions~\eqref{eq:lax-entropy} or the speed chain~\eqref{eq:speed-chain} (the zero $x_0$ itself, being admissible, satisfies them strictly).

\begin{theorem}[Structural stability]\label{thm:structural-stability}
Consider a Riemann solution as above whose states all lie in $\intr(K)$. If the transversality condition holds at \(x_0\), then the Riemann solution is structurally stable on \(K\) in the sense of \cref{def:structural-stability}. If, in addition, the transversality condition holds strictly on $K$, then the perturbed solution is the unique admissible Riemann solution of the same wave pattern in a fixed compact neighborhood of the original.
\end{theorem}

\begin{theorem}[Genericity of the transversality condition]\label{thm:genericity}
Fix a wave pattern, the assignment of shock or rarefaction to each of the \(n\) waves as in~\eqref{eq:wave-condition-cases}, with which the grand objective map is formed.
For almost every\footnote{Here, ``almost every'' is understood in the sense that the set of \((u_l,u_r)\) for which the conclusion fails is of Lebesgue measure zero in \(U^2\).} \((u_l,u_r)\in U^2\) with $y = (u_l, u_r, F)$, the transversality condition holds at every non-degenerate zero of \(\mathcal{J}(\,\cdot\,;\,y)\), a zero with \(u^{(k)}\neq u^{(k-1)}\) for all \(k\).
\end{theorem}

\Cref{thm:genericity} fixes the wave pattern first and then chooses the pair of states. Since there are only finitely many patterns, the order of quantifiers can be interchanged.

\begin{corollary}[Pattern-uniform genericity]\label{cor:pattern-uniform}
For almost every $(u_l,u_r)\in U^2$ with $y=(u_l,u_r,F)$, the transversality condition holds at every non-degenerate zero of the grand objective map formed with any of the $2^n$ wave patterns.
\end{corollary}

\begin{proof}
For each of the $2^n$ wave patterns, \cref{thm:genericity} provides a full-measure set of pairs in $U^2$ for which its conclusion holds. The intersection of these finitely many full-measure sets again has full measure in $U^2$. For every pair in the intersection the conclusion holds for every pattern.
\end{proof}

Combining \cref{thm:structural-stability} and \cref{cor:pattern-uniform} gives the following.

\begin{theorem}[Generic structural stability]\label{thm:generic-structural-stability}
For almost every \((u_l,u_r)\in U^2\), every Riemann solution consisting of \(n\) Lax-admissible waves (if any exists) is structurally stable in the sense of \cref{def:structural-stability} on every compact connected \(K\subset U\) whose interior contains \(u_l\), \(u_r\), the intermediate states, and the arc of each rarefaction wave.
\end{theorem}

\begin{proof}
Fix a nested sequence $K_1\subset K_2\subset\cdots$ of compact connected subsets of $U$ such that every point of $U$ lies in $\intr(K_j)$ for some index $j$ depending on the point. For each \(j\), \cref{cor:pattern-uniform} applied with \(K=K_j\) yields a full-measure set \(G_j\subset U^2\) of pairs for which the transversality condition holds at every non-degenerate zero of the grand objective map determined by \(K_j\), for every wave pattern. The countable intersection \(G:=\bigcap_j G_j\) again has full measure.

Fix \((u_l,u_r)\in G\) and a Riemann solution with \(n\) Lax-admissible waves. Let \(K\subset U\) be compact and connected with \(u_l\), \(u_r\), the intermediate states, and the arc of each rarefaction wave in its interior. Since $K$ is compact, finitely many of the nested sets $\intr(K_j)$ cover $K$, so $K\subset K_j$ for the largest index $j$ among them. Every wave map defined relative to \(K\) is defined relative to \(K_j\). The zero \(x_0\) determined by the solution is therefore also a non-degenerate zero of the grand objective map determined by \(K_j\), formed with the solution's wave pattern. The transversality condition holds at this zero since \((u_l,u_r)\in G_j\). \Cref{thm:structural-stability} then applies on \(K\).
\end{proof}

\begin{remark}[Dependence on the compact set]\label{rem:exhaustion}
The nested sequence has no counterpart in our prior \(2\times2\) work~\cite{TanBertozzi2x2}. There the transversality property is a statement about the full wave curves, which are globally defined curves in \(U\), so a single full-measure set of pairs works for every compact set at once. Here the grand objective map is defined only once a compact set \(K\) is chosen. Its rarefaction maps are built from local eigencharts and continued only within \(\intr(K)\), and its flux argument lives in \([C^2(K)]^n\), so both the map and its domain depend on \(K\). \Cref{cor:pattern-uniform} therefore produces one full-measure set of pairs in \(U^2\) for each choice of \(K\), while \cref{thm:generic-structural-stability} asserts a single full-measure set valid for every \(K\). These sets cannot simply be intersected over all compact \(K\), since a countable intersection of full-measure sets has full measure, but an uncountable one need not. The nested sequence resolves this. Every compact \(K\) lies inside some \(K_j\), so the countably many sets produced for the \(K_j\) suffice, and their countable intersection retains full measure.
\end{remark}

\begin{remark}[Recovery of lower-dimensional cases]\label{rem:recovery-lower-dim}
For \(n=2\), the transversality matrix~\eqref{eq:transversality-general} reduces to the \(2\times 2\) matrix~\eqref{eq:transversality-det-2x2} for the shock-rarefaction case, since there is a single intermediate state and no pushforward chain is needed.
For \(n=3\), the formula~\eqref{eq:transversality-general} specializes to~\eqref{eq:transversality-normalized-3x3} for the triple shock case, the $n=3$ instance of the all-shocks columns~\eqref{eq:transversality-normalized-nxn}, based at the common base point \(u^{(2)}=u_2^*=u^{(n-1)}\). The \(3\times 3\) treatment of Section~\ref{ssec:transversality-3x3} is therefore a direct instance of the general framework.
\end{remark}

\begin{remark}[Relation to the \(2\times2\) theory]\label{rem:transversality-2x2}
The transversality condition at \emph{every} zero of \(\mathcal{J}(\,\cdot\,;\,y_0)\), which underlies the global uniqueness clause of \cref{thm:structural-stability}, is the \(n\times n\) counterpart of the transversality property assumed in our prior \(2\times2\) work~\cite{TanBertozzi2x2}. For \(n=2\) a zero of \(\mathcal{J}\) is an intersection of the loci selected by the wave pattern, $\mathcal{H}_{u_l}$ or $\mathcal{R}_{1,u_l}$ at $u_l$ with $\mathcal{H}_{u_r}$ or $\mathcal{R}_{2,u_r}$ at $u_r$, and the zero is admissible when the Lax inequalities hold there; by \cref{rem:recovery-lower-dim} the transversality matrix reduces to the \(2\times2\) matrix~\eqref{eq:transversality-det-2x2} where \(T\) is defined, in particular at every admissible zero, and \(\det T\neq 0\) at such a zero is exactly the transverse crossing of the two loci; at any remaining zero the condition is invertibility of \(D_x\mathcal{J}\), as in Section~\ref{ssec:grand-objective}. Imposing it at every zero therefore requires that \emph{every} such intersection on the compact set, admissible or not, be transverse, which is precisely what was assumed in our prior work. \Cref{cor:pattern-uniform} shows that, for almost every pair of left and right states, this holds at every non-degenerate zero, exactly as for \(n=2\).
\end{remark}

\subsection{Supporting results}

The proofs of \cref{thm:structural-stability,thm:genericity} rely on the following two results, which we state before the proofs for ease of reference.

\begin{theorem}[Implicit function theorem on Banach spaces {\cite{HunterApplied}}]\label{thm:IFT-Banach}
Suppose that \(X\), \(Y\), and \(Z\) are Banach spaces, and \(\mathbf{F}\colon O\subset X\times Y\to Z\) is a \(C^1\) map defined on an open subset \(O\) of \(X\times Y\). If \((x_0,y_0)\in O\) satisfies \(\mathbf{F}(x_0,y_0)=0\) and \(D_x\mathbf{F}(x_0,y_0)\colon X\to Z\) is a bijective bounded linear map, then there exist \(\varepsilon_1,\varepsilon_2>0\), an open ball \(B_{\varepsilon_1}(y_0)\subset Y\), an open neighborhood \(V(x_0)\) of \(x_0\), and a unique map \(M\colon B_{\varepsilon_1}(y_0)\to V(x_0)\subset B_{\varepsilon_2}(x_0)\) such that
$\mathbf{F}(M(y),y)=0$ for all $y\in B_{\varepsilon_1}(y_0)$, and moreover \(M(y)\) is the only zero of \(\mathbf{F}(\,\cdot\,,y)\) in \(V(x_0)\) for each such \(y\).
The map \(M\) is continuously differentiable, with
\[
  DM(y) = -\bigl[D_x\mathbf{F}(M(y),y)\bigr]^{-1}\,D_y\mathbf{F}(M(y),y).
\]
\end{theorem}

\begin{lemma}[\(C^1\) regularity of the grand objective map]\label{lem:C1-objective}
Let \(K\subset U\) be a compact connected set, and let \(\mathcal{O}\subset[C^2(K)]^n\) be an open neighborhood of \(F\) such that Assumptions~\ref{ass:SH} and~\ref{ass:GN} hold on $\intr(K)$ for every \(\tilde{F}\in\mathcal{O}\).
Let \(\mathcal{D}\subset \intr(K)^{n-1}\times\mathbb{R}^n \times \intr(K)^2 \times \mathcal{O}\) be the set on which every wave map \(W_k\) in~\eqref{eq:extended-objective} is defined. Then \(\mathcal{D}\) is open, and the grand objective map~\eqref{eq:extended-objective}, viewed as a map
\[
  \mathcal{J}\colon \mathcal{D} \;\longrightarrow\;\mathbb{R}^{n^2},
\]
is \(C^1\) with respect to the product topology, where \(\intr(K)^{n-1}\times\mathbb{R}^n\) is open in \(\mathbb{R}^{n^2}\), \([C^2(K)]^n\) carries the Banach norm of Section~\ref{ssec:grand-objective}, and \(\intr(K)^2\) carries the Euclidean norm.
\end{lemma}

Since a zero of $\mathcal{J}$ is in particular a point where every wave map is defined, $\mathcal{D}$ contains all zeros of $\mathcal{J}$.
The proof of \cref{lem:C1-objective} is deferred to Appendix~\ref{app:C1-proof}.

\subsection{Proof of \texorpdfstring{\cref{thm:structural-stability}}{Theorem~4.2}}

\begin{proof}[Proof of \cref{thm:structural-stability}]
We organize the argument into four parts.

\medskip\noindent\textbf{(A) Persistence of existence.}
Let the wave parameters of the given solution be \(x_0=(u^{(1)},\dots,u^{(n-1)},\omega_1^*,\dots,\omega_n^*)\in\intr(K)^{n-1}\times\mathbb{R}^n\), and the unperturbed parameters \(y_0=(u_l,u_r,F)\in\intr(K)^2\times\mathcal{O}\).
The grand objective map satisfies \(\mathcal{J}(x_0;\,y_0)=0\).
The partial Jacobian \(D_x\mathcal{J}(x_0;\,y_0)\) is the \(n^2\times n^2\) matrix~\eqref{eq:jacobian-nxn-block}, its blocks set by the wave types, and is invertible by the transversality hypothesis.
Since \(\mathcal{J}\) is \(C^1\) by \cref{lem:C1-objective}, we may apply \cref{thm:IFT-Banach} (with \(X=\mathbb{R}^{n^2}\), \(Y=\mathbb{R}^{2n}\times[C^2(K)]^n\), and \(Z=\mathbb{R}^{n^2}\), the map \(\mathcal{J}\) defined on the open subset \(\mathcal{D}\subset X\times Y\)) to obtain a \(C^1\) map \(M\colon B_\varepsilon(y_0)\to \intr(K)^{n-1}\times\mathbb{R}^n\) satisfying \(\mathcal{J}(M(y);\,y)=0\) for all \(y\) in a neighborhood of \(y_0\).

\medskip\noindent\textbf{(B) Persistence of admissibility.}
The wave type of each wave (shock or rarefaction) is preserved by construction, since the grand objective map \(\mathcal{J}\) from~\eqref{eq:extended-objective} is defined with a fixed wave pattern, and the implicit function map \(M\) produces solutions to the same system \(\mathcal{J}=0\).

It remains to verify conditions~(i)--(iv) of \cref{def:structural-stability}. At the unperturbed solution \(x_0\), the following open conditions hold:
\begin{itemize}[leftmargin=*]
\item \emph{Lax entropy.} For each \(k\)-shock, \(\lambda_{k-1}(u^{(k-1)})<s_k<\lambda_k(u^{(k-1)})\) and \(\lambda_k(u^{(k)})<s_k<\lambda_{k+1}(u^{(k)})\), with the convention \(\lambda_0=-\infty\), \(\lambda_{n+1}=+\infty\).
\item \emph{Speed chain.} The ordering~\eqref{eq:speed-chain} holds, with speed intervals as defined in Section~\ref{ssec:main-results}. The within-wave strict inequalities (rarefaction fan monotonicity) and the between-wave strict inequalities (inter-wave separation) are all open conditions. When both waves \(k\) and \(k+1\) are rarefactions, the between-wave condition \(\lambda_k(u^{(k)})<\lambda_{k+1}(u^{(k)})\) holds automatically by strict hyperbolicity.
\item \emph{Non-degeneracy.} \(u^{(k)}\neq u^{(k-1)}\) for all \(k=1,\dots,n\).
\item \emph{Containment.} \(u^{(k)}\in\intr(K)\) for all \(k=1,\dots,n-1\).
\end{itemize}
Each condition is an open condition in the wave parameters, intermediate states, and perturbation parameters. Since the implicit function map \(M\) is continuous, composing with \(x=M(y)\) makes each condition an open condition in \(y\) alone, and therefore each condition holds on an open neighborhood of \(y_0\). Taking the finite intersection over all conditions and intersecting with \(B_\varepsilon(y_0)\), we obtain a neighborhood of \(y_0\) on which all admissibility conditions are preserved.

\medskip\noindent\textbf{(C) Persistence of transversality.}
Since \(\mathcal{J}\) is \(C^1\) by \cref{lem:C1-objective} and the implicit function map \(M\) is \(C^1\) by \cref{thm:IFT-Banach}, the map \(y\mapsto\det D_x\mathcal{J}(M(y);\,y)\) is continuous, the determinant being a polynomial in the continuous entries of \(D_x\mathcal{J}\). This determinant is nonzero at \(y_0\) by hypothesis, so it remains nonzero for all \(y\) in a neighborhood of \(y_0\), and the transversality condition holds at each perturbed solution.

Parts (A)--(C) provide, for every \(y\) in a neighborhood of \(y_0\), a unique nearby solution \(M(y)\) of the same wave types satisfying conditions (i)--(iv) of \cref{def:structural-stability} and remaining transverse. The Riemann solution is therefore structurally stable on \(K\).

\medskip\noindent\textbf{(D) Global uniqueness.}
We now invoke the additional hypothesis of \cref{thm:structural-stability}, that the transversality condition holds strictly on \(K\), so that it holds at \emph{every} zero of \(\mathcal{J}(\,\cdot\,;\,y_0)\) in \(\intr(K)^{n-1}\times\mathbb{R}^n\) and every such zero other than \(x_0\) strictly violates admissibility. We then show that the perturbed solution is the unique admissible Riemann solution of the same wave pattern in a fixed compact neighborhood of the original.
By definition of the transversality condition (Section~\ref{ssec:grand-objective}), \(D_x\mathcal{J}(z;\,y_0)\) is invertible (and in particular surjective) at every such zero \(z\).
Hence \(0\) is a regular value of \(\mathcal{J}(\,\cdot\,;\,y_0)\), and by the Regular Value Theorem (\cref{thm:rvt}) the zero set \(\mathcal{J}(\,\cdot\,;\,y_0)^{-1}(\{0\})\cap\bigl(\intr(K)^{n-1}\times\mathbb{R}^n\bigr)\) is a submanifold of dimension \(n^2-n^2=0\), that is, a discrete set of isolated points.
A discrete set may still be infinite, and finiteness requires confining the zeros to a compact set. Since a zero is in particular a point where every wave map is defined, all zeros of \(\mathcal{J}(\,\cdot\,;\,y_0)\) lie in the open slice \(\mathcal{D}_{y_0}:=\{x : (x;\,y_0)\in\mathcal{D}\}\subset\intr(K)^{n-1}\times\mathbb{R}^n\) of the domain of \cref{lem:C1-objective}. We fix a compact set \(C\subset\mathcal{D}_{y_0}\) with \(x_0\in\intr(C)\), the compact neighborhood on which the uniqueness assertion will be proved.
Every zero of \(\mathcal{J}(\,\cdot\,;\,y_0)\) in \(C\) lies in \(\intr(K)^{n-1}\times\mathbb{R}^n\) and is therefore covered by the transversality hypothesis. Were these zeros infinite in number, compactness of \(C\) would furnish a limit point, which lies in \(C\subset\mathcal{D}_{y_0}\) and is therefore again a zero of \(\mathcal{J}(\,\cdot\,;\,y_0)\) by continuity; this would be a non-isolated zero, contradicting the isolatedness established above. Hence the zeros in \(C\) are finite in number; call them \(z_1,\dots,z_m\).

Applying \cref{thm:IFT-Banach} at each \(z_i\) yields pairwise disjoint open neighborhoods \(V_i\) of \(z_i\) (possible since finitely many distinct isolated points) and a common perturbation ball \(B_\delta(y_0)\) such that each \(V_i\) contains a unique zero of \(\mathcal{J}(\,\cdot\,;\,y)\) for every \(y\in B_\delta(y_0)\).
On the compact set \(C\setminus\bigcup_{i=1}^m V_i\), which contains no zeros of \(\mathcal{J}(\,\cdot\,;\,y_0)\), we have \(\|\mathcal{J}(x;\,y_0)\|\ge c\) for some constant \(c>0\).
By continuity of \(\mathcal{J}\) in the product topology, \(\|\mathcal{J}(x;\,y)\|\ge c/2\) for all \(x\in C\setminus\bigcup_{i=1}^m V_i\) and all \(y\) sufficiently close to \(y_0\).
Since \(M\) is continuous and \(x_0\in\intr(C)\), the perturbed solution \(M(y)\) remains in \(\intr(C)\) for \(y\) sufficiently close to \(y_0\).

Combining these estimates and shrinking the perturbation ball if necessary, the following three properties hold for all \(y\) in a neighborhood of \(y_0\):
\begin{enumerate}[label=\textup{(\alph*)},leftmargin=*]
\item \emph{Existing zeros persist.} Each zero \(z_i\) of \(\mathcal{J}(\,\cdot\,;\,y_0)\) in \(C\) gives rise to a unique zero of \(\mathcal{J}(\,\cdot\,;\,y)\) in \(V_i\), and these account for all zeros in \(C\).
\item \emph{No new zeros appear.} On the complement \(C\setminus\bigcup_{i=1}^m V_i\), we have \(\|\mathcal{J}(x;\,y)\|\ge c/2>0\), so no zeros of \(\mathcal{J}(\,\cdot\,;\,y)\) lie outside the neighborhoods \(V_i\).
\item \emph{Non-admissibility is preserved.} By the strict-violation hypothesis, each \(z_i\neq x_0\) satisfies the strict reversal of at least one of the inequalities in the Lax conditions~\eqref{eq:lax-entropy} or the speed chain~\eqref{eq:speed-chain}. Each such reversal is itself an open condition, so the zero persisting in \(V_i\) satisfies the same reversal and remains non-admissible. The zero persisting from \(x_0\) is admissible by part~(B).
\end{enumerate}
In particular, the admissible solution \(x_0\) persists as the unique admissible zero of \(\mathcal{J}(\,\cdot\,;\,y)\) in \(C\).
Together with the structural stability established in parts (A)--(C), this proves the theorem.
\end{proof}

\subsection{Proof of \texorpdfstring{\cref{thm:genericity}}{Theorem~4.3}}

\begin{proof}[Proof of \cref{thm:genericity}]
The wave pattern is fixed in the statement, and the grand objective map \(\mathcal{J}\) is understood with this pattern throughout the proof.

Define the diagonal \(\Delta_{U^2}:=\{(u_l,u_r)\in U^2:u_l=u_r\}\). We exclude it because for $u_l=u_r$ the constant state is the unique entropy solution of small total variation~\cite{Bressan2000}. Since $\Delta_{U^2}$ has Lebesgue measure zero in $U^2$, we restrict to $\mathcal{P}:=U^2\setminus\Delta_{U^2}$ and recover the conclusion on all of $U^2$ in Step~4.
We apply \cref{thm:foliated-parametric-transversality} with this parameter space, parameterized by \(p=(u_l,u_r)\). The flux \(F\) is held fixed throughout the proof, since the theorem asserts genericity in the state pair alone, so the grand objective map is evaluated at the data \((p,F)\) with only \(p\) varying.

\medskip\noindent\textbf{Step 1. Define the foliated domain.}
For each \(p=(u_l,u_r)\), define
\[
  \mathcal{X}_p := \bigl\{\bigl(u^{(1)},\dots,u^{(n-1)},\omega_1,\dots,\omega_n\bigr)\in\mathcal{D}_{(p,F)} : u^{(k)}\neq u^{(k-1)}\text{ for all }k=1,\dots,n\bigr\},
\]
where \(u^{(0)}=u_l\), \(u^{(n)}=u_r\), and \(\mathcal{D}_{(p,F)}:=\{x\in\mathbb{R}^{n^2} : (x;\,(p,F))\in\mathcal{D}\}\) is the open slice of the domain of \cref{lem:C1-objective}, so that \(\mathcal{J}(\,\cdot\,;\,(p,F))\) is defined and \(C^1\) on \(\mathcal{X}_p\), each rarefaction parameter \(\omega_k\) ranging over its admissible interval and each intermediate state over \(\intr(K)\). For \(p\notin\intr(K)^2\) the slice is empty and the conclusion is vacuous.
The conditions \(u^{(1)}\neq u_l\) and \(u^{(n-1)}\neq u_r\) and the slice \(\mathcal{D}_{(p,F)}\) make \(\mathcal{X}_p\) depend on \(p\), so the standard parametric transversality theorem does not apply directly.
\textbf{Define the foliated set.}
$\mathcal{XP} := \bigcup_{p\in\mathcal{P}} \mathcal{X}_p\times\{p\} \;\subset\;\mathbb{R}^{n^2+2n}.$

\medskip\noindent\textbf{Step 2. Verify the hypotheses of \cref{thm:foliated-parametric-transversality}.}
We take \(\mathcal{Y}=\mathbb{R}^{n^2}\), \(\mathcal{Z}=\{0\}\), and the map \(\mathbf{F}\colon\mathcal{XP}\to\mathcal{Y}\) defined by \(\mathbf{F}(x,p)=\mathcal{J}(x;\,(p,F))\) at fixed flux.
\begin{enumerate}[leftmargin=*]
\item \emph{Smoothness threshold.} We need \(r>\max\{0,\,\dim\mathcal{X}+\dim\mathcal{Z}-\dim\mathcal{Y}\}=\max\{0,\,n^2+0-n^2\}=0\). By \cref{lem:C1-objective}, \(\mathcal{J}\) is \(C^1\) in all arguments, so this is satisfied with \(r=1\).
\item \emph{\(\mathcal{XP}\) is a \(C^\infty\) manifold.} The set \(\{(x,p) : (x;\,(p,F))\in\mathcal{D}\}\) is open in \(\mathbb{R}^{n^2+2n}\), being a slice at fixed \(F\) of the open set \(\mathcal{D}\) of \cref{lem:C1-objective}, and each non-degeneracy condition \(u^{(k)}\neq u^{(k-1)}\) removes a closed set, so \(\mathcal{XP}\), the intersection of these sets with the open set \(\mathbb{R}^{n^2}\times\mathcal{P}\), is open in \(\mathbb{R}^{n^2+2n}\) and hence a \(C^\infty\) manifold of dimension \(n^2+2n\).
\item \emph{Tangent splitting.} Since \(\mathcal{XP}\) is open in the product \(\mathbb{R}^{n^2}\times\mathbb{R}^{2n}\), the tangent space splits as \(T_{(x,p)}\mathcal{XP}\cong T_x\mathcal{X}_p\times T_p\mathcal{P}\), with $T_x\mathcal{X}_p\times\{0\}$ the tangent space of the fiber $\mathcal{X}_p\times\{p\}$ and $D\pi_{(x,p)}$ the projection onto the second factor.
\item \emph{\(C^1\) regularity.} The map \((x,p)\mapsto \mathcal{J}(x;\,(p,F))\) is \(C^1\) by \cref{lem:C1-objective}, since \(F\in C^2(K)\) ensures that the composed map is \(C^1\).
\item \emph{Global transversality (\(\mathbf{F}\pitchfork\{0\}\), \cref{def:transversality-map}).} This is verified in Step~3 below.
\end{enumerate}

\medskip\noindent\textbf{Step 3. Check \(\mathbf{F}\pitchfork\{0\}\).}
Since \(\mathcal{Z}=\{0\}\), transversality reduces to surjectivity of the total derivative. We must show that at any zero of the map \((x,p)\mapsto\mathcal{J}(x;\,(p,F))\), the \(n^2\times(n^2+2n)\) Jacobian
\(
  D_{(x,p)}\mathcal{J} = \bigl[\,D_x \mathcal{J} \;\;\big|\;\; D_{u_l}\mathcal{J} \;\;\big|\;\; D_{u_r}\mathcal{J}\,\bigr]
\)
has rank \(n^2\). We exhibit this rank by writing out \(D_x\mathcal{J}\) and the two parameter columns, slotting the parameters in to complete a block bidiagonal staircase, and merging the speed columns to reach a block lower-bidiagonal matrix whose diagonal blocks have full row rank.

Write the \(k\)-th block row of \(\mathcal{J}\) as the derivative of \(W_k(u^{(k)},\omega_k;\,u^{(k-1)})\), with
\[
  A_k=\p_{u^{(k)}}W_k,\qquad B_k=\p_{u^{(k-1)}}W_k,\qquad a_k=\p_{\omega_k}W_k,
\]
and \(u^{(0)}=u_l\), \(u^{(n)}=u_r\); the terminal block row is backward-anchored as in~\eqref{eq:objective-nxn}, and the partial derivatives $A_n=\p_{u^{(n)}}W_n$ and $B_n=\p_{u^{(n-1)}}W_n$ are taken with respect to the same variables regardless of the anchoring. Since \(u^{(k)}\) is the right state of wave \(k\) and the left state of wave \(k+1\), it enters block row \(k\) through \(A_k\) and block row \(k+1\) through \(B_{k+1}\), while the speed \(\omega_k\) enters only block row \(k\), through \(a_k\). In the variables \(x=(u^{(1)},\dots,u^{(n-1)},\omega_1,\dots,\omega_n)\), the state derivative \(D_x\mathcal{J}\) is therefore precisely the matrix \(M\) of \cref{lem:block-reduction},
\[
  D_x\mathcal{J}=\left(\begin{array}{cccc|ccccc}
    A_1 &        &        &         & a_1 &     &        &        &     \\[2pt]
    B_2 & A_2    &        &         &     & a_2 &        &        &     \\[2pt]
        & \ddots & \ddots &         &     &     & \ddots &        &     \\[2pt]
        &        & \ddots & A_{n-1} &     &     &        & a_{n-1}&     \\[2pt]
        &        &        & B_n     &     &     &        &        & a_n
  \end{array}\right),
\]
its state columns \(u^{(1)},\dots,u^{(n-1)}\) carrying the diagonal \(A_1,\dots,A_{n-1}\) and subdiagonal \(B_2,\dots,B_n\), and its speed columns carrying \(\diag(a_1,\dots,a_n)\). The final block row \(W_n\) has no diagonal state block, because its right state \(u^{(n)}=u_r\) is a parameter rather than a component of \(x\); this missing \(A_n\) is what can make \(D_x\mathcal{J}\) singular.

The two parameter columns supply the missing corner blocks,
\[
  D_{u_l}\mathcal{J}=\begin{pmatrix} B_1 \\ 0 \\ \vdots \\ 0 \end{pmatrix},
  \qquad
  D_{u_r}\mathcal{J}=\begin{pmatrix} 0 \\ \vdots \\ 0 \\ A_n \end{pmatrix},
\]
since \(u_l\) is the left state of wave \(1\), contributing \(B_1=\p_{u_l}W_1\) to block row \(1\), and \(u_r\) is the right state of wave \(n\), contributing \(A_n=\p_{u_r}W_n\) to block row \(n\).

Slot \(D_{u_l}\mathcal{J}\) to the left of the state columns and \(D_{u_r}\mathcal{J}\) to their right. Reordered so that \(u_l\) and \(u_r\) flank the intermediate states, the full Jacobian shows a complete block bidiagonal in its state columns beside the diagonal speed block,
\[
  D_{(x,p)}\mathcal{J}=\left(\begin{array}{ccccc|cccc}
    B_1 & A_1 &        &        &     & a_1 &     &        &     \\[2pt]
        & B_2 & A_2    &        &     &     & a_2 &        &     \\[2pt]
        &     & \ddots & \ddots &     &     &     & \ddots &     \\[2pt]
        &     &        & B_n    & A_n &     &     &        & a_n
  \end{array}\right),
\]
each block row \(k\) now holding \(B_k\) and \(A_k\) together. The two parameters have filled the ends of the state diagonal, \(B_1\) from \(u_l\) at the top and \(A_n\) from \(u_r\) at the bottom.

Finally, pair each speed column \(a_k\) with the state column \(u^{(k)}\) that carries \(A_k\) in the same block row. The two merge into a single diagonal block \([A_k\mid a_k]\), and each subdiagonal entry \(B_k\) becomes \([B_k\mid 0]\), giving the block lower-bidiagonal matrix
\[
  \begin{pmatrix}
    {[A_1\mid a_1]} & & & \\
    {[B_2\mid 0]} & {[A_2\mid a_2]} & & \\
     & \ddots & \ddots & \\
     & & {[B_n\mid 0]} & {[A_n\mid a_n]}
  \end{pmatrix}.
\]
The endpoint \(u_l\) has no speed column to pair with, so its block \(B_1\) is left over, forming no \([B_1\mid 0]\) and playing no part.

By Assumption~\ref{ass:regular}, each diagonal block \([A_k\mid a_k]=D_{(u,\omega)}W_k\) has full row rank \(n\), whether or not \(A_k\) itself is invertible, being \(D_{(u,s)}H_{u^{(k-1)}}\) of rank \(n\) on the Hugoniot locus for a shock and having \(A_k=I_n\) for a rarefaction with $k\le n-1$; the non-degeneracy \(u^{(k)}\neq u^{(k-1)}\) built into \(\mathcal{X}_p\) places each state pair in the domain of the assumption.
The terminal block requires its own argument, since the backward anchoring~\eqref{eq:objective-nxn} moves the identity to the $B_n$ slot. For a terminal shock, $[A_n\mid a_n]=-D_{(u,s)}H_{u^{(n-1)}}(u_r,\omega_n)$, which has rank $n$ on the Hugoniot locus by Assumption~\ref{ass:regular}, exactly as before. For a terminal rarefaction, $A_n=-D_{u_r}\Gamma_n(\omega_n;u_r)$ and $a_n=-\alpha_n\,r_n(u^{(n-1)})$ with $\alpha_n\neq0$, so at any zero $[A_n\mid a_n]$ is, up to sign and the factor $\alpha_n$, the augmented matrix of \cref{lem:Gamma-base-rank}, hence has full row rank $n$. A block lower-bidiagonal matrix whose diagonal blocks all have full row rank is itself of full row rank. Hence \(D_{(x,p)}\mathcal{J}\) has rank \(n^2\), and \(\mathbf{F}\pitchfork\{0\}\).

\medskip\noindent\textbf{Step 4. Conclusion.}
By \cref{thm:foliated-parametric-transversality}, for almost every \(p=(u_l,u_r)\in\mathcal{P}=U^2\setminus\Delta_{U^2}\), the restricted map \(\mathcal{J}(\,\cdot\,;\,p)\pitchfork\{0\}\).
Since \(\mathcal{Z}=\{0\}\) is a single point in the target \(\mathbb{R}^{n^2}\), transversality of \(\mathcal{J}(\,\cdot\,;\,p)\) to \(\{0\}\) means that \(D_x\mathcal{J}\) is surjective (and hence invertible, since domain and target have equal dimension \(n^2\)) at every zero of \(\mathcal{J}(\,\cdot\,;\,p)\) in \(\mathcal{X}_p\), that is, at every non-degenerate zero.
This is the transversality condition of Section~\ref{ssec:grand-objective} at every such zero.
Since \(\Delta_{U^2}\) has Lebesgue measure zero in \(U^2\), the conclusion holds for almost every \((u_l,u_r)\in U^2\).
\end{proof}

\begin{remark}[Role of the parameter columns in the genericity proof]\label{rem:parameter-columns}
In the proof of \cref{thm:structural-stability} (structural stability), the transversality condition is assumed, that is, \(D_x\mathcal{J}\) is invertible at the zero, and the implicit function theorem applies directly. In the proof of \cref{thm:genericity} (genericity), however, we work over the full \((x,p)\) space. At an arbitrary zero \((x,p)\), the transversality condition may fail, so \(D_x\mathcal{J}\) alone need not be invertible. The \(D_p\) columns are therefore essential for establishing surjectivity of the \emph{total} derivative \(D_{(x,p)}\mathcal{J}\). The foliated parametric transversality theorem (\cref{thm:foliated-parametric-transversality}) then transfers this global surjectivity to generic surjectivity of \(D_x\mathcal{J}\) alone, recovering the transversality condition for almost every \(p\).
\end{remark}

\begin{remark}[Which end state the proof uses]\label{rem:which-parameter}
Step~3 of the proof of \cref{thm:genericity} uses the $u_r$ column and not the $u_l$ column. This comes from the pairing of each speed column with the right state of its wave, which makes $[A_k\mid a_k]$ the diagonal blocks, so that the last one needs $A_n=\p_{u_r}W_n$ while $B_1=\p_{u_l}W_1$ is left over. Pairing with the left state instead gives diagonal blocks $[B_k\mid a_k]$, of full row rank by an analogous argument, and then $B_1$ is needed and $A_n$ is left over. Either end state therefore suffices, and the theorem is stated for the pair.
\end{remark}

\section{Applications}
\label{sec:applications}

In this section, we illustrate the main theorems with three families of systems.
Section~\ref{ssec:p-system} treats the $p$-system, a classical $2\times 2$ model of compressible isentropic flow in Lagrangian coordinates.
Section~\ref{ssec:PLF} treats gravity-driven particle-laden thin films, where a slurry of $n-1$ distinct particle species leads to an $n \times n$ system and all three main theorems apply once the standing assumptions are verified.
Section~\ref{ssec:ML-flux} treats machine-learned flux approximations, where the structural stability theorems specify how close a learned flux must be to the true one.

\subsection{The p-system}\label{ssec:p-system}

The $p$-system models compressible isentropic flow in Lagrangian coordinates and takes the form
\begin{equation}\label{eq:p-system}
\left\{\;\begin{aligned}
u_t + \bigl(p(v)\bigr)_x &= 0,\\
v_t - u_x &= 0,
\end{aligned}\right.
\end{equation}
with state space $U = \mathbb{R}\times(0,\infty)$ and pressure function $p\in C^2(0,\infty)$.
Here $u$ denotes the velocity and $v$ the specific volume.
Writing~\eqref{eq:p-system} in the form~\eqref{eq:conservation-law} with state variable $(u,v)^T$, the flux is $F(u,v) = \bigl(p(v),\,-u\bigr)^T$ and the Jacobian is given by
\[
A(u,v) = DF(u,v) = \begin{pmatrix} 0 & p'(v) \\ -1 & 0 \end{pmatrix}.
\]
We now verify the three standing assumptions.

\medskip\noindent\textbf{Strict hyperbolicity (Assumption~\ref{ass:SH}).}
The eigenvalues of $A(u,v)$ are $\lambda_{1,2} = \mp\sqrt{-p'(v)}$.
When $p'(v)<0$ on $(0,\infty)$, these are real and distinct, so the system is strictly hyperbolic.
Physically, $p'<0$ follows from thermodynamic considerations \cite{Wendroff_1}.

\medskip\noindent\textbf{Genuine nonlinearity (Assumption~\ref{ass:GN}).}
The right eigenvectors are $r_{1,2} = \bigl({\pm\sqrt{-p'(v)}},\,1\bigr)^T$ and a direct computation gives
\[
\nabla\lambda_k\cdot r_k = \pm\frac{p''(v)}{2\sqrt{-p'(v)}}\neq 0 \qquad\text{for } k=1,2,
\]
provided $p''(v)>0$ on $(0,\infty)$.
This condition holds for a number of physically relevant equations of state \cite{Wendroff_2,Bethe}.

\medskip\noindent\textbf{Regular manifold hypothesis (Assumption~\ref{ass:regular}).}
For a base state $(u_0,v_0)\in U$, the Rankine--Hugoniot map~\eqref{eq:RH-map} reads
\[
H_{(u_0,v_0)}(u,v,s) = \begin{pmatrix} p(v)-p(v_0)-s(u-u_0) \\ -u+u_0-s(v-v_0)\end{pmatrix}.
\]
Its Jacobian with respect to $(u,v,s)$ is given by
\begin{equation}\label{eq:p-system-dH}
D_{(u,v,s)}H_{(u_0,v_0)} = \begin{pmatrix} -s & p'(v) & -(u-u_0) \\ -1 & -s & -(v-v_0) \end{pmatrix}.
\end{equation}
We must show that this $2\times 3$ matrix has rank $2$ at every point of $H_{(u_0,v_0)}^{-1}(0)$ with $(u,v)\neq(u_0,v_0)$.
If $u\neq u_0$, the second component of $H=0$ gives $s(v-v_0) = -(u-u_0)$, so the $2\times 2$ minor formed by columns~1 and~3 of~\eqref{eq:p-system-dH} has determinant $s(v-v_0)-(u-u_0) = -2(u-u_0)\neq 0$.
If $u=u_0$, then $v\neq v_0$, since $(u_0,v_0)$ is excluded from the domain of $H_{(u_0,v_0)}$, and the second component of $H=0$ gives $s(v-v_0)=0$, hence $s=0$; the first component then gives $p(v)=p(v_0)$, which is impossible since $p$ is strictly decreasing. This case is therefore empty.
Hence, $0$ is a regular value of $H_{(u_0,v_0)}$ restricted to $(U\setminus\{(u_0,v_0)\})\times\mathbb{R}$, confirming Assumption~\ref{ass:regular}.

Since the $p$-system satisfies Assumptions~\ref{ass:SH},~\ref{ass:GN}, and~\ref{ass:regular} whenever $p'<0$ and $p''>0$, all three main theorems apply.
By \cref{thm:generic-structural-stability}, for almost every $(u_l,v_l,u_r,v_r)\in U^2$ the Riemann solution is structurally stable under $C^2$-small perturbations to the pressure law $p$ as well as perturbations to the left and right states.
Two remarks on the relation to our prior $2\times 2$ work~\cite{TanBertozzi2x2} are in order. First, that work imposed a graph condition (assumption~(iii) of that paper) on the flux components in order to reparameterize the rarefaction ODE in a fixed Cartesian chart; here the eigenchart construction of Section~\ref{ssec:rarefaction} removes the need for any such condition, and in particular no assumption on the component structure of $F$ enters the $p$-system verification. Second, the regular manifold hypothesis is present in both papers but stated differently. In the prior $2\times 2$ paper it is phrased on a scalar Hugoniot objective function in $(u,v)$ with the shock speed already eliminated, whereas Assumption~\ref{ass:regular} is stated on the lifted Rankine--Hugoniot map $H_{u_0}(u,s)$, and the rank computation above checks surjectivity of $D_{(u,s)}H_{u_0}$ directly on the lifted locus.
The physical interpretation and all applications discussed in~\cite{TanBertozzi2x2} for the $p$-system, including structural stability under interpolation of experimentally tabulated pressure data and stability with respect to measurement errors, remain valid within the present framework.

\subsection{Gravity-driven particle-laden thin films}\label{ssec:PLF}

We consider gravity-driven particle-laden thin films, a class of systems studied extensively in~\cite{Murisic2011,Murisic2013,Liwangshock,Bidensity,luong2025brazilnuteffectbidisperseparticle,2024spiral}.
A thin-film flow of a viscous liquid carrying a suspension of negatively buoyant, rigid, spherical particles is driven down a wide inclined channel.
In our prior work~\cite{TanBertozzi2x2}, a monodisperse suspension (a single particle species) was modeled as a $2\times 2$ system governing the free-surface height $h$ and the depth-averaged particle volume fraction $\phi$.
Here, for $n\geq 2$, we consider a polydisperse suspension consisting of $n{-}1$ distinct particle species (for instance, particles of differing densities or sizes as in~\cite{Bidensity,luong2025brazilnuteffectbidisperseparticle}), so that the state variable
\[
\mathbf{u} = (h,\,h\phi_1,\,\dots,\,h\phi_{n-1})^T \in \mathbb{R}^n
\]
consists of the free-surface height $h>0$ and the conserved quantities $h\phi_i$ for each species, with $\phi_i\in(0,\phi_m)$ denoting the depth-averaged volume fraction of the $i$-th species and $\phi_m$ the maximum packing fraction.
Setting $n=2$ recovers the monodisperse $2\times 2$ system of~\cite{TanBertozzi2x2}.

By conservation of the suspension volume and the particle count for each species, the evolution is governed by the $n\times n$ system of conservation laws
\begin{equation}\label{eq:PLF-nxn}
\begin{aligned}
\p_t h + \p_x \bigl(h^3 f_0(\phi_1,\dots,\phi_{n-1})\bigr) &= 0, \\
\p_t (h\phi_i) + \p_x \bigl(h^3 f_i(\phi_1,\dots,\phi_{n-1})\bigr) &= 0, \qquad i=1,\dots,n{-}1,
\end{aligned}
\end{equation}
with Riemann initial data of the form~\eqref{eq:Riemann-data}.
The flux functions $f_0, f_1,\dots,f_{n-1}$ are obtained by solving a coupled system of nonlinear ODEs in the normal-to-incline direction (generalizing the monodisperse ODE system described in~\cite{Murisic2011,TanBertozzi2x2}) and integrating the resulting velocity and concentration profiles across the film depth.
Writing~\eqref{eq:PLF-nxn} in the form~\eqref{eq:conservation-law}, the flux map is given by
\[
F(\mathbf{u}) = h^3\bigl(f_0(\phi_1,\dots,\phi_{n-1}),\; f_1(\phi_1,\dots,\phi_{n-1}),\;\dots,\; f_{n-1}(\phi_1,\dots,\phi_{n-1})\bigr)^T.
\]
The state space is the open set
\[
U = \bigl\{(h,h\phi_1,\dots,h\phi_{n-1}) \in \mathbb{R}^n \,\big|\, h>0,\;\phi_i\in(0,\phi_m),\; \textstyle\sum_{i=1}^{n-1}\phi_i < \phi_m\bigr\}.
\]

Provided the functions $f_0,\dots,f_{n-1}$ are $C^2$ and the system satisfies Assumptions~\ref{ass:SH},~\ref{ass:GN}, and~\ref{ass:regular} on a suitable open subset of $U$, all three main theorems apply to~\eqref{eq:PLF-nxn}.
In the monodisperse case $n=2$, these conditions were imposed as hypotheses in~\cite{TanBertozzi2x2}, whose results apply once they are verified for specific physical parameters (see also~\cite{Murisic2011,Liwangshock}).
For the polydisperse case $n\geq 3$, strict hyperbolicity and genuine nonlinearity depend on the specific particle interactions and can likewise be checked numerically for a given set of physical parameters.
The regular manifold hypothesis (Assumption~\ref{ass:regular}) is verified by computing the Jacobian $D_{(\mathbf{u},s)}H_{\mathbf{u}_0}$ at each point of the Hugoniot locus and checking surjectivity, as was done for the $p$-system in Section~\ref{ssec:p-system}.

By \cref{thm:generic-structural-stability}, for almost every choice of left and right states, the Riemann solution to~\eqref{eq:PLF-nxn} is structurally stable under $C^2$-small perturbations to the flux functions $f_0,\dots,f_{n-1}$ and to the left and right states.
In particular, the following applications discussed in~\cite{TanBertozzi2x2} for the monodisperse $2\times 2$ system generalize to the polydisperse $n\times n$ setting.
\begin{enumerate}[leftmargin=30pt]
\item[(i)] \textbf{Interpolation of precomputed flux data.}
Since the flux functions are obtained by solving ODEs on a grid of parameter values, a numerical scheme typically interpolates between precomputed grid points. Structural stability guarantees that the Riemann solution structure is preserved under a sufficiently accurate interpolation in the $C^2$ norm.
\item[(ii)] \textbf{Polynomial approximation of flux functions.}
Similarly, if $f_0,\dots,f_{n-1}$ are approximated by piecewise polynomials that are $C^2$-close, the Riemann solution structure persists.
\item[(iii)] \textbf{Comparison of competing physical models.}
If the flux functions obtained from two different modeling frameworks (such as the diffusive flux and suspension balance models~\cite{ugrad_comparative,Fastequilibrium}) are $C^2$-close, the Riemann solutions exhibit the same wave structure.
\item[(iv)] \textbf{Measurement errors in initial data.}
For almost every left and right state, the Riemann solution structure persists under small perturbations to the initial data, providing robustness with respect to experimental measurement errors.
\end{enumerate}

\subsection{Machine-learned flux approximations}\label{ssec:ML-flux}

Recent advances in scientific machine learning have produced a variety of methods that learn or approximate the flux function of a system of conservation laws from data.
The structural stability theorems of Section~\ref{sec:structural-stability} provide a rigorous guarantee for such methods.
If a learning algorithm produces an approximation $\widetilde{F}$ to the true flux $F$ with $\|\widetilde{F}-F\|_{[C^2(K)]^n}$ sufficiently small, and the learned system satisfies Assumptions~\ref{ass:SH},~\ref{ass:GN}, and~\ref{ass:regular}, then by \cref{thm:generic-structural-stability} the Riemann solution structure is preserved for almost every left and right state.
This applies regardless of whether the approximation arises from a neural network, a symbolic regression, or any other data-driven method.
The guarantee covers systems satisfying the standing assumptions, such as the shallow water equations and the $p$-system; non-convex fluxes and the Euler equations, with their linearly degenerate contact field, fall under the extensions discussed in Section~\ref{sec:conclusion}.

We organize the discussion according to the type of object that is learned.

\medskip\noindent\textbf{Learning the physical flux.}
Several methods learn the flux function $F$ itself from solution data.
Patsatzis et al.~\cite{GoRINNs2025} embed a shallow neural network inside a Godunov-type finite volume scheme with an approximate Riemann solver, so that the network learns the physical flux closure $\mathcal{N}(\mathbf{u})$ rather than a discretization-dependent numerical flux.
The Rankine--Hugoniot condition enters the training loss as a penalty term, and because the learned object is a physical flux rather than a discretization-dependent quantity, it is not tied to the grid on which it was trained.
Liu, Zhang, and Gelb~\cite{SymCLaw2026} take a complementary approach in which the flux is parameterized through a pair of neural networks, one representing a strictly convex entropy function (an input-convex network) and the other an entropy flux potential.
Because the flux Jacobian factors as the product of a symmetric matrix and a symmetric positive-definite matrix, hyperbolicity of the learned system, in the sense of real eigenvalues with a complete set of eigenvectors, is guaranteed architecturally rather than by regularization.
In both cases, provided the networks use smooth activation functions, the learned flux $\widetilde{F}$ is a smooth map from $U$ to $\mathbb{R}^n$, so the $C^2$-closeness hypothesis of \cref{thm:generic-structural-stability} is directly meaningful.
The SymCLaw framework is well-suited to the setting of this work, since its architecture builds in hyperbolicity by construction, leaving only the strictness of the eigenvalues to be checked for Assumption~\ref{ass:SH}. Assumptions~\ref{ass:GN} and~\ref{ass:regular} must still be verified on the learned flux.

\medskip\noindent\textbf{Symbolic flux discovery.}
An alternative to neural approximation is to recover a closed-form symbolic expression for the flux from data.
Li and Evje~\cite{ConsLawNet2023,ConsLawNet2D2024} develop a symbolic neural network (S-Net), coupled with an entropy-consistent discretization, that identifies the nonlinear flux function of scalar conservation laws and produces interpretable expressions for both one- and two-dimensional problems.
Because the output is a symbolic formula, the standing assumptions can be checked on the learned flux in closed form, and the $C^2$ distance $\|\widetilde{F}-F\|_{[C^2(K)]^n}$ can be estimated once the true flux is known.

\medskip\noindent\textbf{Learning the numerical flux.}
Rather than learning the physical flux, one may learn the numerical flux $\hat{f}_{j+1/2}$ used at cell interfaces in a finite volume scheme.
Chen, Gelb, and Lee~\cite{CFN2024} propose a conservative form network (CFN) in which the neural network maps a local stencil of cell averages to an intercell numerical flux.
Because the scheme updates cell averages via flux differences, conservation is enforced by architecture, and the Lax--Wendroff theorem guarantees that convergent solutions are weak solutions of the conservation law.
Every consistent numerical flux satisfies $\hat{f}(\mathbf{u},\mathbf{u})=F(\mathbf{u})$, so the learned scheme is associated with an effective physical flux $\widetilde{F}(\mathbf{u})=\hat{f}^{\mathrm{NN}}(\mathbf{u},\dots,\mathbf{u})$.
Provided the scheme converges and $\|\widetilde{F}-F\|_{[C^2(K)]^n}$ is small, \cref{thm:generic-structural-stability} applies to the original system with $\widetilde{F}$ as the perturbed flux, so for almost every left and right state the Riemann solution of the conservation law with flux $\widetilde{F}$ has the same wave structure as that of the original system.

\medskip\noindent\textbf{Learned Riemann solvers.}
A growing body of work replaces exact or approximate Riemann solvers with trained neural networks that take the left and right states $(u_l,u_r)$ as input and return the solution of the Riemann problem, whether as its intermediate states, the associated wave speeds, or directly the interface flux~\cite{MagieraRiemann2020,RuggeriRiemann2022,FluxNet2023,NogueiraRiemann2025,HCNRS2026}.
These methods do not perturb the flux; each is a surrogate for the Riemann map $(u_l,u_r)\mapsto(u^{(1)},\dots,u^{(n-1)},\omega_1,\dots,\omega_n)$ from the left and right states to the intermediate states and wave parameters of the Riemann solution.
When the transversality condition $\det T\neq 0$ holds, the proof of \cref{thm:structural-stability} via the implicit function theorem (\cref{thm:IFT-Banach}) shows that this map is $C^1$ in $(u_l,u_r,F)$, and in particular locally Lipschitz on compact subsets.
For systems that satisfy Assumptions~\ref{ass:SH},~\ref{ass:GN}, and~\ref{ass:regular}, this $C^1$ dependence is one reason why learned Riemann solvers are robust in practice.

\section{Conclusion and Future Work}
\label{sec:conclusion}

This paper establishes generic structural stability for Riemann solutions to $n\times n$ systems of hyperbolic conservation laws under $C^2$ perturbations of the flux and of the left and right states.
The central mechanism is \emph{sequential transversality}, which chains the $n$ wave curves through the intermediate states $u^{(0)}=u_l,\dots,u^{(n)}=u_r$ and reduces structural stability to invertibility of an $n\times n$ transversality matrix $T$, whose columns collect the pushforward-transported wave-curve tangents at the final intermediate state.
Paired with the eigenchart construction for rarefaction curves and the coordinate-free regular manifold hypothesis (Assumption~\ref{ass:regular}) on the lifted Hugoniot locus, this framework removes the dimensional obstruction that prevents the pairwise transversality condition of our prior work from extending beyond $n=2$.
A second methodological contribution is the lifted-space formalism underlying the analysis, which embeds both shock and rarefaction objects as one-dimensional submanifolds of $U\times\mathbb{R}$ by adjoining the shock speed $s$ or the characteristic speed $\xi=\lambda_k$ as an extra coordinate, so that the Lax inequalities become geometric speed-ordering conditions along the staircase trajectory from $u_l$ to $u_r$ rather than auxiliary algebraic constraints on the projected wave curves.

Several directions remain open.
Our analysis operates under Assumptions~\ref{ass:SH} (strict hyperbolicity) and~\ref{ass:GN} (genuine nonlinearity), which together restrict attention to classical Lax shocks and rarefactions and exclude composite waves and contact discontinuities.
Relaxing Assumption~\ref{ass:GN} admits strictly hyperbolic systems with non-convex flux, such as the van der Waals $p$-system in the supercritical regime~\cite{Bethe}, whose inflection points generate composite waves, while relaxing Assumption~\ref{ass:SH} opens the door to systems with coinciding or crossing characteristic speeds.
Extending the sequential-transversality framework to these regimes, with admissibility criteria adapted to each wave type, is a natural next step.

In this paper, Assumption~\ref{ass:GN} is used only for the rarefaction objects of Section~\ref{ssec:rarefaction}, for reparameterizing the rarefaction curve and in Appendix~\ref{app:Gamma-base-rank}. The Hugoniot theory for shocks, the Lax conditions, and the block reduction do not require it. Since the eigencharts of \cref{lem:C1-evec-IFT} persist across the inflection locus $\{\nabla\lambda_k\cdot r_k=0\}$, wave curves for non-convex flux could be built as in Wendroff~\cite{Wendroff_1,Wendroff_2} and Liu~\cite{Liu2x2,Liu_shocks}, with alternating fans and sonic shocks forming at most $n$ wave groups. In that setting, structural stability would be defined for a fixed wave pattern, with one block of~\eqref{eq:objective-nxn} per elementary wave, one attachment equation per sonic junction, and the E-condition of Liu~\cite{Liu2x2,Liu1975} in place of the Lax condition, in the wave-group framework of~\cite{Schecter}. Genericity would again be posed as a parametric transversality statement for the enlarged objective map. The only difficulty we foresee is the regularity of composite wave curves at their attachment points, and resolving it constitutes potential future work.

The framework opens several application-oriented questions.
For learned Riemann solvers, Section~\ref{ssec:ML-flux} shows that $C^2$-close flux approximations preserve wave structure on compact subsets; a concrete next step is to formulate analogous guarantees for neural architectures that directly approximate the Riemann map $(u_l,u_r)\mapsto (u^{(1)},\dots,u^{(n-1)},\omega_1,\dots,\omega_n)$.
For polydisperse particle-laden models, the results here cover the strictly hyperbolic, genuinely nonlinear regime, and a complete structural-stability theory would combine the present analysis with a treatment of the degeneracies and compound-wave solutions that arise in the full physical model.
A related physical question, noted already in our prior work, is that vanishing-viscosity selection may not be the correct regularization limit for particle-laden flows, since surface tension contributes a fourth-order regularizing term and also modifies the fluxes~\cite{SurfaceTension2018}, and such regularizations need not select the same Riemann solutions.
Because the framework refers only to the inviscid flux and the Lax inequalities, and to no particular regularization, one might be able to adapt it to the non-classical waves that other regularizations select by imposing the appropriate admissibility criterion, such as an Oleinik-type chord condition or a kinetic relation. Since these need not be open conditions, the genericity argument, which rests on the strictness of the Lax inequalities, would have to be re-established rather than inherited.

\appendix

\section{Background from differential topology}
\label{app:diff-topo-prelims}

In this appendix, we collect tools from differential topology that are used throughout the paper.  Standard references include Hirsch~\cite{Hirsch}, Lee~\cite{LeeManifold},
and Guillemin--Pollack~\cite{Pollack}.

\begin{definition}[Embedded Submanifold]\label{def:embedded}
Let \(\mathcal{M}\subset\mathbb{R}^N\) be a subset.  We say \(\mathcal{M}\) is a
\emph{\(C^k\) embedded submanifold of dimension \(m\)} if for every
\(p\in \mathcal{M}\) there exist an open set \(W\subset\mathbb{R}^N\) containing
\(p\) and a \(C^k\) map \(\Phi\colon W\to\mathbb{R}^{N-m}\) such that
\begin{enumerate}
\item \(\mathcal{M}\cap W = \Phi^{-1}(0)\), and
\item \(D\Phi(p)\) has rank \(N-m\).
\end{enumerate}
Equivalently, \(\mathcal{M}\) is locally the zero set of a submersion.
\end{definition}

\begin{definition}[Regular Value and Submersion]\label{def:regval}
Let \(W\subset\mathbb{R}^N\) be open and let \(G\colon W\to\mathbb{R}^k\) be \(C^1\).  A point
\(y\in\mathbb{R}^k\) is a \emph{regular value} of \(G\) if
\(DG(x)\) has rank \(k\) (i.e.\ is surjective) for every
\(x\in G^{-1}(y)\).  A \(C^1\) map is a \emph{submersion} at \(x\)
if \(DG(x)\) is surjective.
\end{definition}

\begin{theorem}[Regular Value Theorem]\label{thm:rvt}
Let \(W\subset\mathbb{R}^N\) be open, let \(G\colon W\to\mathbb{R}^k\) be \(C^r\)
(\(r\ge 1\)), and suppose \(y\in\mathbb{R}^k\) is a regular value of
\(G\).  Then \(G^{-1}(y)\) is a \(C^r\) embedded submanifold of
\(\mathbb{R}^N\) of dimension \(N-k\), with tangent space
$T_x\bigl(G^{-1}(y)\bigr) = \ker DG(x)$ for every $x\in G^{-1}(y)$.
\end{theorem}

Here \(DG(x)\colon \mathbb{R}^N\to\mathbb{R}^k\) denotes the derivative (Jacobian)
of \(G\) at \(x\), viewed as a linear map; thus
\(\ker DG(x)=\{v\in\mathbb{R}^N : DG(x)\,v = 0\}\). Furthermore, in dimension-counting terms, each scalar equation imposed by
\(G(x)=y\) ``generically'' reduces the dimension by one, giving
\(\dim G^{-1}(y)=N-k\).

\begin{proposition}[Immersions are local diffeomorphisms onto their images {\cite[Thm.~4.25 and Prop.~5.2]{LeeManifold}, \cite[Ch.~1, \S3]{Hirsch}}]\label{cor:immersion-local}
Let \(M\subset\mathbb{R}^N\) be a \(C^1\) embedded submanifold and let
\(f\colon M\to\mathbb{R}^k\) be a \(C^1\) immersion.
Then for every \(p\in M\) there is a neighborhood \(V\subset M\) of \(p\) such that
\(f(V)\) is a \(C^1\) embedded submanifold of \(\mathbb{R}^k\) and
\(f|_V\colon V\to f(V)\) is a \(C^1\) diffeomorphism.
\end{proposition}

\begin{definition}[Immersions and Embeddings]\label{def:imm-emb}
Let \(\mathcal{M}\) and \(\mathcal{N}\) be \(C^1\) manifolds and let \(f\colon \mathcal{M}\to \mathcal{N}\)
be \(C^1\).
\begin{itemize}
\item \(f\) is an \emph{immersion} if \(Df_p\) is injective for
      every \(p\in \mathcal{M}\).  (Thus a \(C^1\) immersion is simply a \(C^1\) map
      that is an immersion.)
\item \(f\) is a \emph{(topological) embedding} if it is a homeomorphism
      \(\mathcal{M}\to f(\mathcal{M})\) onto its image \(f(\mathcal{M})\) (with the
      subspace topology inherited from $\mathcal{N}$).
\item \(f\) is a \emph{\(C^1\) embedding} if it is a \(C^1\) immersion and a
      (topological) embedding.
\end{itemize}
\end{definition}

An injective immersion need not be an embedding in general. For example,
\(\beta\colon(-\pi,\pi)\to\mathbb{R}^2\), \(\beta(t)=(\sin 2t,\sin t)\), is an
injective \(C^\infty\) immersion whose image accumulates at \(\beta(0)\) as
\(t\to\pm\pi\), so \(\beta\) fails to be a homeomorphism onto its image.

\begin{theorem}[Image of a \(C^1\) embedding {\cite[Prop.~5.2]{LeeManifold}, \cite[Thm.~1.3.1]{Hirsch}}]\label{thm:embedding-image}
Let \(\mathcal{M}\) and \(\mathcal{N}\) be \(C^1\) manifolds, and let \(f\colon \mathcal{M}\to \mathcal{N}\) be a \(C^1\)
embedding. Then \(f(\mathcal{M})\) is a \(C^1\) embedded submanifold of \(\mathcal{N}\), and
\(f\colon \mathcal{M}\to f(\mathcal{M})\) is a \(C^1\) diffeomorphism. In particular,
\(\dim f(\mathcal{M})=\dim \mathcal{M}\).
\end{theorem}

\begin{remark}[Local versus global embedding]\label{rem:local-vs-global}
Proposition~\ref{cor:immersion-local} is the local counterpart of
Theorem~\ref{thm:embedding-image}. An immersion is automatically a local
\(C^1\) embedding near each point of its domain, so locally the image is always
an embedded submanifold. However, global conclusions require additional
work. The example after Definition~\ref{def:imm-emb} shows that an
injective immersion can fail to be a topological embedding, so
Proposition~\ref{cor:immersion-local} alone does not guarantee that the
\emph{entire} image is a submanifold. Upgrading to a global \(C^1\)
embedding (and hence to Theorem~\ref{thm:embedding-image}) typically
requires showing that the map is a homeomorphism onto its image, e.g.\ by
constructing a continuous (or \(C^1\)) inverse explicitly.
\end{remark}

\section{Proofs of auxiliary lemmas}
\label{app:aux-proofs}

\subsection{Proof of \texorpdfstring{\cref{prop:equv-hugoniot}}{Proposition~2.6}}
\label{app:equiv-hugoniot-proof}

\begin{proof}[Proof of Proposition~\ref{prop:equv-hugoniot}]
By Proposition~\ref{prop:lifted-mfld}, \(\widehat{\mathcal{H}}_{u_0}\) is a
one-dimensional \(C^1\) embedded submanifold.  By
Lemma~\ref{lem:immersion}, \(\pi|_{\widehat{\mathcal{H}}_{u_0}}\) is an
immersion, and by Lemma~\ref{lem:injectivity} it is injective. To show that $\pi\big|_{\widehat{\mathcal{H}}_{u_0}} \colon \widehat{\mathcal{H}}_{u_0} \to U\setminus\{u_0\}$ is a $C^1$ embedding, it remains to show that the inverse is \(C^1\).

Let \(u\in\mathcal{H}_{u_0}=\pi(\widehat{\mathcal{H}}_{u_0})\).  By
Lemma~\ref{lem:injectivity}, there is a unique \(s\in\mathbb{R}\) such that
\((u,s)\in\widehat{\mathcal{H}}_{u_0}\); define
\(\sigma(u):=s\), so that
\[(\pi|_{\widehat{\mathcal{H}}_{u_0}})^{-1}(u)=(u,\sigma(u)).\]
Fix \((u_*,s_*)\in\widehat{\mathcal{H}}_{u_0}\).  Since \(u_*\neq u_0\), choose an
index \(j\in\{1,\dots,n\}\) such that \((u_* - u_0)_j\neq 0\).  By continuity,
there is a neighborhood \(V\subset U\setminus\{u_0\}\) of \(u_*\) on which
\((u-u_0)_j\neq 0\).  For any \(u\in V\cap \mathcal{H}_{u_0}\), the
Rankine--Hugoniot condition \eqref{eq:rankine-hugoniot} gives
\(F_j(u)-F_j(u_0)=\sigma(u)\,(u-u_0)_j\), hence
\[
  \sigma(u)=\frac{F_j(u)-F_j(u_0)}{(u-u_0)_j}.
\]
Since \(F\in C^2\), the right-hand side is \(C^1\) on \(V\), so \(\sigma\) is \(C^1\)
locally on \(\mathcal{H}_{u_0}\).  On overlaps of such neighborhoods the formulas
agree by uniqueness of \(\sigma\), hence \(\sigma\) (and therefore
\((\pi|_{\widehat{\mathcal{H}}_{u_0}})^{-1}\)) is globally \(C^1\).
Hence \(\pi|_{\widehat{\mathcal{H}}_{u_0}}\) is a \(C^1\) embedding. By
Theorem~\ref{thm:embedding-image}, its image
\(\mathcal{H}_{u_0}=\pi(\widehat{\mathcal{H}}_{u_0})\) is a \(C^1\) embedded
submanifold of \(U\setminus\{u_0\}\). Since
\(\dim \widehat{\mathcal{H}}_{u_0}=1\) by Proposition~\ref{prop:lifted-mfld}, the
submanifold \(\mathcal{H}_{u_0}\) is one-dimensional.
\end{proof}

\subsection{Line fields, sections, and the Grassmannian}
\label{app:eigenline-geometry}

This subsection collects the geometric vocabulary behind the eigenline construction of Section~\ref{ssec:rarefaction}, together with the picture that motivates it. Nothing beyond the definitions is used in the proofs, which work entirely with the local representatives produced by \cref{lem:C1-evec-IFT}.

A \emph{line field} over an open set $W\subset\mathbb{R}^n$ assigns to each $u\in W$ a one-dimensional subspace $E(u)\subset\mathbb{R}^n$, called the \emph{fiber} over $u$. A \emph{local section} on an open set $\Omega\subset W$ is a map $r\colon\Omega\to\mathbb{R}^n\setminus\{0\}$ with $r(u)\in E(u)$ for every $u\in\Omega$; it selects a nonzero representative of each fiber, and the line field is called $C^1$ when every point admits a local $C^1$ section. If $r^\alpha$ and $r^\beta$ are local $C^1$ sections with overlapping domains, both span $E(u)$ at each common state, so $r^\beta(u)=c_{\alpha\beta}(u)\,r^\alpha(u)$ for a nonvanishing \emph{transition scalar} $c_{\alpha\beta}$. The $C^1$ regularity of $c_{\alpha\beta}$ is not automatic from one-dimensionality alone, but follows by taking components. Near any state, choose an index $i$ with $(r^\alpha)_i\neq 0$; then $c_{\alpha\beta}=(r^\beta)_i/(r^\alpha)_i$ is a ratio of nonvanishing $C^1$ functions. For the sections of \cref{lem:C1-evec-IFT} the normalization $(r^\alpha)_{i_\alpha}=1$ makes this explicit, $c_{\alpha\beta}=(r^\beta)_{i_\alpha}$, as computed at the end of the proof in Appendix~\ref{app:evec-IFT-proof}.

The set of all one-dimensional subspaces of $\mathbb{R}^n$ is the Grassmannian $\operatorname{Gr}_1(\mathbb{R}^n)$, equivalently the projective space $\mathbb{RP}^{n-1}$, the quotient of $\mathbb{R}^n\setminus\{0\}$ by the relation $r\sim cr$ for $c\neq 0$. A line field is precisely a map $E\colon W\to\mathbb{RP}^{n-1}$, and the $C^1$ property above says exactly that this map is $C^1$ into the manifold $\mathbb{RP}^{n-1}$. The quotient discards both the length and the sign of a representative, which is the precise sense in which the eigenline $E_k(u)$ is canonical although no preferred eigenvector exists. The unit normalization $|r_k|=1$, for instance, still leaves the two-fold choice $\pm r_k$ in every fiber, and a continuous global sign need not exist unless the line field is orientable; parameterizing lines rather than vectors removes the sign from the outset.

The standard atlas of $\mathbb{RP}^{n-1}$ consists of affine charts, the $i$-th of which contains the lines with nonzero $i$-th coordinate, each represented by its unique point on the affine hyperplane $\{x_i=1\}$. The normalization $r_i=1$ in \cref{lem:C1-evec-IFT} selects exactly this representative of $E_k(u)$, so the eigencharts of that lemma are the pullbacks under $E_k$ of the standard charts of projective space, and the transition scalars $c_{\alpha\beta}$ implement the corresponding changes of chart. \Cref{fig:eigenline-bundle} illustrates two overlapping eigencharts and their transition scalar.

\begin{figure}[t]
  \centering
  \begin{tikzpicture}[>=Stealth, scale=1.0,
      every node/.style={font=\small}]

    \draw[thick] (0,0) ellipse (5.6cm and 2.4cm);
    \node[anchor=west] at (5.7,0.0) {$U$};

    \draw[rounded corners=6pt, fill=blue!5, draw=blue!50, thick]
      (-4.5,-1.25) rectangle (1.0,1.25);
    \node[blue!70!black, anchor=north west]
      at (-4.45,1.20) {$\Omega^{\alpha}$};
    \node[blue!70!black, anchor=south west, font=\scriptsize]
      at (-4.45,-1.20) {$(r_{k}^{\alpha})_{i_\alpha}=1$};

    \draw[rounded corners=6pt, fill=red!5, draw=red!50, thick]
      (-1.0,-1.25) rectangle (4.5,1.25);
    \node[red!70!black, anchor=north east]
      at (4.45,1.20) {$\Omega^{\beta}$};
    \node[red!70!black, anchor=south east, font=\scriptsize]
      at (4.45,-1.20) {$(r_{k}^{\beta})_{i_\beta}=1$};

    \draw[rounded corners=6pt, fill=violet!12, draw=violet!55,
          thick, opacity=0.85]
      (-1.0,-1.25) rectangle (1.0,1.25);
    \node[violet!55!black, font=\scriptsize, anchor=south]
      at (0.0,1.27) {$\Omega^{\alpha}\cap\Omega^{\beta}$};

    \coordinate (u)      at (-0.10, 0.20);
    \coordinate (uTangL) at ($(u)+(-1.05,-0.16)$);
    \coordinate (uTangR) at ($(u)+(1.05, 0.16)$);

    \coordinate (u0)  at (-3.80, 0.30);
    \coordinate (uRt) at ( 4.20, 0.30);
    \draw[thick]
      (u0) .. controls (-2.40, 0.10) and (uTangL) .. (u)
           .. controls (uTangR) and (2.60, 0.35) .. (uRt);

    \node[font=\small, anchor=south] at (-2.50, 0.05) {$\mathcal{R}_{k,u_0}$};

    \draw[fill=white, thick] (u0) circle (2.3pt);
    \node[above=2pt] at (u0) {$u_0$};

    \draw[black!75, dashed, line width=0.7pt]
      ($(u)+(-1.30,-0.195)$) -- ($(u)+(1.85, 0.278)$);

    \node[black!75, font=\scriptsize, anchor=south west]
      at ($(u)+(1.55, 0.30)$) {$E_k(u)$};

    \draw[->, very thick, blue!70!black]
      (u) -- ($(u)+(1.10, 0.165)$);
    \node[blue!70!black, font=\scriptsize, anchor=south]
      at ($(u)+(0.85, 0.20)$) {$r_{k}^{\alpha}(u)$};

    \draw[->, very thick, red!70!black]
      (u) -- ($(u)+(0.55, 0.0825)$);
    \node[red!70!black, font=\scriptsize, anchor=north]
      at ($(u)+(0.40, 0.02)$) {$r_{k}^{\beta}(u)$};

    \draw[fill=black] (u) circle (1.8pt);
    \node[below left=-2pt and 1pt, font=\small] at (u) {$u$};

    \node[font=\scriptsize, align=center, anchor=north,
          draw=violet!55, fill=white, rounded corners=2pt,
          inner sep=3pt]
      at (0.0,-2.75)
      {$r_{k}^{\beta}(u)=c_{\alpha\beta}(u)\,r_{k}^{\alpha}(u)$
       \quad on $\Omega^{\alpha}\cap\Omega^{\beta}$,\\[-1pt]
       with $c_{\alpha\beta}\in C^{1}$ and $c_{\alpha\beta}(u)\ne 0$};

  \end{tikzpicture}
  \caption{Local sections of the eigenline bundle on two overlapping
    eigencharts \(\Omega^{\alpha}\) and \(\Omega^{\beta}\), both contained in
    \(U\). At any state \(u\in\Omega^{\alpha}\cap\Omega^{\beta}\) the eigenline
    \(E_k(u)\) is one-dimensional, so the two IFT-produced representatives
    \(r_k^\alpha(u)\) and \(r_k^\beta(u)\) lie on a common line and differ by a
    nonvanishing \(C^1\) scalar \(c_{\alpha\beta}(u)\). The rarefaction curve
    \(\mathcal R_{k,u_0}\) based at \(u_0\) is an integral curve of any such
    local representative, hence tangent to \(E_k(u)\) at each of its
    states.}
  \label{fig:eigenline-bundle}
\end{figure}
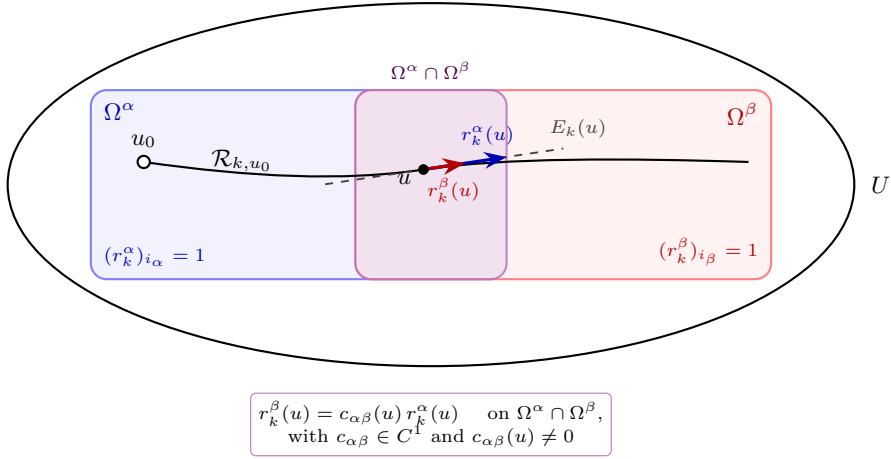

\subsection{Proof of \texorpdfstring{\cref{lem:C1-evec-IFT}}{Lemma~2.12}}
\label{app:evec-IFT-proof}

\begin{proof}[Proof of \cref{lem:C1-evec-IFT}]
We enforce the coordinate normalization \emph{as an equation} as follows. Define $\mathcal{G}: V \times \mathbb{R}^{n+1} \rightarrow \mathbb{R}^{n+1}$ as
\[
\mathcal{G}(u,\lambda,r):=
\begin{pmatrix}
(\mathcal{A}(u)-\lambda I_n)r\\
r_i-1
\end{pmatrix}\in\mathbb{R}^{n+1}.
\]
This map is \(C^1\) because \(\mathcal{A}\in C^1(V)\) and \(r\) is an \emph{independent variable} (not yet a function of \(u\)), and the operations \((u,\lambda,r)\mapsto (\mathcal{A}(u)-\lambda I_n)r\) and \(r\mapsto r_i\) are \(C^1\).

Set \(\mathcal{A}_0:=\mathcal{A}(u_0)\) and \(r_*:=r_0/(r_0)_i\), so that \((r_*)_i=1\) and \(\mathcal{G}(u_0,\lambda_0,r_*)=0\). The Jacobian of \(\mathcal{G}\) with respect to the unknowns \((\lambda,r)\) at the base point is the square \((n+1)\times(n+1)\) matrix
\[
  D_{(\lambda,r)}\mathcal{G}(u_0,\lambda_0,r_*)
  =
  \begin{pmatrix}
    -r_* & (\mathcal{A}_0-\lambda_0 I_n)\\
    0 & e_i^{\!T}
  \end{pmatrix},
\]
where \(e_i\in\mathbb{R}^n\) is the \(i\)-th standard basis vector. To check invertibility, suppose \((\delta\lambda,\delta r)\) lies in the kernel. Then
\[
  (\mathcal{A}_0-\lambda_0 I_n)\,\delta r = (\delta\lambda)\,r_*,\qquad
  (\delta r)_i=0.
\]
Let \(l_0\) be a left eigenvector of \(\mathcal{A}_0\) for \(\lambda_0\), which exists since \(\lambda_0\) is also an eigenvalue of \(\mathcal{A}_0^T\). Left-multiplying the first equation by \(l_0^{\!T}\) yields
\[
  0=l_0^{\!T}(\mathcal{A}_0-\lambda_0 I_n)\,\delta r=(\delta\lambda)\,l_0^{\!T}r_*.
\]
Since \(\lambda_0\) is simple, the associated left and right eigenvectors are not orthogonal, so \(l_0^{\!T}r_*\neq 0\), and we conclude \(\delta\lambda=0\). Hence \((\mathcal{A}_0-\lambda_0 I_n)\delta r=0\), so simplicity of \(\lambda_0\) implies \(\delta r\in\spn\{r_*\}\), and the coordinate constraint forces \(\delta r=0\) (because \((r_*)_i=1\) and \((\delta r)_i = 0\)). Thus the kernel is trivial, the square Jacobian is invertible, and the Implicit Function Theorem yields a neighborhood \(\Omega\subset V\) of \(u_0\) and \(C^1\) functions \(\lambda(u)\), \(r(u)\) solving \(\mathcal{G}(u,\lambda(u),r(u))=0\) for \(u\in\Omega\).

Since \(\lambda_0\) is simple, it is isolated in the spectrum of \(\mathcal{A}_0\); by continuity of \(\mathcal{A}\) and \(\lambda(\cdot)\) and continuity of the spectrum, after shrinking \(\Omega\) the value \(\lambda(u)\) is the unique eigenvalue of \(\mathcal{A}(u)\) near \(\lambda_0\), and it is simple. In particular, for \(\mathcal{A}=DF\) under Assumption~\ref{ass:SH}, the distinct continuous eigenvalue branches give \(\lambda(u)=\lambda_k(u)\) on \(\Omega\).

It remains to justify the overlap statement in the lemma. The IFT step above produces, for each \(u_0\in V\), an eigenchart \(\Omega^\alpha\ni u_0\) and a nonvanishing \(C^1\) section \(r_k^\alpha\colon\Omega^\alpha\to\mathbb{R}^n\setminus\{0\}\) of the eigenline \(E_k(u)=\ker(\mathcal{A}(u)-\lambda(u)I_n)\), normalized by \((r_k^\alpha)_{i_\alpha}=1\) for some index \(i_\alpha\) with \((r_0)_{i_\alpha}\neq 0\). If \(\Omega^\alpha\) and \(\Omega^\beta\) are two such eigencharts (in general with different normalization indices \(i_\alpha,i_\beta\)) and \(u\in\Omega^\alpha\cap\Omega^\beta\), then each fiber \(E_k(u)\) is one-dimensional, so
\[
  r_k^\beta(u)=c_{\alpha\beta}(u)\,r_k^\alpha(u),
  \qquad c_{\alpha\beta}(u)\ne 0.
\]
Reading off the \(i_\alpha\)-th component gives \(c_{\alpha\beta}(u)=(r_k^\beta(u))_{i_\alpha}\), which is \(C^1\) and nonvanishing on the overlap, proving the final assertion of the lemma.
\end{proof}

\begin{remark}[Left representatives]\label{rem:left-representative}
Only the right representative \(r_k\) enters the rarefaction construction. Should a left representative be needed, the same argument applied to \(\mathcal{A}(u)^T\), with normalization \(l_j=1\) for an index \(j\) with \((l_0)_j\neq 0\), produces \(C^1\) maps \(\mu(u)\), \(l(u)\) with \(l(u)^T\mathcal{A}(u)=\mu(u)\,l(u)^T\), and the uniqueness of the eigenvalue near \(\lambda_0\) gives \(\mu=\lambda\) on the common chart.
\end{remark}

\subsection{Proof of \texorpdfstring{\cref{prop:Gamma-maximal}}{Proposition~2.14}}
\label{app:Gamma-maximal-proof}

\begin{proof}[Proof of \cref{prop:Gamma-maximal}]
\emph{Step 1 (Local pieces from the \(t\)-flow).}
Fix \(u_*\in \operatorname{int}(K)\). By \cref{lem:C1-evec-IFT}, there is an eigenchart \(\Omega_*\subset \operatorname{int}(K)\) containing \(u_*\) on which \(r_k\in C^1(\Omega_*)\), and by \cref{prop:rk-maximal} the ODE \(\dot u=r_k(u)\) with \(u(0)=u_*\) has a unique maximal solution \(\phi\). Define \(\Xi(t):=\lambda_k(\phi(t))\). Then
\[
  \Xi'(t)=\nabla\lambda_k(\phi(t))\cdot r_k(\phi(t)),
\]
which is nonzero by genuine nonlinearity. Hence the inverse function theorem gives a local \(C^1\) inverse \(t=t(\xi)\) near \(\xi_*:=\lambda_k(u_*)\). Setting
\[
  \Gamma_{k,*}(\xi):=\phi(t(\xi))
\]
gives a \(C^1\) curve with \(\Gamma_{k,*}(\xi_*)=u_*\) and \(\lambda_k(\Gamma_{k,*}(\xi))=\xi\). The chain rule gives
\[
  \Gamma_{k,*}'(\xi)
  =
  \frac{r_k(\Gamma_{k,*}(\xi))}
  {\nabla\lambda_k(\Gamma_{k,*}(\xi))\cdot r_k(\Gamma_{k,*}(\xi))}.
\]
Differentiating the inverse relation \(\Xi(t(\xi))=\xi\) and integrating from \(\xi_*\) gives the explicit chartwise formula
\begin{equation}\label{eq:t-of-xi}
  t(\xi)=\int_{\xi_*}^{\xi}\frac{d\eta}{\nabla\lambda_k(\Gamma_{k,*}(\eta))\cdot r_k(\Gamma_{k,*}(\eta))},
\end{equation}
valid as long as the curve remains in the eigenchart \(\Omega_*\).

\emph{Step 2 (Uniqueness on overlapping charts).}
Let \((J^{(1)},\Gamma^{(1)})\) and \((J^{(2)},\Gamma^{(2)})\) be two local \(\xi\)-curves satisfying \eqref{eq:Gamma-lambda-param}--\eqref{eq:Gamma-xi-ODE}, and suppose \(\Gamma^{(1)}(\xi_*)=\Gamma^{(2)}(\xi_*)\) at some \(\xi_*\in J^{(1)}\cap J^{(2)}\). Set \(J:=J^{(1)}\cap J^{(2)}\). The agreement set
\[
  S:=\{\xi\in J:\Gamma^{(1)}(\xi)=\Gamma^{(2)}(\xi)\}
\]
is nonempty and closed. To see it is open, fix \(\bar\xi\in S\), set \(\bar u:=\Gamma^{(1)}(\bar\xi)=\Gamma^{(2)}(\bar\xi)\), choose an eigenchart \(\Omega_{\bar u}\ni\bar u\), and restrict to a subinterval where both curves remain in \(\Omega_{\bar u}\). With the same local representative \(r_k\), define the analogues of \eqref{eq:t-of-xi} based at \(\bar\xi\),
\[
  \tau_i(\xi):=\int_{\bar\xi}^{\xi}
  \frac{d\eta}{\nabla\lambda_k(\Gamma^{(i)}(\eta))\cdot r_k(\Gamma^{(i)}(\eta))},
  \qquad i=1,2.
\]
After shrinking the subinterval, each \(\tau_i\) is a \(C^1\) diffeomorphism onto a neighborhood of \(0\), and \(u_i(t):=\Gamma^{(i)}(\tau_i^{-1}(t))\) satisfies \(\dot u_i=r_k(u_i)\), \(u_i(0)=\bar u\). Picard--Lindel\"of uniqueness on \(\Omega_{\bar u}\) gives \(u_1\equiv u_2=:u\) near \(t=0\). Evaluating \(u=u_i\) at \(t=\tau_i(\xi)\) gives \(u(\tau_i(\xi))=\Gamma^{(i)}(\xi)\), and applying \(\lambda_k\) with \eqref{eq:Gamma-lambda-param} yields \(\lambda_k(u(\tau_i(\xi)))=\xi\). Hence \(\tau_1\) and \(\tau_2\) both invert the map \(t\mapsto\lambda_k(u(t))\), which is strictly monotone near \(t=0\) since its derivative \(\nabla\lambda_k\cdot r_k\) is continuous and nonvanishing, so \(\tau_1=\tau_2\), and \(\Gamma^{(i)}=u\circ\tau_i\) gives \(\Gamma^{(1)}\equiv\Gamma^{(2)}\) near \(\bar\xi\). Thus \(S\) is open, and since \(J\) is connected, \(S=J\).

\emph{Step 3 (Gluing and maximality).}
Let \(\mathcal{F}\) denote the collection of open intervals \(J\ni\xi_0\) supporting a curve \(\Gamma_J\in C^1(J;\operatorname{int}(K))\) with \(\Gamma_J(\xi_0)=u_0\) satisfying \eqref{eq:Gamma-lambda-param}--\eqref{eq:Gamma-xi-ODE} chartwise. Step~1 gives \(\mathcal{F}\neq\varnothing\), and Step~2 gives agreement on overlaps. Set
\[
  I:=\bigcup_{J\in\mathcal{F}}J.
\]
This is an open interval containing \(\xi_0\). For \(\xi\in I\), choose \(J\in\mathcal{F}\) with \(\xi\in J\) and define \(\Gamma_k(\xi;u_0):=\Gamma_J(\xi)\); Step~2 makes this independent of \(J\). The resulting curve is \(C^1\) and satisfies the asserted initial condition, parameterization, and derivative formula. Maximality follows from the union construction, since any interval supporting such a curve belongs to \(\mathcal{F}\) and is therefore contained in \(I\), and uniqueness of the maximal pair follows from Step~2 and maximality.

\emph{Step 4 (Behavior at a finite endpoint).}
It remains to show that at any finite endpoint the curve leaves every compact subset of \(\operatorname{int}(K)\). Suppose for example that \(\xi_+<\infty\) and that the conclusion fails, so that there are a compact set \(C\subset\operatorname{int}(K)\) and a sequence \(\xi_j\uparrow\xi_+\) with \(\Gamma_k(\xi_j;u_0)\in C\); we derive a contradiction by extending the curve past \(\xi_+\). Choose a compact set \(C'\subset\operatorname{int}(K)\) with \(C\subset\operatorname{int}(C')\) and set \(d:=\operatorname{dist}(C,\mathbb{R}^n\setminus\operatorname{int}(C'))>0\). The right-hand side of \eqref{eq:Gamma-xi-ODE} does not depend on the representative, so by \cref{lem:C1-evec-IFT} and genuine nonlinearity it is a continuous vector field on \(\operatorname{int}(K)\), and its norm is bounded by some \(M\) on \(C'\). Fix \(j\) with \(M(\xi_+-\xi_j)<d\). If the curve left \(\operatorname{int}(C')\) at some first \(\eta\in(\xi_j,\xi_+)\), integrating \eqref{eq:Gamma-xi-ODE} over \([\xi_j,\eta]\) would give \(d\le|\Gamma_k(\eta;u_0)-\Gamma_k(\xi_j;u_0)|\le M(\eta-\xi_j)<d\). Hence the tail \(\Gamma_k([\xi_j,\xi_+);u_0)\) lies in \(C'\) and is Lipschitz, so it has a limit \(u_*\) as \(\xi\uparrow\xi_+\). Since \(\Gamma_k(\xi_j;u_0)\in C\) and \(C\) is closed, \(u_*\in C\subset\operatorname{int}(K)\), and \(\lambda_k(u_*)=\xi_+\) by continuity.

Fix an eigenchart \(\Omega_{u_*}\subset\operatorname{int}(K)\) containing \(u_*\) with representative \(r_k\); after replacing \(r_k\) by \(-r_k\) if necessary (again a \(C^1\) section, leaving \eqref{eq:Gamma-xi-ODE} unchanged), assume \(\nabla\lambda_k\cdot r_k>0\) on \(\Omega_{u_*}\). Since \(\Gamma_k(\xi;u_0)\to u_*\), we may choose \(\epsilon>0\) so that \(\Gamma_k(\xi;u_0)\in\Omega_{u_*}\) for \(\xi\in[\xi_+-\epsilon,\xi_+)\); this tail together with its limit \(u_*\) is a compact subset of \(\Omega_{u_*}\), so \(\nabla\lambda_k\cdot r_k\geq c>0\) along it. We now pass to the \(t\)-flow, continue it through \(u_*\), and convert back to \(\xi\). First, the time change \(t(\xi)\) of \eqref{eq:t-of-xi}, based at \(\xi_+-\epsilon\) and taken along the tail, is strictly increasing, and since its integrand lies in \((0,1/c]\) it increases to a finite limit \(T\) as \(\xi\uparrow\xi_+\). Writing \(\xi(t)\) for its inverse, the curve \(u(t):=\Gamma_k(\xi(t);u_0)\) solves \(\dot u=r_k(u)\) on \([0,T)\) by \eqref{eq:Gamma-xi-ODE} and the chain rule, and \(u(t)\to u_*\) as \(t\uparrow T\). Second, by \cref{prop:rk-maximal} the flow through \(u_*\) exists on \((-\delta,\delta)\) for some \(\delta>0\), so we extend \(u\) by \(u(T):=u_*\) and \(u(t):=\phi_k(t-T;u_*)\) for \(T\le t<T+\delta\). The extended curve is continuous and solves \(\dot u=r_k(u)\) on either side of \(T\), hence satisfies the integral equation \(u(t)=u(0)+\int_0^t r_k(u(s))\,ds\) on \([0,T+\delta)\) and is a \(C^1\) solution passing through \(u_*\) at time \(T\). Third, \(\lambda_k(u(t))\) is strictly increasing and equals \(\xi_+\) at \(t=T\), so it exceeds \(\xi_+\) for \(t>T\), and inverting \(t\mapsto\lambda_k(u(t))\) as in Step~1 reparameterizes \(u\) into a \(\xi\)-curve on \((\xi_+-\epsilon,\xi_++\delta')\) for some \(\delta'>0\), which returns the original tail for \(\xi<\xi_+\) since \(u(t(\xi))=\Gamma_k(\xi;u_0)\) by construction. The union of \(\Gamma_k(\cdot\,;u_0)\) and this curve is admissible on \((\xi_-,\xi_++\delta')\), so this interval belongs to \(\mathcal{F}\) and is therefore contained in \(I=(\xi_-,\xi_+)\), a contradiction.
\end{proof}

\subsection{Proof of \texorpdfstring{\cref{lem:C1-objective}}{Lemma~4.10}}
\label{app:C1-proof}

Since the grand objective map~\eqref{eq:extended-objective} is the vertical stack of the individual wave maps \(W_1,\dots,W_n\), it suffices to show that each \(W_k\) is \(C^1\) in all its arguments, \((u^{(k)}, \omega_k, u^{(k-1)}, F)\) for \(k\le n-1\) and \((u^{(n-1)}, \omega_n, u_r, F)\) for the backward-anchored terminal map $W_n$ of~\eqref{eq:objective-nxn}. The terminal map $W_n$ shares the two functional forms treated in Parts~A and~B below, with \(u^{(n-1)}\) in the running slot and \(u_r\) in the base slot. Both parts establish joint \(C^1\) regularity in the running state, the wave parameter, the base state, and the flux, so \(W_n\) needs no separate argument.
Throughout this appendix, we make explicit the dependence on the flux \(F\in\mathcal{O}\subset[C^2(K)]^n\) that was suppressed in Sections~\ref{sec:geometric} and~\ref{sec:transversality}.
The key observation is that the eigendata \(\lambda_k, r_k\) and the rarefaction flow \(\Gamma_k\) all depend on \(F\) through the Jacobian matrix \(A(u) = DF(u)\), and this dependence is \(C^1\) in the \([C^2(K)]^n\) topology.
The argument is local, and both assertions of the lemma follow from a single construction. Around each point of \(\mathcal{D}\), Parts~A and~B below produce a neighborhood in the ambient product on which every wave map \(W_k\) is defined and \(C^1\) in all its arguments, the flux \(F\) included. Such a neighborhood lies in \(\mathcal{D}\), so \(\mathcal{D}\) is open; and since each \(W_k\) is \(C^1\) there, so is the stack \(\mathcal{J}\). For rarefactions the neighborhood, and in particular its extent in the \(F\)-direction, is furnished by Steps~1--3, while openness needs only that each \(W_k\) is defined there.

\medskip\noindent\textbf{Part A: Shock waves.}
If the \(k\)-th wave is a shock, then
\[
  W_k(u^{(k)},s_k;\,u^{(k-1)},F) = F(u^{(k)}) - F(u^{(k-1)}) - s_k\,(u^{(k)} - u^{(k-1)}).
\]
The partial Fr\'echet derivatives are
\begin{align}
  D_{u^{(k)}} W_k &= DF(u^{(k)}) - s_k I_n, \label{eq:app-DuH} \\
  D_{s_k} W_k &= -(u^{(k)} - u^{(k-1)}), \label{eq:app-DsH} \\
  D_{u^{(k-1)}} W_k &= -DF(u^{(k-1)}) + s_k I_n, \label{eq:app-Du0H} \\
  D_F W_k(\delta F) &= \delta F(u^{(k)}) - \delta F(u^{(k-1)}), \label{eq:app-DFH}
\end{align}
for \(\delta F \in [C^2(K)]^n\).
Each partial derivative exists and is jointly continuous in the arguments \((u^{(k)}, s_k, u^{(k-1)}, F)\).
For~\eqref{eq:app-DuH} and~\eqref{eq:app-Du0H}, this follows from the joint continuity of the evaluation map \((u, F) \mapsto DF(u)\) on \(\intr(K) \times [C^2(K)]^n\), since \(F \in C^2\) implies \(DF \in C^1(K)\), and small perturbations in \(u\) or \(F\) produce small changes in the matrix \(DF(u)\).
The derivative~\eqref{eq:app-DsH} is clearly continuous in $u^{(k)}$ and $u^{(k-1)}$, and~\eqref{eq:app-DFH} is a bounded linear operator from \([C^2(K)]^n\) to \(\mathbb{R}^n\) with operator norm at most~\(2\), depending continuously on the evaluation points \((u^{(k)}, u^{(k-1)})\).
Hence \(W_k\) is \(C^1\) in all arguments.

\medskip\noindent\textbf{Part B: Rarefaction waves.}
If the \(k\)-th wave is a rarefaction, then
\[
  W_k(u^{(k)},\lambda_k^*;\,u^{(k-1)},F) = u^{(k)} - \Gamma_k(\lambda_k^*;\,u^{(k-1)},F),
\]
where \(\Gamma_k(\xi;\,u_0,F)\) is the \(\xi\)-parameterized rarefaction curve from Proposition~\ref{prop:Gamma-maximal}, now with flux dependence made explicit.
Since \(u^{(k)}\) enters only as an additive term, it suffices to show that \(\Gamma_k(\xi;\,u_0,F)\) is \(C^1\) in \((\xi, u_0, F)\).
Throughout Part~B, \(C^1\) regularity is a local claim, established at a base triple \((\xi_*,u_{0,*},F_*)\) in which \(\xi_*\) lies in the maximal \(\xi\)-interval of the rarefaction curve based at \(u_{0,*}\) for the flux \(F_*\).
We proceed in three steps.

\medskip\noindent\textit{Step~1 (Eigendata are \(C^1\) in \((u,F)\) jointly).}
In \cref{lem:C1-evec-IFT}, the eigendata \((\lambda_k, r_k)\) were shown to be \(C^1\) functions of \(u\) at fixed flux.
To incorporate the flux as a parameter, define the extended map \(\widehat{\mathcal{G}}\colon \intr(K) \times \mathcal{O} \times \mathbb{R}^{n+1} \to \mathbb{R}^{n+1}\) by
\[
  \widehat{\mathcal{G}}(u, F, \lambda, r) :=
  \begin{pmatrix}
    (DF(u) - \lambda I_n)\,r \\[2pt]
    r_i - 1
  \end{pmatrix},
\]
which coincides with the map~\(\mathcal{G}\) from the proof of \cref{lem:C1-evec-IFT}, now with the flux \(F\) promoted to an additional argument.
The map \(\widehat{\mathcal{G}}\) is \(C^1\) in all its arguments jointly, since it depends on \((u,F)\) only through \(DF(u)\), and the evaluation map
\begin{equation}\label{eq:app-eval}
  (u,F) \mapsto DF(u) \;\colon\; \intr(K) \times [C^2(K)]^n \to \mathbb{R}^{n \times n}
\end{equation}
is \(C^1\).
Indeed, \(D_u(DF(u)) = D^2\!F(u)\) exists and is continuous since \(F \in C^2\). In the flux direction, \(D_F(DF(u))(\delta F) = D(\delta F)(u)\) is linear in \(\delta F\), with
\(\|D(\delta F)(u)\| \leq \|\delta F\|_{[C^2(K)]^n}\).
The Jacobian \(D_{(\lambda,r)}\widehat{\mathcal{G}}\) is the square matrix computed in \cref{lem:C1-evec-IFT} and is therefore invertible at the base point.
By the implicit function theorem (with parameters \((u,F)\) and unknowns \((\lambda,r)\)), there exist neighborhoods of any given \((u_0, F_0)\) and \(C^1\) maps
\[
  (u,F) \mapsto \bigl(\lambda_k(u;\,F),\; r_k(u;\,F)\bigr)
\]
satisfying the eigenvalue equation with the prescribed normalization.
The coordinate index \(i\) is fixed only on this neighborhood. If the rarefaction segment passes to another eigenchart, the same construction is repeated there with a possibly different index; agreement on overlaps follows from the one-dimensional eigenspace and the normalization/scaling invariance used in \cref{prop:Gamma-maximal}.

\medskip\noindent\textit{Step~2 (The \(t\)-flow is \(C^1\) in \((t, u_0, F)\)).}
On each eigenchart \(\Omega \subset \intr(K)\), the \(t\)-flow satisfies the ODE
\begin{equation}\label{eq:app-t-flow}
  \dot{\phi}_k(t;\,u_0,F) = r_k\bigl(\phi_k(t;\,u_0,F);\,F\bigr), \qquad \phi_k(0;\,u_0,F) = u_0.
\end{equation}
By Step~1, the vector field \((u,F) \mapsto r_k(u;\,F)\) is \(C^1\) on \(\Omega \times \mathcal{O}\).
After shrinking the neighborhood of \((t_*,u_{0,*},F_*)\), all trajectories under consideration remain in a compact subset of the same eigenchart \(\Omega\) for the time interval used below.
Classical ODE theory gives \(C^1\) dependence of the flow on \((t,u_0)\) for a \(C^1\) (hence locally Lipschitz) vector field.
It remains to verify \(C^1\) dependence on the infinite-dimensional parameter \(F\).

Fix \((u_0, F_0)\) and write \(u(t) = \phi_k(t;\,u_0,F_0)\).
Integrating~\eqref{eq:app-t-flow} gives
\begin{equation}\label{eq:app-integral}
  \phi_k(t;\,u_0,F) = u_0 + \int_0^t r_k\bigl(\phi_k(\tau;\,u_0,F);\,F\bigr)\;\D\tau.
\end{equation}
Setting \(Y(t) := D_F\phi_k(t;\,u_0,F_0)(\delta F)\), formally applying the chain rule for Fr\'echet derivatives under the integral in~\eqref{eq:app-integral} and differentiating with respect to \(t\) yields the variational equation
\begin{equation}\label{eq:app-variational}
  \left\{\begin{aligned}
    \dot{Y}(t) &= M(t)\,Y(t) + B(t), \\
    Y(0) &= 0,
  \end{aligned}\right.
\end{equation}
with
\begin{equation}\label{eq:app-AB}
  M(t) := D_u r_k\bigl(u(t);\,F_0\bigr), \qquad
  B(t) := D_F r_k\bigl(u(t);\,F_0\bigr)(\delta F).
\end{equation}
This is a linear ODE in \(\mathbb{R}^n\) with continuous coefficient \(M(t)\in\mathbb{R}^{n\times n}\) and continuous forcing \(B(t)\in\mathbb{R}^n\).
By Step~1, both \(D_u r_k(u;\,F_0)\) and \(D_F r_k(u;\,F_0)\) are continuous in \(u\).
Since the trajectory \(u(t)\) remains in a compact subset of \(\Omega\) for \(t\) in any bounded interval, there exist constants \(L > 0\) and \(C_1 > 0\) such that
\[
  \|M(t)\| = \|D_u r_k(u(t);\,F_0)\| \leq L, \qquad
  \|D_F r_k(u(t);\,F_0)\|_{\mathrm{op}} \leq C_1
\]
uniformly along the trajectory, where \(\|\cdot\|_{\mathrm{op}}\) denotes the operator norm from \([C^2(K)]^n\) to \(\mathbb{R}^n\).
In particular, \(\|B(t)\| = \|D_F r_k(u(t);\,F_0)(\delta F)\| \leq C_1\,\|\delta F\|_{[C^2(K)]^n}\).
Integrating~\eqref{eq:app-variational} and taking norms, for \(t\ge0\) (the case \(t\le0\) is analogous, integrating over \([t,0]\)),
\begin{equation}\label{eq:app-integral-ineq}
  \|Y(t)\| \leq \int_0^t \bigl[L\,\|Y(\tau)\| + C_1\,\|\delta F\|_{[C^2(K)]^n}\bigr]\;\D\tau.
\end{equation}
The Gronwall inequality then yields
\begin{equation}\label{eq:app-Gronwall}
  \|Y(t)\| \leq \frac{C_1}{L}\bigl(e^{L|t|} - 1\bigr)\,\|\delta F\|_{[C^2(K)]^n}.
\end{equation}
This shows that \(D_F\phi_k(t;\,u_0,F_0)\) is a bounded linear map from \([C^2(K)]^n\) to \(\mathbb{R}^n\).

So far, \(Y(t)\) is only a candidate for the Fr\'echet derivative, obtained by formally interchanging \(D_F\) with the integral in~\eqref{eq:app-integral}.
To verify that \(Y(t) = D_F\phi_k(t;\,u_0,F_0)(\delta F)\) rigorously, we check the definition of Fr\'echet differentiability directly.
Write \(\tilde{u}(t) := \phi_k(t;\,u_0,F_0{+}\delta F)\), so that \(\tilde{u}\) satisfies~\eqref{eq:app-t-flow} with \(F_0{+}\delta F\) in place of~\(F\), and define the remainder
\[
  w(t) := \tilde{u}(t) - u(t) - Y(t), \qquad w(0) = 0.
\]
Then
\[
  \dot{w}(t)
  = r_k(\tilde{u};\,F_0{+}\delta F) - r_k(u;\,F_0) - M(t)\,Y(t) - B(t).
\]
Since \(\tilde{u}\) and \(u\) solve the flow~\eqref{eq:app-t-flow} for the fluxes \(F_0{+}\delta F\) and \(F_0\), subtracting the two equations and applying Gronwall's inequality gives \(\|\tilde{u}(t)-u(t)\| = O(\|\delta F\|)\) uniformly on the time interval under consideration.
We expand \(\dot{w}\) into its linear part plus three remainders in a single step, inserting the first-order terms \(r_k(\tilde{u};F_0)\), \(D_u r_k(u;F_0)(\tilde{u}{-}u)\), and \(D_F r_k(\tilde{u};F_0)(\delta F)\), which gives
\begin{equation}\label{eq:app-remainder}
\begin{aligned}
  \dot{w}(t)
  &= \underbrace{D_u r_k(u;F_0)\,w(t)}_{\textrm{(I)}}
   + \underbrace{\bigl[r_k(\tilde{u};F_0) - r_k(u;F_0) - D_u r_k(u;F_0)(\tilde{u}{-}u)\bigr]}_{\textrm{(II)}} \\
  &\quad + \underbrace{\bigl[r_k(\tilde{u};F_0{+}\delta F) - r_k(\tilde{u};F_0) - D_F r_k(\tilde{u};F_0)(\delta F)\bigr]}_{\textrm{(III)}} \\
  &\quad + \underbrace{\bigl[D_F r_k(\tilde{u};F_0) - D_F r_k(u;F_0)\bigr](\delta F)}_{\textrm{(IV)}}.
\end{aligned}
\end{equation}
We treat the four terms in turn.
\begin{itemize}[leftmargin=*]
\item[\textrm{(I)}] The linear part \(M(t)\,w(t)\), with \(\|M(t)\|\le L\).
\item[\textrm{(II)}] The first-order Taylor remainder of \(u\mapsto r_k(u;F_0)\) at \(u=u(t)\). By uniform continuity of \(D_u r_k\) on the compact trajectory it is \(o(\|\tilde{u}-u\|)\) uniformly in \(t\), hence \(o(\|\delta F\|)\) since \(\|\tilde{u}-u\|=O(\|\delta F\|)\).
\item[\textrm{(III)}] The first-order Taylor remainder of \(F\mapsto r_k(\tilde{u};F)\) at \(F_0\), namely
\[
  \int_0^1\bigl[D_F r_k(\tilde{u};F_0{+}s\delta F)-D_F r_k(\tilde{u};F_0)\bigr](\delta F)\,\D s .
\]
Uniform continuity of \(D_F r_k\) makes it \(o(\|\delta F\|)\) uniformly in \(\tilde{u}\).
\item[\textrm{(IV)}] Bounded by \(\|D_F r_k(\tilde{u};F_0)-D_F r_k(u;F_0)\|_{\mathrm{op}}\,\|\delta F\|\). Since \(\tilde{u}\to u\) as \(\delta F\to0\) and \(D_F r_k\) is uniformly continuous, the operator norm tends to \(0\), so the term is \(o(\|\delta F\|)\).
\end{itemize}
Combining, \(\|\dot{w}(t)\| \leq L\,\|w(t)\| + o(\|\delta F\|)\) with \(w(0)=0\), and Gronwall's inequality gives \(\|w(t)\| = o(\|\delta F\|)\).
That is, \(\phi_k(t;\,u_0,F_0{+}\delta F) = \phi_k(t;\,u_0,F_0) + Y(t) + o(\|\delta F\|)\), which is the definition of Fr\'echet differentiability with derivative \(Y(t) = D_F\phi_k(t;\,u_0,F_0)(\delta F)\).

It remains to assemble the partial derivatives into joint \(C^1\) regularity in \((t,u_0,F)\).
We first note that \(\phi_k\) is itself jointly continuous in \((t,u_0,F)\). Indeed, the vector field \(r_k(\,\cdot\,;F)\) depends continuously on \(F\), uniformly on the compact trajectory, so the flow depends continuously on the parameter \(F\) by the standard continuous-dependence theorem, and continuously on \((t,u_0)\) by classical theory.
Granting this, each first-order partial of \(\phi_k\) exists and is jointly continuous in \((t,u_0,F)\).
\begin{itemize}[leftmargin=*]
\item \(\p_t\phi_k(t;u_0,F)=r_k(\phi_k(t;u_0,F);F)\) is jointly continuous, being a composition of the continuous maps \(\phi_k\) and \(r_k\).
\item \(D_{u_0}\phi_k\) solves the variational equation \(\dot Z = D_u r_k(\phi_k(t;u_0,F);F)\,Z\) with \(Z(0)=I_n\), whose coefficient is jointly continuous in \((t,u_0,F)\) by Step~1, so its solution \(Z\) is jointly continuous as well.
\item \(D_F\phi_k=Y\) solves the variational equation~\eqref{eq:app-variational}, with coefficients \(D_u r_k\) and \(D_F r_k\) evaluated along the running trajectory \(\phi_k(t;u_0,F)\); these are jointly continuous in \((t,u_0,F)\) by Step~1, so \(Y\) is jointly continuous as well.
\end{itemize}
Since all first-order partials of \(\phi_k\) exist and are jointly continuous in \((t,u_0,F)\), the flow \(\phi_k(t;u_0,F)\) is \(C^1\) in \((t,u_0,F)\) jointly.

\medskip\noindent\textit{Step~3 (Reparameterization to \(\xi\) preserves \(C^1\) regularity).}
Define the characteristic speed along the flow,
\begin{equation}\label{eq:app-xi-t}
  \Xi(t,u_0,F) := \lambda_k\bigl(\phi_k(t;\,u_0,F);\,F\bigr),
\end{equation}
which is \(C^1\) in \((t,u_0,F)\) by Steps~1 and~2. Its derivative in \(t\) is given by
\[
  \p_t\Xi = \nabla\lambda_k\bigl(\phi_k(t;\,u_0,F);\,F\bigr)\cdot r_k\bigl(\phi_k(t;\,u_0,F);\,F\bigr),
\]
which is nonzero at the base point because \(\p_t\Xi=\nabla\lambda_k\cdot r_k\) is precisely the genuine-nonlinearity quantity \(\nabla\lambda_k\cdot r_k\neq0\) of Assumption~\ref{ass:GN}.
The parameter-dependent implicit function theorem therefore gives, on a neighborhood of the base triple, a unique \(C^1\) solution \(t=t(\xi;u_0,F)\) of the equation \(\Xi(t,u_0,F)=\xi\).
Therefore,
\[
  \Gamma_k(\xi;\,u_0,F) = \phi_k\bigl(t(\xi;\,u_0,F);\,u_0,F\bigr)
\]
is \(C^1\) in \((\xi,u_0,F)\) as a composition of \(C^1\) maps.

This argument is local to each eigenchart, giving \(C^1\) regularity of \(\Gamma_k\) while \(u_0\) and the target parameter \(\xi\) stay in a single chart.
A rarefaction segment used by a wave map may cross several charts, so we extend the conclusion by following the curve chart by chart.

Fix the base triple \((\xi_*,u_{0,*},F_*)\) and consider the compact subarc of the unperturbed curve from \(u_{0,*}\) out to parameter \(\xi_*\). Cover it by finitely many eigencharts \(\Omega_1,\dots,\Omega_m\) and choose parameter values
\[
  \eta_0=\lambda_k(u_{0,*})<\eta_1<\cdots<\eta_m=\xi_*
\]
so that the piece of the arc with \(\xi\in[\eta_{j-1},\eta_j]\) lies in \(\Omega_j\). By the continuous dependence established above, the same charts and cutoffs still work for all \((u_0,F)\) near the base triple, after shrinking the neighborhood.

Define the states reached at these successive parameter values,
\[
  z_0(u_0,F):=u_0, \qquad z_j(u_0,F):=\Gamma_k\bigl(\eta_j;\,z_{j-1}(u_0,F),\,F\bigr) \quad (j=1,\dots,m),
\]
where each \(\Gamma_k(\eta_j;\,\cdot\,,\,\cdot\,)\) is computed within the single chart \(\Omega_j\). Because the rarefaction curve is a single integral curve of the eigenline field, continuing it from a point it already passes through retraces the same curve, and by uniqueness of the maximal curve (\cref{prop:Gamma-maximal}),
\[
  \Gamma_k(\xi;\,u_0,F)=\Gamma_k\bigl(\xi;\,z_j(u_0,F),\,F\bigr) \qquad (\xi\ge\eta_j).
\]
In particular each \(z_j\) is the state at parameter \(\eta_j\) on the single curve through \(u_0\).

By the single-chart result of Step~3, each map \((w,F)\mapsto\Gamma_k(\eta_j;\,w,\,F)\) is \(C^1\), so the states \(z_j\) are \(C^1\) in \((u_0,F)\) by induction on \(j\). The base state \(z_0=u_0\) is trivially \(C^1\), and \(z_j\) is the \(C^1\) chart-\(\Omega_j\) map applied to the \(C^1\) state \(z_{j-1}\). Finally, for \(\xi\) in the last interval \([\eta_{m-1},\xi_*]\),
\[
  \Gamma_k(\xi;\,u_0,F)=\Gamma_k\bigl(\xi;\,z_{m-1}(u_0,F),\,F\bigr)
\]
is \(C^1\) in \((\xi,u_0,F)\) as a composition of \(C^1\) maps, and the same holds on each earlier interval. Hence \((\xi,u_0,F)\mapsto\Gamma_k(\xi;u_0,F)\) is \(C^1\) on a neighborhood of the base triple.
If instead \(\xi_*<\lambda_k(u_{0,*})\), as for the backward-anchored terminal rarefaction, whose parameter \(\omega_n=\lambda_n(u^{(n-1)})\) lies below the base value \(\lambda_n(u_r)\), the same construction applies verbatim with a decreasing chain \(\eta_0=\lambda_k(u_{0,*})>\eta_1>\cdots>\eta_m=\xi_*\) and the retracing identity for \(\xi\le\eta_j\); nothing in Steps~1--3 distinguishes the two directions of traversal, the flow simply being traversed in the direction of decreasing \(\xi\).
This completes Part~B, and together with Part~A, each wave map \(W_k\) is \(C^1\) in all its arguments. \qed

\subsection{Rank of the base-state derivative of the rarefaction curve}\label{app:Gamma-base-rank}

The genericity proof of \cref{thm:genericity} differentiates the backward-anchored terminal rarefaction condition of~\eqref{eq:objective-nxn} in the datum $u_r$, which occupies the base slot of $\Gamma_n$, and Step~3 of that proof needs the ranks recorded below. The derivative itself exists and is continuous by \cref{lem:C1-objective}, established above. The result is also cited in Section~\ref{ssec:transversality-nxn} for the singularity of the interior rarefaction pushforwards and in \cref{rem:backward-anchoring}.

\begin{lemma}[Base-state derivative of the rarefaction curve]\label{lem:Gamma-base-rank}
Fix $u_0\in\operatorname{int}(K)$ and $\xi\in I$, and write $u_\xi:=\Gamma_k(\xi;u_0)$. The derivative $D_{u_0}\Gamma_k(\xi;u_0)\in\mathbb{R}^{n\times n}$ of the map $u\mapsto\Gamma_k(\xi;u)$ has rank exactly $n-1$, and the augmented matrix $[\,D_{u_0}\Gamma_k(\xi;u_0)\mid r_k(u_\xi)\,]$ has full row rank $n$.
\end{lemma}

\begin{proof}
The two arguments of $\Gamma_k$ play different roles. The base state selects the curve, since $\Gamma_k(\,\cdot\,;u)$ traverses the maximal integral curve through $u$, and the parameter selects the point on it, since the parameterization~\eqref{eq:Gamma-lambda-param} normalizes $\lambda_k(\Gamma_k(\xi;u))=\xi$, making $\xi$ an absolute eigenvalue level rather than a displacement from the base state. Along any single integral curve the level determines the point uniquely, since the parameterization~\eqref{eq:Gamma-lambda-param} uses the level itself as the curve parameter.

\emph{Range.} At fixed $\xi$, the map $u\mapsto\Gamma_k(\xi;u)$ takes all its values in the level hypersurface $\{\lambda_k=\xi\}$, by the normalization, so its derivative maps into the tangent hyperplane of this hypersurface at the image point. Concretely, differentiating $\lambda_k(\Gamma_k(\xi;u))=\xi$ with respect to $u$ at $u_0$ gives $\nabla\lambda_k(u_\xi)^{T}\,D_{u_0}\Gamma_k(\xi;u_0)=0$, the right-hand side being independent of $u$. Every column of the derivative is thus orthogonal to $\nabla\lambda_k(u_\xi)$, the range lies in the hyperplane $\nabla\lambda_k(u_\xi)^{\perp}$, and the rank is at most $n-1$.

\emph{Kernel.} Moving the base state along the curve itself does not move the image point. To see this, replace $u_0$ by the slid base point $\phi_k(s;u_0)$ for small $s$. The curve is unchanged, since the maximal integral curve through $\phi_k(s;u_0)$ coincides with the one through $u_0$ by the uniqueness in \cref{prop:Gamma-maximal}, and the target level is unchanged, since the first slot still reads $\xi$. Both sides of
\[
  \Gamma_k\bigl(\xi;\,\phi_k(s;u_0)\bigr)=\Gamma_k(\xi;u_0)
\]
therefore ask for the same point, the unique point of the common curve at level $\xi$, and the identity holds for all small $s$. Differentiating in $s$ at $s=0$, where the base point moves with velocity $\p_s\phi_k(s;u_0)\big|_{s=0}=r_k(u_0)$, gives $D_{u_0}\Gamma_k(\xi;u_0)\,r_k(u_0)=0$, so $r_k(u_0)$ lies in the kernel.

\emph{Rank, one eigenchart.} Suppose first that the segment of the curve between $u_0$ and $u_\xi$ lies in a single eigenchart. There the image point is reached by flowing from the base point for a definite time, $\Gamma_k(\xi;u)=\phi_k(t(\xi,u);u)$, where the flow time $t(\xi,u)$ solves $\lambda_k(\phi_k(t;u))=\xi$. The function $\Xi(t)=\lambda_k(\phi_k(t;u))$ is strictly monotone with $\Xi'=\nabla\lambda_k\cdot r_k\neq 0$ by Assumption~\ref{ass:GN}, and the implicit function theorem then makes $t(\xi,u)$ a $C^1$ function of both arguments. The chain rule differentiates through both slots of $\phi_k$, and the derivative is given by
\[
  D_{u_0}\Gamma_k(\xi;u_0)
  = r_k(u_\xi)\,\bigl(\nabla_u t(\xi,u_0)\bigr)^{T} + D_u\phi_k\bigl(t(\xi,u_0);u_0\bigr),
\]
the first term coming from the time slot, where $\p_t\phi_k=r_k$ at the image point, and the second from the base slot. The second term is invertible, being the derivative of a flow map at fixed time, and the first is a rank-one matrix, so their sum has rank at least $n-1$. Combined with the range inclusion above, the rank equals exactly $n-1$, and a dimension count upgrades both inclusions to equalities, so the range is the whole hyperplane $\nabla\lambda_k(u_\xi)^{\perp}$ and the kernel, one-dimensional and containing $r_k(u_0)$, is exactly $\operatorname{span}\,r_k(u_0)$.

\emph{Rank, several eigencharts.} If the segment traverses several eigencharts, use compactness of the segment between $u_0$ and $u_\xi$ to choose finitely many intermediate levels $\xi_0=\lambda_k(u_0),\,\xi_1,\dots,\xi_m=\xi$, monotone from $\xi_0$ to $\xi$, with each subsegment between consecutive levels lying in a single eigenchart. Reaching level $\xi$ from $u$ directly or by way of the intermediate level $\xi_{m-1}$ lands at the same point of the same curve, by the same uniqueness as in the kernel paragraph, so $\Gamma_k(\xi;u)=\Gamma_k\bigl(\xi;\Gamma_k(\xi_{m-1};u)\bigr)$ wherever both sides are defined. Both sides are defined and $C^1$ for all $u$ in a neighborhood of $u_0$, the chartwise chaining argument in Part~B of the proof of \cref{lem:C1-objective} keeping each perturbed subsegment in its eigenchart, so the identity may be differentiated at $u_0$. Write $v_j:=\Gamma_k(\xi_j;u_0)$ for the intermediate points, with $v_0=u_0$ and $v_m=u_\xi$, and let $D_j$ denote the one-chart derivative of $\Gamma_k(\xi_j;\,\cdot\,)$ at $v_{j-1}$, so that iterating the identity above splits the derivative into the product
\[
  D_{u_0}\Gamma_k(\xi;u_0)=D_m\,D_{m-1}\cdots D_1,
\]
where each $D_j$ has rank $n-1$, range $\nabla\lambda_k(v_j)^{\perp}$, and kernel $\operatorname{span}\,r_k(v_{j-1})$ by the one-chart case. For any matrix product, rank--nullity applied to $B$ restricted to $\operatorname{range}A$ gives
\[
  \operatorname{rank}(BA)=\operatorname{rank}(A)-\dim\bigl(\operatorname{range}A\cap\ker B\bigr),
\]
so rank is lost exactly where the range of a partial product meets the kernel of the next factor in a nonzero vector, and here this never happens. By induction on $j$, the partial product $P_j:=D_j\cdots D_1$ has rank $n-1$ and range $\nabla\lambda_k(v_j)^{\perp}$. The base case $j=1$ is the one-chart case. For the step, the intersection $\operatorname{range}P_j\cap\ker D_{j+1}=\nabla\lambda_k(v_j)^{\perp}\cap\operatorname{span}\,r_k(v_j)$ is trivial by Assumption~\ref{ass:GN}, so $\operatorname{rank}P_{j+1}=n-1$, and the range of $P_{j+1}$, an $(n-1)$-dimensional subspace of the $(n-1)$-dimensional hyperplane $\nabla\lambda_k(v_{j+1})^{\perp}$, equals it. The full derivative $P_m$ therefore has rank $n-1$, and its kernel, one-dimensional and containing $\ker D_1=\operatorname{span}\,r_k(u_0)$, is exactly that line.

\emph{Augmented rank.} The columns of the derivative span the hyperplane $\nabla\lambda_k(u_\xi)^{\perp}$, as just established, and the appended column $r_k(u_\xi)$ is transverse to this hyperplane, again because $\nabla\lambda_k(u_\xi)\cdot r_k(u_\xi)\neq 0$ by Assumption~\ref{ass:GN}. The columns of the augmented matrix therefore span all of $\mathbb{R}^n$, and the matrix has full row rank $n$.
\end{proof}

\section{Note on AI-assisted preparation}
\label{app:machine-assisted}

This appendix describes how we used AI assistance and shows three artifacts from the process. The disclosure itself, together with the models and the dates, is in the Acknowledgments.

We used it in four ways.
\begin{enumerate}
  \item \textbf{Transcription.} We wrote the mathematics out by hand, then prompted the models to type the notes into \LaTeX. We checked every transcribed formula against the notes, which are the source of record, and prompted again where it came back wrong.
  \item \textbf{Figures.} We sketched each diagram by hand and wrote down the conventions it had to follow. The model turned the sketch into Ti\emph{k}Z source. We refined it by prompting, then checked the result against the mathematics it depicts.
  \item \textbf{Review.} We passed complete drafts through the models and asked for technical errors, notational inconsistencies, and poor writing. We judged each report ourselves, and acted on a technical one only after re-deriving the point.
  \item \textbf{Discussion.} We put arguments we had already drafted to the models and asked them to object.
\end{enumerate}
In each case we set the direction and supplied the mathematics, and the model supplied source, prose, or objections. We give three examples below, two for the figures and one for transcription and review.

\cref{fig:handwritten-double-rare} is the sketch behind \cref{fig:double-rare-example}. It already carries the layout of the published figure, the lifted state space above the $(x,t)$-plane and a plain panel beside an annotated one. We fixed the drawing conventions in advance, one colour per wave family, set positions for axes and speed labels, and a set layer ordering, so that the diagrams in the paper would match one another. We settled the miscellaneous differences by prompting, and checked the finished figure against the mathematics.

\begin{figure}[htbp]
  \centering
  \includegraphics[width=0.88\textwidth]{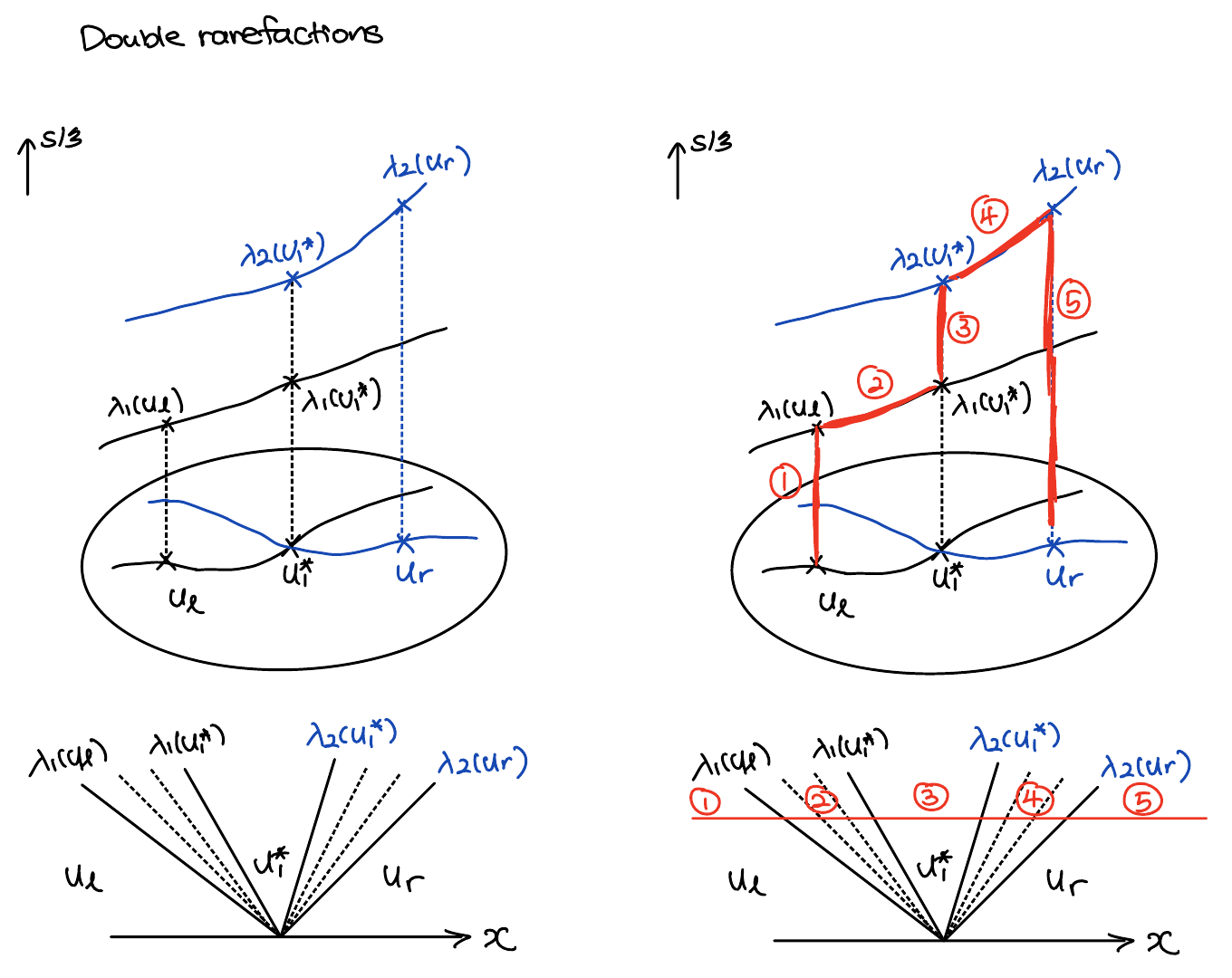}
  \caption{Our handwritten sketch for the $2\times 2$ double-rarefaction Riemann solution. \emph{Top:} the two rarefaction branches in the lifted state space, plain on the left and carrying the trajectory on the right. \emph{Bottom:} the two rarefaction fans in the $(x,t)$-plane, with the constant-$t$ slice and its five numbered regions on the right. \cref{fig:double-rare-example} is the redraw, which keeps this layout and adds the red termini at $u_l$ and $u_r$ with dashed stems locating them in the state space below.}
  \label{fig:handwritten-double-rare}
\end{figure}

\cref{fig:handwritten-seq-trans} is the sketch behind \cref{fig:sequential-transversality}, where the same loop took a few more turns. The left panel has the surface swept out by the 2-Hugoniot curves branching off the 1-Hugoniot locus, and the right panel has the pushforward as a map between tangent spaces. The sketch and the redraw orient the pushforward differently, the sketch sending tangent vectors at $u_2^*$ to tangent vectors at $u_1^*$ and \cref{fig:sequential-transversality} sending them the other way. The maps are inverse to each other and both are defined at an admissible zero (Section~\ref{ssec:transversality-3x3}), so neither is wrong, and we settled on the direction of~\eqref{eq:pushforward-hugoniot} after a few rounds of prompting and checking.

\begin{figure}[htbp]
  \centering
  \includegraphics[width=0.92\textwidth]{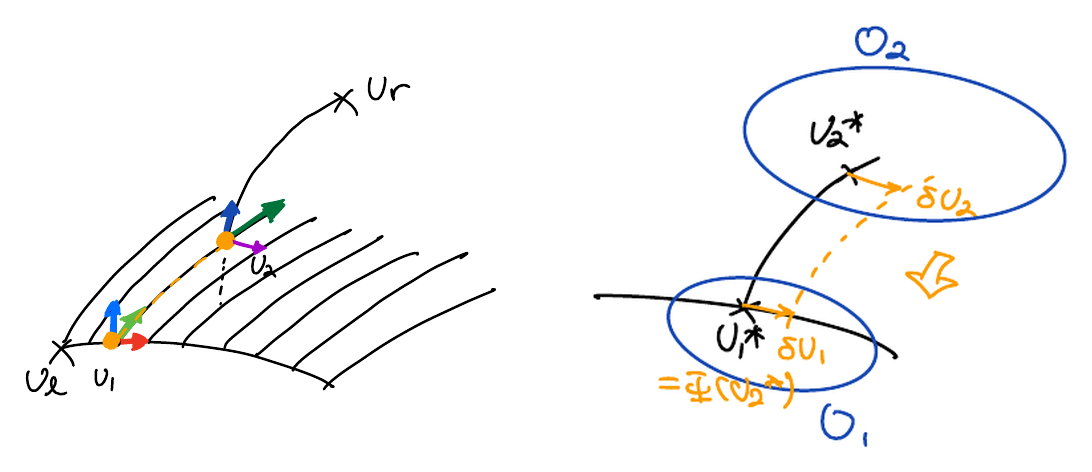}
  \caption{Two panels from our handwritten sketch for the $3\times 3$ sequential transversality condition, cropped from one page and set side by side. \emph{Left:} the 2-Hugoniot curves branching off the 1-Hugoniot locus sweep out a surface, the backward 3-Hugoniot from $u_r$ meets it at $u_2^*$, and the coloured arrows are the tangent directions compared there. \emph{Right:} the pushforward as a map between tangent spaces. \cref{fig:sequential-transversality} is the redraw. The sketch orients the pushforward from $u_2^*$ to $u_1^*$, the redraw the other way, matching $\Phi_{12}$ in~\eqref{eq:pushforward-hugoniot}.}
  \label{fig:handwritten-seq-trans}
\end{figure}

\cref{fig:handwritten-3x3} shows the stages interacting. The page has the $3\times 3$ objective map and its block Jacobian. Above the third Rankine--Hugoniot row there is a margin note saying the anchoring of that row does not matter, since the two orderings are equivalent. On this page that is true. The two anchorings have the same zero set by the antisymmetry $H(u_2^*,s_3^*;\,u_r)=-H(u_r,s_3^*;\,u_2^*)$ (Section~\ref{ssec:transversality-3x3}), and all three waves here are shocks, so either anchoring gives an invertible $B_3$ and the two differ only by a sign. It fails once the terminal wave is a rarefaction. The forward-anchored block is then singular by \cref{lem:Gamma-base-rank}, so \cref{lem:block-reduction} does not apply. We transcribed the forward-anchored form as written, and it stayed in the draft until a review pass caught it. The fix is to anchor the terminal wave at $u_r$ throughout, as in~\eqref{eq:objective-nxn}. \cref{rem:backward-anchoring} gives the reason. Later passes returned smaller things, notational inconsistencies and sign conventions in a worked example.

\begin{figure}[htbp]
  \centering
  \includegraphics[width=0.82\textwidth]{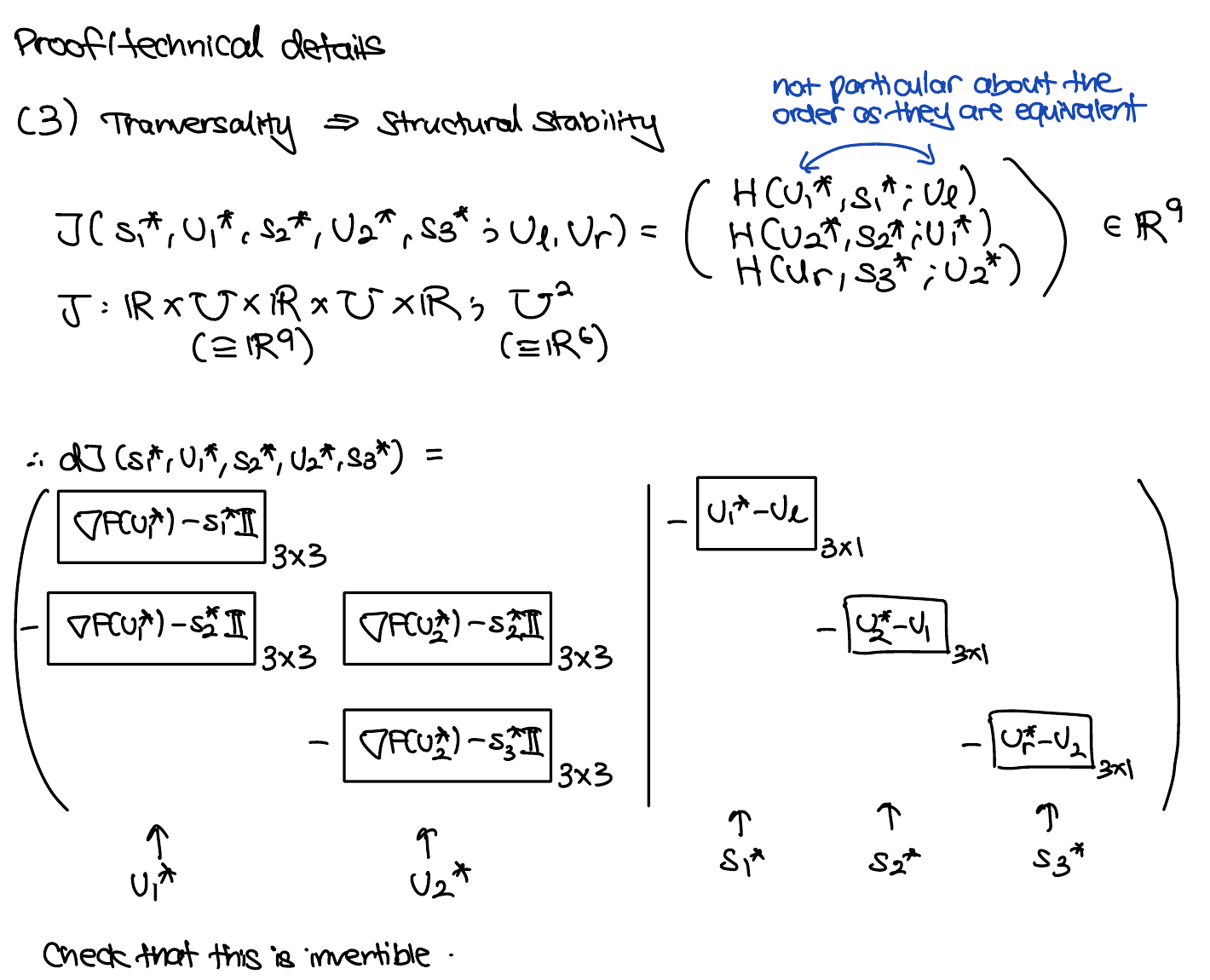}
  \caption{Page~3 of our handwritten notes for the $3\times 3$ case, as written. The objective map at the top is forward-anchored, its third row reading $H(u_r,s_3^*;\,u_2^*)$, and the note above it says the ordering does not matter. The published form~\eqref{eq:objective-3x3} anchors that row at the right state instead, reading $H(u_2^*,s_3^*;\,u_r)$, and the matching block and speed entries of~\eqref{eq:jacobian-3x3-block} change sign. \cref{rem:backward-anchoring} says why the two are not interchangeable when the terminal wave is a rarefaction. The instruction at the foot of the page is answered by \cref{lem:block-reduction}.}
  \label{fig:handwritten-3x3}
\end{figure}

AI assistance does not establish correctness. A formal proof in a system such as Lean would establish it, and doing so is possible in principle, but the PDE and differential topology libraries are still being built and their conventions are not settled, and hyperbolic conservation laws are not a current focus of that work. Most of the effort would go into that groundwork rather than into the arguments of this paper.

\clearpage
\section*{Acknowledgments}
AI assistance was used in preparing this paper, for transcribing the authors' handwritten notes into \LaTeX, for generating \LaTeX\ and Ti\emph{k}Z source to the authors' specification, for reading complete drafts for technical errors, notational inconsistencies, and poor writing, and for discussing arguments the authors had already drafted, all of which Appendix~\ref{app:machine-assisted} describes in more detail.
No theorem, definition, proposition, or lemma was stated by a model, no proof was accepted without the authors re-deriving it in full, and reports from a reading pass were acted on only after the point at issue had been re-derived.
The models used were Anthropic's Claude Opus~4.6 and OpenAI's GPT-5.4 in the earlier drafting, and Anthropic's Claude Opus~5, Fable~5, and Fable~5.1, together with OpenAI's GPT-5.6 in the later checking and refining of the proofs, between March and September 2026.
The authors assume responsibility for all content.

\bibliographystyle{siamplain}
\bibliography{references}

\end{document}